\documentclass[11pt,a4paper]{article}

\usepackage[T1]{fontenc}
\usepackage[utf8]{inputenc}
\IfFileExists{lmodern.sty}
  {\usepackage{lmodern}\usepackage{microtype}}
  {\usepackage[expansion=false]{microtype}}

\usepackage[a4paper,width=155mm,top=28mm,bottom=30mm]{geometry}
\usepackage{enumitem}
\setlist[enumerate]{itemsep=2pt,topsep=4pt}
\setlist[itemize]{itemsep=2pt,topsep=4pt}
\usepackage{booktabs}
\usepackage{tocloft}
\usepackage{amsmath}
\usepackage{amssymb}
\usepackage{amsthm}
\usepackage{mathtools}
\usepackage{bm}
\allowdisplaybreaks[1]

\usepackage{xcolor}
\definecolor{linkcol}{RGB}{0,60,130}
\definecolor{citecol}{RGB}{0,100,70}
\usepackage[colorlinks=true,linkcolor=linkcol,citecolor=citecol,
            urlcolor=linkcol,breaklinks=true]{hyperref}
\usepackage[nameinlink,capitalise]{cleveref}
\Crefname{equation}{Eq.}{Eqs.}

\usepackage{tikz}
\usetikzlibrary{arrows.meta,positioning,fit}

\usepackage[round,authoryear]{natbib}
\newcommand{\doi}[1]{\textsc{doi}:~\href{https://doi.org/#1}{\nolinkurl{#1}}}
\newcommand{\biburl}[1]{\href{#1}{\nolinkurl{#1}}}
\usepackage{aliascnt}
\newcommand{\sharedtheorem}[4]{
  \newaliascnt{#1}{theorem}
  \newtheorem{#1}[#1]{#2}
  \aliascntresetthe{#1}
  \crefname{#1}{#3}{#4}
  \Crefname{#1}{#3}{#4}
  \AddToHook{env/#1/begin}{\crefalias{theorem}{#1}}
}

\theoremstyle{plain}
\newtheorem{theorem}{Theorem}[section]
\sharedtheorem{lemma}{Lemma}{Lemma}{Lemmas}
\sharedtheorem{proposition}{Proposition}{Proposition}{Propositions}
\sharedtheorem{corollary}{Corollary}{Corollary}{Corollaries}

\theoremstyle{definition}
\sharedtheorem{definition}{Definition}{Definition}{Definitions}
\sharedtheorem{assumption}{Assumption}{Assumption}{Assumptions}
\sharedtheorem{example}{Example}{Example}{Examples}
\sharedtheorem{algo}{Algorithm}{Algorithm}{Algorithms}

\theoremstyle{remark}
\sharedtheorem{remark}{Remark}{Remark}{Remarks}
\sharedtheorem{notation}{Notation}{Notation}{Notations}

\newcommand{\R}{\mathbb{R}}
\newcommand{\C}{\mathbb{C}}
\newcommand{\N}{\mathbb{N}}
\newcommand{\Ints}{\mathbb{Z}}
\newcommand{\Rnn}{\R_{\ge 0}}
\newcommand{\Rpos}{\R_{>0}}

\newcommand{\Prb}{\mathbb{P}}
\newcommand{\Ex}{\mathbb{E}}
\DeclareMathOperator{\Var}{Var}
\DeclareMathOperator{\Cov}{Cov}
\DeclareMathOperator{\kurt}{kurt}
\DeclareMathOperator{\sech}{sech}
\newcommand{\Fisher}{\mathcal{I}}
\newcommand{\tsc}{\eta^{\star}}
\DeclareMathOperator{\Log}{Log}
\DeclareMathOperator{\Real}{Re}
\DeclareMathOperator{\Imag}{Im}
\newcommand{\Normal}{\mathcal{N}}
\newcommand{\Law}{\mathcal{L}}
\newcommand{\Dirac}{\delta}

\newcommand{\eqd}{\stackrel{d}{=}}

\newcommand{\cf}[1]{\varphi_{#1}}

\newcommand{\indep}{\mathrel{\perp\mkern-9mu\perp}}

\newcommand{\T}{^{\mkern-1.5mu\mathsf{T}}}
\newcommand{\pinv}{^{\dagger}}
\newcommand{\invT}{^{-\mkern-1.5mu\mathsf{T}}}
\DeclareMathOperator{\rank}{rank}
\DeclareMathOperator{\im}{im}
\DeclareMathOperator{\spn}{span}
\DeclareMathOperator{\diag}{diag}
\DeclareMathOperator{\tr}{tr}
\DeclareMathOperator{\supp}{supp}
\DeclareMathOperator{\Deg}{deg}
\newcommand{\Id}[1]{I_{#1}}
\newcommand{\Orth}[1]{O(#1)}
\newcommand{\bmat}[1]{\begin{bmatrix}#1\end{bmatrix}}

\DeclarePairedDelimiter{\abs}{\lvert}{\rvert}
\DeclarePairedDelimiter{\norm}{\lVert}{\rVert}

\newcommand{\bigabs}[1]{\abs[\big]{#1}}
\newcommand{\Bigabs}[1]{\abs[\Big]{#1}}
\newcommand{\given}{\,\vert\,}

\newcommand{\up}[1]{^{(#1)}}
\newcommand{\uT}[1]{^{(#1)\mkern-1.5mu\mathsf{T}}}
\newcommand{\upinv}[1]{^{(#1)\dagger}}

\title{Foundations of Independent Component Analysis}
\author{Patrick Forr\'e\\[+10pt]
    \small{AI4Science Lab}\\[0pt]
 \small{Korteweg-de Vries Institute for Mathematics}\\[0pt]
    \small{University of Amsterdam}\\[0pt] \small{\texttt{p.d.forre@uva.nl}}}
\date{}

\hypersetup{
  pdftitle={Foundations of Independent Component Analysis},
  pdfauthor={Patrick Forr\'e},
  pdfsubject={Independent component analysis: identifiability and estimation},
  pdfkeywords={independent component analysis, ICA, identifiability,
    characteristic functions, Darmois--Skitovich, LiNGAM}}

\begin{document}

\maketitle

\begin{abstract}
\noindent
We present the mathematical foundations of linear independent component
analysis (ICA) models based on standard literature in a self-contained note.
It is aimed at readers with a background in measure-theoretic probability
theory.  We first develop the theory of the characteristic functions of
probability measures on $\R^d$, including their analyticity and the way in
which they determine and characterise the distributions.
We then focus on several identifiability results of ICA models with
successively strengthened assumptions on the sources: from merely
non-constant, to non-Gaussian, to Gaussian-free independent sources.
Under the strictest assumptions, we show that the independent sources are
identifiable up to translation, permutation, scales and signs, and this even
in the presence of additive Gaussian noise.
Furthermore, we present the online equivariant gradient descent ICA algorithm
for recovering the independent sources from data, in the standard complete
noiseless non-Gaussian ICA setting.
\end{abstract}

\tableofcontents

\section{Introduction}
\label{sec:intro}

The linear independent component analysis (ICA) model postulates that an
observed $p$-dimensional random vector $X$ is an affine image of a
$k$-dimensional random vector $Z$ with \emph{mutually independent} components:
\begin{align}
  \label{eq:ica-intro}
  X &= A Z + \mu, & A &\in \R^{p \times k}, & \mu &\in \R^p .
\end{align}
The components $Z_1,\dots,Z_k$ are called the \emph{sources} and $A$ is called
the \emph{mixing matrix}; both are unobserved.
The \emph{identifiability question} asks how much of the pair $(A,\Law(Z))$ is
determined by the law of $X$ alone
\citep{comon1994,hyvarinen2000,eriksson2004}.

Some ambiguity is unavoidable and built into the model: we may rescale a
column of $A$ and inversely rescale the corresponding source, we may permute
the columns of $A$ together with the sources, and we may shift a source into
the offset $\mu$. Written out, for any permutation matrix $P$, any invertible
diagonal matrix $\Lambda$ and any $c \in \R^k$,
\begin{align}
  A Z + \mu &= (A P \Lambda)\bigl(\Lambda^{-1}P^{-1}(Z - c)\bigr) + (Ac + \mu),
\end{align}
and the right hand side is again a representation of the form
\cref{eq:ica-intro} with independent sources.
The content of the classical identifiability theory is that, under suitable
assumptions, this is \emph{the only} ambiguity -- with one important exception:
Gaussian sources.
That exception is genuine and not an artefact of the proofs. If $Z \sim
\Normal(0,\Id{k})$ then $QZ \sim \Normal(0,\Id{k})$ for every orthogonal $Q$,
so a purely Gaussian source vector can be rotated arbitrarily without changing
the law of $X$.

The results collected here go back to the characterisation theory of the
Gaussian law that grew out of the theorems of \citet{cramer1936},
\citet{marcinkiewicz1939}, \citet{darmois1953} and \citet{skitovich1954}, and
that is developed systematically in the monographs of
\citet{kagan1973characterization} and \citet{linnik1977}.
The ICA community rediscovered and popularised these results, notably through
\citet{comon1994} and \citet{eriksson2004}, and they underpin causal discovery
methods such as LiNGAM \citep{shimizu2006}.

\paragraph{Prerequisites.}
These notes are written for readers who have taken a measure-theoretic
probability course.  We assume laws and weak convergence, product measures and
Fubini's theorem, and basic complex analysis up to Morera's theorem and the
identity theorem.  Everything else -- in particular the entire theory of
characteristic functions that the arguments rest on -- is developed in
\cref{sec:cf}, with precise references for the standard results and full
proofs for everything specific to our setting.

\paragraph{Outline and contributions.}
\Cref{sec:prelim} fixes notation.
\Cref{sec:cf} develops characteristic functions of probability measures on
$\R^d$: their elementary properties, how they determine and characterise
distributions (uniqueness, inversion, the Cram\'er--Wold device, L\'evy's
continuity theorem), why local agreement of two characteristic functions is
\emph{not} enough, moments and cumulants via the distinguished logarithm,
analyticity in a strip and the resulting rigidity, and finally the
characterisation of the Gaussian law together with the sub- and
super-Gaussian (platykurtic and leptokurtic) distributions that ICA is built
on.  It closes with the precise sense in which kurtosis acts as a contrast
function, and with the observation that the identifiability theory needs
strictly less than that: non-normality, not non-zero kurtosis.
\Cref{sec:non-constant} treats sources that are merely assumed to be
non-constant and states the fundamental result of
\citet{kagan1973characterization} in the form we need; we take the opportunity
to reinstate a proportionality constant that is lost in the statement given
there (see \cref{fn:missing-constant}).
\Cref{sec:non-gaussian} specialises to non-Gaussian sources. The main technical
device is \cref{lem:noise-trade}, which shows how a Gaussian noise
vector can be traded against extra columns of the mixing matrix, and back;
this is what allows us to reduce the non-Gaussian case to the non-constant case.
\Cref{sec:gaussian-free} introduces the stronger requirement that the sources
be \emph{Gaussian-free}, i.e.\ that no independent Gaussian noise can be split
off them at all.  We show in \cref{thm:gauss-split} that \emph{every}
real-valued random variable decomposes, essentially uniquely, into a
Gaussian-free part and independent Gaussian noise, and we then prove the
strongest statement of these notes,
\cref{thm:gf-identifiability}: for full column rank mixing matrices the sources
are then identifiable up to permutation, scale and translation even in the
presence of additive Gaussian noise with an arbitrary, possibly non-diagonal
and degenerate covariance matrix.
\Cref{sec:algorithm} turns to estimation in the case that dominates
applications -- a square invertible mixing matrix and no noise.  We derive the
classical identifiability statement as \cref{cor:complete-ident}, set up the
maximum likelihood objective, show that the customary ``preconditioner'' in
the gradient step is not an approximate inverse Hessian but the exact gradient
for a natural right-invariant metric, and prove in \cref{thm:stability} exactly when a
separating solution is a locally asymptotically stable equilibrium of the mean
dynamics underlying the online algorithm -- which, for correctly specified
sources, is precisely when \cref{cor:complete-ident} says the model is
identifiable.  The same section then shows in
\cref{cor:lingam} how a causal ordering removes the last remaining
ambiguities, giving the linear non-Gaussian acyclic model LiNGAM, and closes
with a short survey of how the model has been generalised.
\Cref{app:cf-proofs} collects the proofs of the four classical results about
characteristic functions that \cref{sec:cf} quotes: the uniqueness and
inversion theorem, L\'evy's continuity theorem, Marcinkiewicz' theorem and
Cram\'er's decomposition theorem.
\Cref{app:kagan-proof} proves \cref{thm:kagan}, the identifiability theorem of
\citet{kagan1973characterization} on which \cref{sec:non-constant} onwards
rest.  The proof is a self-contained finite-difference argument: it needs
only the distinguished logarithm, Marcinkiewicz' theorem, and Fr\'echet's
functional equation, which we also prove.
With these appendices in place, every result in these notes is proved here,
with one exception: Bochner's \cref{thm:bochner}, which we quote and which is
used only once, in \cref{rem:local-not-enough}, where it could be replaced by
a direct computation.  Standard background is of course used throughout and is
listed at the start of \cref{app:cf-proofs}; \cref{sec:algorithm} additionally
uses standard matrix calculus and, in \cref{thm:stability}, the
unstable-manifold theorem for maps.

\paragraph{How to read this note.}
The logical skeleton is short.  Everything downstream of
\cref{sec:non-constant} rests on the single result \cref{thm:kagan}, proved
in \cref{app:kagan-proof}, and branches from there into the two lines of
development that the rest of the note pursues:
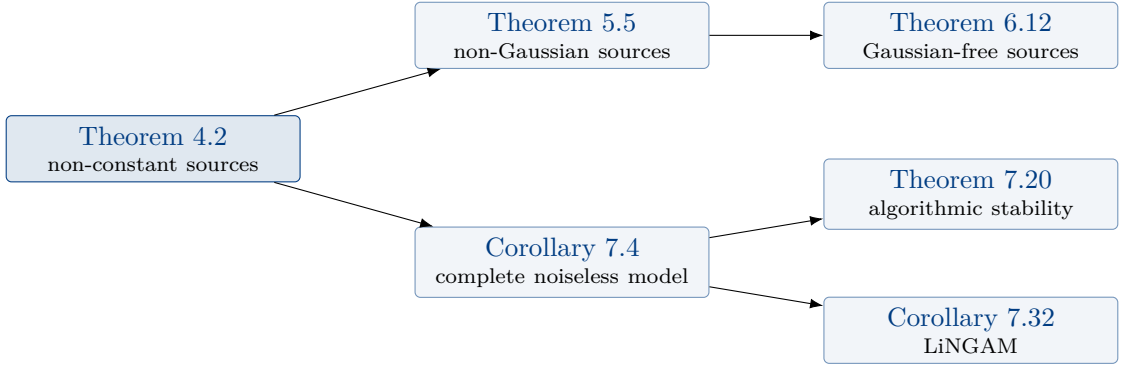
\begin{center}
\begin{tikzpicture}[
    >={Latex[length=5pt]},
    box/.style={draw=linkcol!55, rounded corners=2pt, fill=linkcol!5,
                inner sep=4pt, font=\small, align=center, text width=36mm},
    root/.style={box, draw=linkcol, fill=linkcol!12}]
  \node[root] (kagan) at (0,0)
    {\cref{thm:kagan}\\[-2pt]\scriptsize non-constant sources};
  \node[box]  (nong)  at (5.4, 1.5)
    {\cref{thm:non-gaussian}\\[-2pt]\scriptsize non-Gaussian sources};
  \node[box]  (comon) at (5.4,-1.5)
    {\cref{cor:complete-ident}\\[-2pt]\scriptsize complete noiseless model};
  \node[box]  (gf)    at (10.8, 1.5)
    {\cref{thm:gf-identifiability}\\[-2pt]\scriptsize Gaussian-free sources};
  \node[box]  (stab)  at (10.8,-0.6)
    {\cref{thm:stability}\\[-2pt]\scriptsize algorithmic stability};
  \node[box]  (lingam) at (10.8,-2.4)
    {\cref{cor:lingam}\\[-2pt]\scriptsize LiNGAM};
  \draw[->] (kagan) -- (nong);
  \draw[->] (kagan) -- (comon);
  \draw[->] (nong)  -- (gf);
  \draw[->] (comon) -- (stab);
  \draw[->] (comon) -- (lingam);
\end{tikzpicture}
\end{center}
The upper branch strengthens the hypotheses on the sources and asks how much
of the model this buys back; the lower branch fixes the hypotheses at their
most convenient and asks how to estimate.  A reader interested only in the
algorithm may go directly from \cref{sec:cf} to \cref{sec:algorithm}, taking
\cref{cor:complete-ident} on faith.

\section{Notation and Conventions}
\label{sec:prelim}

\begin{notation}[Random objects versus deterministic objects]
\label{not:conventions}
All random objects are defined on one underlying probability space
$(\Omega,\mathcal{F},\Prb)$ and are denoted by \emph{capital} Latin letters:
$X$ for the observed random vector, $Z$ for the source vector, $E$, $V$, $G$,
$H$, $N$, $U$ for the various (mostly Gaussian) noise and remainder terms.
Deterministic matrices are capital Latin or Greek letters ($A$, $B$, $P$,
$Q$, $R$, $W$, $\Theta$, $\Lambda$, $\Sigma$; in particular $W$ always denotes
the deterministic unmixing matrix of \cref{sec:algorithm}, never a random
vector), deterministic vectors and scalars are lower case
($a_j$, $b_l$, $\mu$, $\nu$, $\xi$, $\lambda_j$, $t$).
Components of a random vector carry a subscript, $Z = [Z_1,\dots,Z_k]\T$, and
$e_j$ denotes the $j$-th standard basis vector.
Superscripts in brackets, $A\up{1}$, $Z\up{2}$, index the two competing
representations of the same observed vector and are never exponents.
\end{notation}

\begin{notation}[Multi-indices and weak convergence]
\label{not:multiindex}
For a multi-index $\alpha = (\alpha_1,\dots,\alpha_d)\in\N_0^d$ we write
$\abs{\alpha} := \alpha_1+\dots+\alpha_d$ for its order, and, for
$x\in\R^d$ and a sufficiently differentiable $f\colon\R^d\to\C$,
\begin{align}
  x^\alpha &:= \prod_{j=1}^{d} x_j^{\alpha_j},
  & \partial^\alpha f
    &:= \frac{\partial^{\abs{\alpha}} f}
             {\partial x_1^{\alpha_1}\cdots\partial x_d^{\alpha_d}} ,
\end{align}
with the convention $x^0 := 1$ and $\partial^0 f := f$.
The symbol $\Rightarrow$ denotes weak convergence of laws: for random vectors
$X_n, X$ in $\R^d$ we write $X_n \Rightarrow X$ if
$\Ex[f(X_n)]\to\Ex[f(X)]$ for every bounded continuous
$f\colon\R^d\to\R$.
\end{notation}

\begin{notation}[Distributional relations]
\label{not:dist}
For a random vector $X$ in $\R^p$ we write $\Law(X)$ for its law and
\begin{align}
  \cf{X}(t) &:= \Ex\bigl[\exp(i\, t\T X)\bigr], & t &\in \R^p,
\end{align}
for its characteristic function.
We reserve
\begin{itemize}
  \item $X \sim \Normal(\mu,\Sigma)$ for ``$X$ \emph{is distributed according
        to} the law $\Normal(\mu,\Sigma)$'', and
  \item $X \eqd Y$ for ``$X$ and $Y$ \emph{have the same law}'',
        i.e.\ $\Law(X) = \Law(Y)$, equivalently $\cf{X} = \cf{Y}$.
\end{itemize}
Stochastic independence is written $X \indep Y$.
Mutual independence of a family $\{X_1,\dots,X_k\}$ always means independence
of the whole family, not merely pairwise independence.
Statements such as ``$Z_j$ is non-constant'' are always understood
$\Prb$-almost surely, i.e.\ $\Law(Z_j)$ is not a Dirac measure $\Dirac_c$.
\end{notation}

\begin{notation}[Degenerate Gaussian distributions]
\label{not:degenerate}
We call $E$ a (possibly degenerate) \emph{Gaussian} random
vector, $E \sim \Normal(\mu,\Sigma)$ with $\mu \in \R^p$ and a
positive semi-definite $\Sigma \in \R^{p\times p}$, if
\begin{align}
  \cf{E}(t) &= \exp\Bigl(i\,t\T\mu - \tfrac{1}{2} t\T \Sigma t\Bigr),
  & t &\in \R^p .
\end{align}
This includes the degenerate cases: $\Sigma$ is allowed to be singular, and
for $p=1$, $\sigma^2 = 0$, we have $\Normal(\mu,0) = \Dirac_\mu$, so that
constants count as Gaussian.
Whenever we want to exclude this case we say \emph{non-degenerate}.
\end{notation}

\section{Characteristic Functions of Probability Measures on \texorpdfstring{$\R^d$}{R\^{}d}}
\label{sec:cf}

Every argument in these notes is carried out at the level of characteristic
functions.  This section collects the facts we need, in the generality in
which we need them.  We assume familiarity with measure-theoretic probability
-- laws, weak convergence, product measures, Fubini's theorem -- and we cite
the standard results rather than reproving them; everything that is specific
to the independent component analysis (ICA) setting is proved in full.
Standard references are \citet[Chapter~15]{klenke2020} and
\citet[Chapters~5--6]{kallenberg2021} for the general theory,
\citet{lukacs1970} for characteristic functions specifically, and
\citet[Chapter~1]{kagan1973characterization} for the characterisation-theoretic
material.

\subsection{Definition and elementary properties}
\label{ssec:cf-basics}

\begin{definition}[Characteristic function]
\label{def:cf}
Let $\mu$ be a probability measure on $\bigl(\R^d,\mathcal{B}(\R^d)\bigr)$.
Its \emph{characteristic function} (c.f.) is the map
\begin{align}
  \cf{\mu} \colon \R^d \to \C,
  \qquad
  \cf{\mu}(t) := \int_{\R^d} \exp\bigl(i\, t\T x\bigr)\, \mu(dx).
\end{align}
For a random vector $X$ in $\R^d$ with law $\Law(X) = \mu$ we write
$\cf{X} := \cf{\mu}$, so that
$\cf{X}(t) = \Ex\bigl[\exp(i\,t\T X)\bigr]$.
The integral is well defined because $\abs{\exp(i\,t\T x)} = 1$, so the
integrand is bounded and measurable.
\end{definition}

\begin{proposition}[Elementary properties]
\label{prop:cf-elementary}
Let $X$ be a random vector in $\R^d$ with characteristic function $\cf{X}$.
Then:
\begin{enumerate}[label=(\roman*)]
  \item $\cf{X}(0) = 1$ and $\abs{\cf{X}(t)}\le 1$ for all $t \in \R^d$;
  \item $\cf{X}$ is uniformly continuous on $\R^d$;
  \item $\cf{X}(-t) = \overline{\cf{X}(t)}$, and $\cf{X}$ is real valued if
        and only if $X \eqd -X$;
  \item $\cf{X}$ is positive semi-definite: for all $n \in \N$, all
        $t_1,\dots,t_n \in \R^d$ and all $z_1,\dots,z_n \in \C$,
        \begin{align}
          \sum_{j,l=1}^{n} z_j\overline{z_l}\,\cf{X}(t_j-t_l) &\ge 0;
        \end{align}
  \item \emph{(affine images)} for $A \in \R^{m\times d}$ and $b \in \R^m$,
        \begin{align}
          \label{eq:cf-affine}
          \cf{AX+b}(t) &= \exp\bigl(i\,t\T b\bigr)\cdot\cf{X}\bigl(A\T t\bigr),
          & t &\in \R^m;
        \end{align}
  \item \emph{(independent sums)} if $X \indep Y$ are random vectors in
        $\R^d$, then
        \begin{align}
          \label{eq:cf-product}
          \cf{X+Y}(t) &= \cf{X}(t)\cdot\cf{Y}(t), & t&\in\R^d;
        \end{align}
  \item \emph{(one-dimensional projections)} for $u \in \R^d$ and
        $s \in \R$,
        \begin{align}
          \label{eq:cf-projection}
          \cf{u\T X}(s) &= \cf{X}(su).
        \end{align}
\end{enumerate}
\end{proposition}

\begin{proof}
(i) is immediate.
For (ii) note that
$\abs{\cf{X}(t+h)-\cf{X}(t)} \le \Ex\bigl[\abs{\exp(i\,h\T X)-1}\bigr]$,
whose right hand side does not depend on $t$ and tends to $0$ as $h\to 0$ by
dominated convergence.
(iii): the identity $\overline{\exp(i\,t\T X)} = \exp(-i\,t\T X)$ gives
$\cf{X}(-t)=\overline{\cf{X}(t)}$, and hence $\cf{X}$ is real as soon as
$X\eqd -X$, since then $\cf{X}(-t)=\cf{-X}(t)=\cf{X}(t)$.
Conversely, if $\cf{X}$ is real then
$\cf{-X}(t)=\cf{X}(-t)=\overline{\cf{X}(t)}=\cf{X}(t)$ for every $t$, so
$X\eqd-X$ by \cref{thm:uniqueness}\,(i) below, whose proof does not use this
proposition.
For (iv) we argue
$\sum_{j,l} z_j\overline{z_l}\exp\bigl(i(t_j-t_l)\T X\bigr)
 = \bigabs{\sum_j z_j \exp(i\,t_j\T X)}^2 \ge 0$
after taking expectations.
For (v),
$\Ex[\exp(i\,t\T(AX+b))] = \exp(i\,t\T b)\,\Ex[\exp(i\,(A\T t)\T X)]$.
(vi) is the factorisation of the expectation of a product of independent
bounded random variables, and (vii) is the special case $A = u\T$, $b=0$ of
(v).
\end{proof}

Properties (\ref{eq:cf-affine}), (\ref{eq:cf-product}) and
(\ref{eq:cf-projection}) are the three
identities that carry all the ICA arguments: a linear ICA model is by
definition an independent sum of affine images, so its characteristic function
factorises into the characteristic functions of the sources evaluated along
the columns of the mixing matrix.

\begin{theorem}[Bochner's theorem; see
{\citealp[Chapter~15]{klenke2020}}; {\citealp[Section~4.2]{lukacs1970}}]
\label{thm:bochner}
A function $\phi\colon\R^d\to\C$ is the characteristic function of a
probability measure on $\R^d$ if and only if $\phi$ is continuous,
$\phi(0)=1$, and $\phi$ is positive semi-definite in the sense of
\cref{prop:cf-elementary}\,(iv).
\end{theorem}

Beyond the standard background listed at the start of \cref{app:cf-proofs},
\cref{thm:bochner} is the only result from the theory of characteristic
functions that any proof in these notes uses without our having proved it: the
difficult direction is a genuinely deep theorem, and reproducing it here would
take us too far afield.  (Two further classical facts are quoted in remarks --
P\'olya's criterion in \cref{rem:local-not-enough} and the Lukacs
characterisation in \cref{rem:ridge} -- but no proof depends on either.)
It is what makes statements of the form ``$\phi$ \emph{is} a characteristic
function'' verifiable without exhibiting the underlying law.
We shall repeatedly need such statements -- for instance in
\cref{thm:gauss-split}, where a supremum over admissible Gaussian factors is
taken -- although in practice we usually certify them through the more
convenient \cref{thm:levy} rather than by verifying positive
semi-definiteness.  The one place where we do verify positive
semi-definiteness directly is \cref{rem:local-not-enough}.

\subsection{How characteristic functions determine distributions}
\label{ssec:cf-uniqueness}

\begin{theorem}[Uniqueness and inversion]
\label{thm:uniqueness}
Let $\mu,\nu$ be probability measures on $\R^d$.
\begin{enumerate}[label=(\roman*)]
  \item If $\cf{\mu} = \cf{\nu}$ on all of $\R^d$, then $\mu = \nu$.
        Equivalently, for random vectors $X,Y$ in $\R^d$,
        \begin{align}
          X \eqd Y &\iff \cf{X} = \cf{Y}.
        \end{align}
  \item If in addition $\cf{\mu}\in L^1(\R^d)$, then $\mu$ has a bounded
        continuous Lebesgue density given by the inversion formula
        \begin{align}
          f(x) &= (2\pi)^{-d}\int_{\R^d}
                  \exp\bigl(-i\,t\T x\bigr)\,\cf{\mu}(t)\,dt .
        \end{align}
\end{enumerate}
\end{theorem}

The proof is in \cref{ssec:proof-uniqueness}; see also
{\citealp[Chapter~15]{klenke2020}} and {\citealp[Chapter~5]{kallenberg2021}}.

\begin{theorem}[Cram\'er--Wold device; see
{\citealp[Chapter~15]{klenke2020}}]
\label{thm:cramer-wold}
Let $X,Y$ be random vectors in $\R^d$.  Then $X \eqd Y$ if and only if
$u\T X \eqd u\T Y$ for every $u \in \R^d$.
Likewise, $X_n \Rightarrow X$ weakly if and only if
$u\T X_n \Rightarrow u\T X$ for every $u \in \R^d$.
\end{theorem}

\begin{proof}
By \cref{eq:cf-projection} we have $\cf{u\T X}(1) = \cf{X}(u)$, so knowing the
laws of all one-dimensional projections is the same as knowing $\cf{X}$
pointwise, and \cref{thm:uniqueness}\,(i) applies.

For the second statement, suppose first $X_n\Rightarrow X$.  By
\cref{thm:levy}\,(i) in $\R^d$, $\cf{X_n}\to\cf{X}$ pointwise, so
$\cf{u\T X_n}(s) = \cf{X_n}(su)\to\cf{X}(su) = \cf{u\T X}(s)$ for every
$s$; the limit is continuous at $s=0$, so \cref{thm:levy}\,(ii) in $d=1$
gives $u\T X_n\Rightarrow u\T X$.
Conversely, suppose $u\T X_n\Rightarrow u\T X$ for every $u$.  By
\cref{thm:levy}\,(i) in $d=1$ and \cref{eq:cf-projection} at $s=1$ we get
$\cf{X_n}(u)\to\cf{X}(u)$ for every $u\in\R^d$, and $\cf{X}$ is continuous
at the origin, so \cref{thm:levy}\,(ii) yields a random vector $X'$ with
$\cf{X'}=\cf{X}$ and $X_n\Rightarrow X'$; by \cref{thm:uniqueness}\,(i),
$X'\eqd X$ and hence $X_n\Rightarrow X$.
\end{proof}

\begin{remark}[Two different theorems of Cram\'er]
\label{rem:two-cramers}
\Cref{thm:cramer-wold} is the \emph{Cram\'er--Wold device}, a statement about
reducing $\R^d$ to $\R$.  It should not be confused with \emph{Cram\'er's
decomposition theorem}, \cref{thm:cramer} below, which says that a Gaussian law
has only Gaussian factors.  Both are used in these notes, for entirely different
purposes.
\end{remark}

\begin{theorem}[L\'evy's continuity theorem]
\label{thm:levy}
Let $(X_n)_{n\in\N}$ be random vectors in $\R^d$.
\begin{enumerate}[label=(\roman*)]
  \item If $X_n \Rightarrow X$ weakly, then $\cf{X_n}\to\cf{X}$ pointwise.
  \item Conversely, if $\cf{X_n}(t) \to \phi(t)$ for every $t\in\R^d$ and
        $\phi$ is continuous at $t=0$, then $\phi$ is the characteristic
        function of a random vector $X$ and $X_n \Rightarrow X$ weakly.
\end{enumerate}
\end{theorem}

The proof is in \cref{ssec:proof-levy}; see also
{\citealp[Section~15.3]{klenke2020}}.

The continuity requirement in (ii) is not a technicality: for
$X_n \sim \Normal(0,n)$ we have $\cf{X_n}(t) = \exp(-nt^2/2) \to
\mathbf{1}_{\{t=0\}}$ pointwise, and the limit -- discontinuous at the origin
-- is not a characteristic function; the mass has escaped to infinity.

\begin{remark}[Local equality of characteristic functions is \emph{not}
enough]
\label{rem:local-not-enough}
\Cref{thm:uniqueness} requires $\cf{\mu}=\cf{\nu}$ on \emph{all} of $\R^d$.
Agreement on a neighbourhood of the origin does not suffice.
The classical counterexample goes back to \citet{polya1949}; in the form given
by \citet[Chapter~XV]{feller1971} and \citet[p.~85]{lukacs1970} it
reads as follows.
Let
\begin{align}
  \phi_1(t) &:= \bigl(1-\abs{t}\bigr)_+ ,
\end{align}
the ``tent function''.  It \emph{is} a characteristic function.  Indeed, with
$h := \mathbf{1}_{[-1/2,1/2]}$ one has, for every $t\in\R$,
\begin{align}
  \label{eq:tent-selfconv}
  (h*h)(t) &= \int_{\R} h(u)\,h(t-u)\,du = \bigl(1-\abs{t}\bigr)_+ = \phi_1(t),
\end{align}
and since $h$ is real-valued and even, $\phi_1(t_j-t_l) = \int_\R
h(t_j-u)h(t_l-u)\,du$, whence for all $n\in\N$, $t_1,\dots,t_n\in\R$ and
$z_1,\dots,z_n\in\C$
\begin{align}
  \sum_{j,l=1}^{n} z_j\overline{z_l}\,\phi_1(t_j-t_l)
    &= \int_{\R}\Bigabs{\sum_{j=1}^{n} z_j\,h(t_j-u)}^{2}du \;\ge\; 0 .
\end{align}
So $\phi_1$ is continuous, $\phi_1(0)=1$ and positive semi-definite, and
\cref{thm:bochner} certifies that $\phi_1 = \cf{\mu_1}$ for a probability
measure $\mu_1$ on $\R$.  (Alternatively one may invoke P\'olya's criterion,
\citet[Section~XV.2]{feller1971}: every continuous even
$\phi\colon\R\to\R$ that is convex on $[0,\infty)$ and satisfies $\phi(0)=1$
and $\phi(t)\to0$ as $t\to\infty$ is a characteristic function; the tent
function is the standard example, being the maximum of two affine functions on
$[0,\infty)$.)
\emph{Now} \cref{thm:uniqueness}\,(ii) may be applied, since
$\phi_1\in L^1(\R)$: the law $\mu_1$ is absolutely continuous with density
\begin{align}
  f(x) &= \frac{1}{2\pi}\int_{-1}^{1}\bigl(1-\abs{t}\bigr)e^{-itx}\,dt
        = \frac{1}{\pi}\int_0^1 (1-t)\cos(tx)\,dt
        = \frac{1-\cos x}{\pi x^{2}} .
\end{align}
Let $\phi_2$ be the $2$-periodic extension of $1-\abs{t}$ from $[-1,1]$ to
all of $\R$.  Expanding the resulting triangular wave into its Fourier series
gives
\begin{align}
  \phi_2(t) &= \frac{1}{2}
    + \sum_{n \text{ odd}} \frac{4}{\pi^2n^2}\cos(n\pi t),
\end{align}
which is the characteristic function of the \emph{discrete} law with mass
$\tfrac{1}{2}$ at $0$ and mass $2/(\pi^2n^2)$ at each of $\pm n\pi$ for odd
$n\ge 1$; the masses do sum to
$\tfrac12 + \tfrac{4}{\pi^2}\sum_{n\text{ odd}}n^{-2}
 = \tfrac12+\tfrac{4}{\pi^2}\cdot\tfrac{\pi^2}{8} = 1$.
Then $\phi_1 = \phi_2$ on $[-1,1]$, while the two laws are as different as
they could be: one is absolutely continuous, the other purely atomic.

This is why every local statement in these notes is stated with the explicit
qualifier ``in a neighbourhood of the origin'', and why such statements are
never silently upgraded to statements about laws.  The upgrade is legitimate
under an analyticity hypothesis; see \cref{cor:rigidity}.
\end{remark}

\subsection{Moments, the distinguished logarithm, and cumulants}
\label{ssec:cf-cumulants}

\begin{proposition}[Moments and derivatives]
\label{prop:moments}
Let $X$ be a random vector in $\R^d$ and $n\in\N$ with
$\Ex\bigl[\norm{X}^n\bigr]<\infty$.
Then $\cf{X}\in C^n(\R^d;\C)$ and, for every multi-index $\alpha$ with
$\abs{\alpha}\le n$ and \emph{every} $t\in\R^d$,
\begin{align}
  \label{eq:cf-derivative}
  \partial^\alpha \cf{X}(t)
    &= i^{\abs{\alpha}}\,\Ex\bigl[X^\alpha e^{i\,t\T X}\bigr];
  & \text{in particular} \quad
  \partial^\alpha \cf{X}(0) &= i^{\abs{\alpha}}\,\Ex\bigl[X^\alpha\bigr].
\end{align}
In the scalar case $d=1$ this gives the Taylor expansion
\begin{align}
  \label{eq:cf-taylor}
  \cf{X}(t) &= \sum_{m=0}^{n}\frac{\Ex[X^m]}{m!}(it)^m + o\bigl(\abs{t}^n\bigr),
  \qquad t\to 0 .
\end{align}
\end{proposition}

\begin{proof}
We prove \cref{eq:cf-derivative} by induction on $\abs{\alpha}$, the case
$\alpha=0$ being the definition.  Suppose it holds for $\alpha$ with
$\abs{\alpha}<n$ and let $j\in\{1,\dots,d\}$.  For $h\neq0$,
\begin{align}
  \frac{\partial^\alpha\cf{X}(t+he_j)-\partial^\alpha\cf{X}(t)}{h}
    &= i^{\abs{\alpha}}\,
       \Ex\left[X^\alpha e^{i\,t\T X}\,
       \frac{e^{ihX_j}-1}{h}\right],
\end{align}
and the integrand is dominated in modulus by
$\abs{X^\alpha}\abs{X_j}\le\norm{X}^{\abs{\alpha}+1}$, which is integrable
because $\abs{\alpha}+1\le n$; here we used
$\abs{e^{ihx}-1}\le\abs{hx}$.  Since
$(e^{ihX_j}-1)/h\to iX_j$ pointwise as $h\to0$, dominated convergence gives
$\partial_j\partial^\alpha\cf{X}(t) = i^{\abs{\alpha}+1}
 \Ex[X^\alpha X_j e^{i\,t\T X}]$, which is the claim for the multi-index
$\alpha+e_j$.
The same domination shows that each such derivative is continuous in $t$, so
$\cf{X}\in C^n$.
\Cref{eq:cf-taylor} is Taylor's theorem with Peano remainder applied to the
$C^n$ function $\cf{X}$, together with \cref{eq:cf-derivative} at $t=0$.
\end{proof}

A partial converse holds, and it needs no moment hypothesis at all: for a
real-valued random variable $Y$ and $m\in\N$, if $\cf{Y}$ is $2m$ times
differentiable at the single point $0$, then $\Ex[Y^{2m}]<\infty$; the odd
moments up to order $2m$ are then finite by Lyapunov's inequality.  This is
\cref{lem:moments-from-derivatives}, proved in \cref{app:cf-proofs}.

All our arguments manipulate \emph{logarithms} of characteristic functions.
Since a characteristic function is complex valued and may vanish, this
requires a small amount of care, which the following lemma settles once and
for all.

\begin{lemma}[Distinguished logarithm]
\label{lem:distinguished-log}
Let $X$ be a random vector in $\R^d$.
\emph{(Uniqueness.)}  If $U\ni0$ is a connected open subset of $\R^d$ and
$\psi,\tilde\psi\colon U\to\C$ are continuous with
$\psi(0)=\tilde\psi(0)=0$ and $\exp\psi = \exp\tilde\psi = \cf{X}$ on
$U$, then $\psi=\tilde\psi$; so the germ at the origin is well defined and
we may compare distinguished logarithms on any such $U$.
\emph{(Existence.)}  There exist $\delta>0$ and a continuous function
$\psi_X\colon B_\delta(0)\to\C$ with $\psi_X(0)=0$ and
\begin{align}
  \cf{X}(t) &= \exp\bigl(\psi_X(t)\bigr), & t &\in B_\delta(0),
\end{align}
where $B_\delta(0) := \{t\in\R^d : \norm{t}<\delta\}$.
We call $\psi_X$ the \emph{cumulant generating function} of $X$ (near the
origin).
If $\Ex\bigl[\norm{X}^n\bigr]<\infty$, then $\psi_X \in C^n$ on
$B_\delta(0)$.
\end{lemma}

\begin{proof}
By \cref{prop:cf-elementary}\,(i)--(ii) the function $\cf{X}$ is continuous
with $\cf{X}(0)=1$, so there is a $\delta>0$ with
$\abs{\cf{X}(t)-1}<1$ for all $t \in B_\delta(0)$.
Hence $\cf{X}$ maps $B_\delta(0)$ into the open disc of radius $1$ around $1$,
which is contained in the slit plane $\C\setminus(-\infty,0]$ on which the
principal branch $\Log$ of the complex logarithm is defined and continuous.
Put $\psi_X := \Log\circ\,\cf{X}$; then $\psi_X$ is continuous,
$\psi_X(0)=\Log 1 = 0$ and $\exp\circ\,\psi_X = \cf{X}$.

For uniqueness, let $U\ni0$ be any connected open subset of $\R^d$ and let
$\psi,\tilde\psi\colon U\to\C$ be continuous with
$\psi(0)=\tilde\psi(0)=0$ and $\exp\psi = \exp\tilde\psi = \cf{X}$ on $U$.
Then $\exp(\tilde\psi-\psi)\equiv1$ on $U$, so $(\tilde\psi-\psi)/(2\pi i)$
is a continuous integer-valued function on the connected set $U$ vanishing at
the origin, hence identically $0$.  This covers in particular the case
$U=B_\delta(0)$, $\psi=\psi_X$.
The regularity statement follows from \cref{prop:moments} because $\Log$ is
holomorphic on the slit plane.
\end{proof}

\begin{definition}[Cumulants]
\label{def:cumulants}
Let $X$ be a real-valued random variable with $\Ex\bigl[\abs{X}^n\bigr] <
\infty$ and let $\psi_X$ be as in \cref{lem:distinguished-log}.
The \emph{cumulants} $\kappa_1(X),\dots,\kappa_n(X)$ of $X$ are the
coefficients in the expansion
\begin{align}
  \label{eq:cumulant-expansion}
  \psi_X(t) &= \sum_{m=1}^{n}\frac{\kappa_m(X)}{m!}(it)^m + o(t^n),
  \qquad t\to 0,
\end{align}
equivalently $\kappa_m(X) = i^{-m}\,\psi_X^{(m)}(0)$.
The first four cumulants are
\begin{align}
  \kappa_1(X) &= \Ex[X], &
  \kappa_2(X) &= \Var(X), \\
  \kappa_3(X) &= \Ex\bigl[(X-\Ex X)^3\bigr], &
  \kappa_4(X) &= \Ex\bigl[(X-\Ex X)^4\bigr] - 3\Var(X)^2 .
\end{align}
\end{definition}

\begin{proposition}[Cumulant calculus]
\label{prop:cumulant-calculus}
Let $X,Y$ be real-valued random variables with finite $n$-th absolute moments
and let $a,b\in\R$.  Then, for $m \le n$,
\begin{enumerate}[label=(\roman*)]
  \item \emph{(additivity)} if $X \indep Y$ then
        $\kappa_m(X+Y) = \kappa_m(X)+\kappa_m(Y)$;
  \item \emph{(homogeneity)} $\kappa_m(aX) = a^m\kappa_m(X)$;
  \item \emph{(translation invariance)} $\kappa_m(X+b) = \kappa_m(X)$ for
        $m\ge 2$, and $\kappa_1(X+b) = \kappa_1(X)+b$.
\end{enumerate}
\end{proposition}

\begin{proof}
By \cref{eq:cf-product} and \cref{lem:distinguished-log} we have
$\psi_{X+Y} = \psi_X+\psi_Y$ on a neighbourhood of the origin, since both
sides are continuous, vanish at $0$ and exponentiate to $\cf{X+Y}$;
comparing Taylor coefficients gives (i).
For (ii), $\cf{aX}(t) = \cf{X}(at)$ by \cref{eq:cf-affine}, so
$\psi_{aX}(t) = \psi_X(at)$ and the $m$-th coefficient picks up $a^m$.
For (iii), $\psi_{X+b}(t) = \psi_X(t) + ibt$, which changes only the
coefficient of $(it)^1$.
\end{proof}

Additivity is the reason cumulants, rather than moments, are the natural
bookkeeping device for ICA: a mixture is an independent sum, so its cumulants
are sums of the source cumulants.

\subsection{Analytic characteristic functions}
\label{ssec:cf-analytic}

\begin{theorem}[Analyticity in a strip]
\label{thm:analytic-strip}
Let $X$ be a random vector in $\R^d$ and suppose that there is an $a>0$ with
\begin{align}
  \label{eq:exp-moment}
  \Ex\bigl[\exp\bigl(a\norm{X}\bigr)\bigr] &< \infty .
\end{align}
Then the map
\begin{align}
  z &\longmapsto \Ex\bigl[\exp\bigl(i\,z\T X\bigr)\bigr]
\end{align}
is well defined and holomorphic on the tube
$S_a := \{z\in\C^d : \norm{\Imag z} < a\}$, and it restricts to $\cf{X}$ on
$\R^d$.
If \cref{eq:exp-moment} holds for every $a>0$, then $\cf{X}$ extends to an
entire function on $\C^d$.
\end{theorem}

\begin{proof}
We give the argument for $d=1$, which is the only case used below: every
appeal to this theorem in these notes is to a real-valued random variable, and
the one place where a multivariate statement is wanted, \cref{cor:rigidity},
is reduced to one variable via \cref{thm:cramer-wold}.  For general $d$
the same argument applied in each coordinate separately, with
$\abs{\cdot}$ replaced by $\norm{\cdot}$, yields \emph{separate} holomorphy
together with local boundedness on $S_a$; joint holomorphy then follows from
Osgood's lemma (or, dispensing with boundedness, from Hartogs' theorem).
For $z = s+iu$ with $\abs{u} < a$ we have
$\abs{\exp(izX)} = \exp(-uX) \le \exp(a\abs{X})$,
which is integrable by \cref{eq:exp-moment}; so
$F(z) := \Ex[\exp(izX)]$ is well defined on the strip $S_a$, and it is
continuous there by dominated convergence.
For any closed triangle $\Delta \subseteq S_a$, Fubini's theorem applies
because $(z,\omega)\mapsto \exp(izX(\omega))$ is jointly measurable and
dominated by the integrable $\exp(a\abs{X})$ uniformly on the compact set
$\partial\Delta$; hence
\begin{align}
  \oint_{\partial\Delta} F(z)\,dz
    &= \Ex\left[\oint_{\partial\Delta}\exp(izX)\,dz\right] = 0,
\end{align}
the inner integral vanishing because $z\mapsto\exp(izX)$ is entire for each
fixed value of $X$.
By Morera's theorem $F$ is holomorphic on $S_a$, and $F|_{\R} = \cf{X}$ by
construction.
\end{proof}

\begin{corollary}[Rigidity of analytic characteristic functions]
\label{cor:rigidity}
Let $X,Y$ be random vectors in $\R^d$ both satisfying \cref{eq:exp-moment} for
some $a>0$.  If $\cf{X} = \cf{Y}$ on some neighbourhood of the origin in
$\R^d$, then $X \eqd Y$.
\end{corollary}

\begin{proof}
Replacing $a$ by the smaller of the two constants we may assume that $X$ and
$Y$ satisfy \cref{eq:exp-moment} with the same $a>0$, and let $\delta>0$ be
such that $\cf{X}=\cf{Y}$ on $B_\delta(0)$.
Fix $u\in\R^d\setminus\{0\}$.  By \cref{eq:cf-projection},
\begin{align}
  \cf{u\T X}(s) = \cf{X}(su) = \cf{Y}(su) = \cf{u\T Y}(s),
  \qquad \abs{s} < \delta/\norm{u}.
\end{align}
Moreover $\abs{u\T X}\le\norm{u}\norm{X}$, so the real-valued random variables
$u\T X$ and $u\T Y$ satisfy \cref{eq:exp-moment} with the constant
$a/\norm{u}$, and by \cref{thm:analytic-strip} their characteristic functions
extend holomorphically to a strip in $\C$.
Two functions holomorphic on a connected open subset of $\C$ which agree on a
real interval agree on all of it, by the identity theorem in one complex
variable; hence $\cf{u\T X} = \cf{u\T Y}$ on all of $\R$ and therefore
$u\T X \eqd u\T Y$ by \cref{thm:uniqueness}\,(i).
Since $u$ was arbitrary, \cref{thm:cramer-wold} gives $X \eqd Y$.
\end{proof}

We deliberately reduced to one complex variable here.  In several variables
the \emph{naive} form of the identity theorem fails: agreement of two
holomorphic functions on a set with accumulation points is not sufficient, as
$f(z_1,z_2)=z_1$, vanishing on the uncountable set $\{0\}\times\C$, shows.
A $d$-dimensional argument is nevertheless available -- at a real point of a
real neighbourhood on which $h$ vanishes, every complex partial derivative of
$h$ agrees with the corresponding real one and hence vanishes, so $h\equiv0$
on a polydisc and then on all of the connected $S_a$ -- but the reduction to
one variable is shorter and is all we need.

\Cref{cor:rigidity} is exactly what \cref{rem:local-not-enough} rules out in
general: the tent function of that remark is the characteristic function of a
law whose density is of order $x^{-2}$, which has no exponential moments at
all.

\begin{proposition}[Ridge property]
\label{prop:ridge}
Let $d=1$ and let $X$ satisfy \cref{eq:exp-moment} for some $a>0$, and denote
again by $\cf{X}$ the holomorphic extension of \cref{thm:analytic-strip}.
Then for all $s\in\R$ and $\abs{u}<a$,
\begin{align}
  \label{eq:ridge}
  \abs{\cf{X}(s+iu)} &\le \cf{X}(iu) = \Ex\bigl[\exp(-uX)\bigr] \in \Rpos,
\end{align}
that is, on each horizontal line the modulus is maximal on the imaginary axis.
\end{proposition}

\begin{proof}
On the strip we have $\cf{X}(z) = \Ex[\exp(izX)]$, so
\begin{align}
  \abs{\cf{X}(s+iu)} = \bigabs{\Ex\bigl[e^{isX}e^{-uX}\bigr]}
    \le \Ex\bigl[\bigabs{e^{isX}}e^{-uX}\bigr]
     = \Ex\bigl[e^{-uX}\bigr] = \cf{X}(iu),
\end{align}
which is real and strictly positive because $e^{-uX}>0$.
\end{proof}

\begin{remark}[Ridge property and the moment generating function]
\label{rem:ridge}
Evaluating the holomorphic extension of \cref{thm:analytic-strip} at
$z=-iu$ gives $\cf{X}(-iu) = \Ex[\exp(uX)] = M_X(u)$, the moment generating
function of $X$; \cref{eq:ridge} then says that $\abs{\cf{X}}$ is maximal on
each horizontal line at its purely imaginary point.
If $\cf{X}$ is analytic in \emph{some} neighbourhood of the origin --
equivalently, if $0$ lies in the interior of $\{u : M_X(u)<\infty\}$ -- then
its maximal strip of analyticity is exactly the one cut out by that interior
({\citealp[Theorem~7.1.1]{lukacs1970}}); the moment generating function may
well be finite at a boundary abscissa while $\cf{X}$ is analytic only on the
open strip.  Without the hypothesis at the origin the statement is false: a
law with an exponential right tail and a $\abs{x}^{-3}$ left tail has
$\{M_X<\infty\} = [0,1)$ with non-empty interior, yet $\Ex[X^2]=\infty$, so
$\cf{X}$ is not even twice differentiable at the origin, let alone analytic
there.
\end{remark}

The following theorem is the single most important analytic input for these
notes.  It says that the exponential of a polynomial is a characteristic
function only in the Gaussian case.

\begin{theorem}[Marcinkiewicz' theorem, \citealp{marcinkiewicz1939}]
\label{thm:marcinkiewicz}
Let $Y$ be a real-valued random variable whose characteristic function is of
the form
\begin{align}
  \cf{Y}(t) &= \exp\bigl(g(t)\bigr)
\end{align}
in a neighbourhood of the origin, with a (complex) polynomial $g$.
Then $\Deg(g) \le 2$; more precisely, after subtracting from $g$ the constant
$g(0) \in 2\pi i\,\Ints$, which changes $\exp(g)$ not at all and the degree
only in the trivial case of a constant $g$,
\begin{align}
  g(t) &= -\tfrac{1}{2}\sigma^2 t^2 + i\mu t,
  & \sigma^2 &\in \Rnn, & \mu &\in \R,
\end{align}
and consequently $Y$ has a (possibly degenerate) Gaussian distribution,
$Y \sim \Normal(\mu,\sigma^2)$, where the degenerate case $\sigma^2 = 0$
means $Y \sim \Normal(\mu,0) = \Dirac_\mu$.
\end{theorem}

The proof is in \cref{ssec:proof-marcinkiewicz}; see also
\citet[Lemma~1.4.2]{kagan1973characterization} and
\citet[Section~7.3]{lukacs1970}.

In the language of \cref{def:cumulants}: there is no probability distribution
whose cumulant generating function \emph{equals} a polynomial of degree $3$ or
higher near the origin.
The superficially similar statement that the vanishing of all cumulants of
order $>n$ already forces the vanishing of all cumulants of order $>2$ is also
true, but it is not a restatement of \cref{thm:marcinkiewicz}: vanishing
Taylor coefficients do not by themselves make $\psi_Y$ a polynomial, and an
additional analytic argument is required, namely a growth bound on the
moments strong enough to make $\cf{Y}$ entire.
For $n=2$ -- the only case we shall need -- such a bound is immediate,
because then the moments \emph{are} Gaussian moments; we give the argument
inside the proof of \cref{thm:gaussian-char}, in the implication
(v)$\Rightarrow$(ii).
For general $n$ one argues instead as follows.  The number of set partitions
of $\{1,\dots,m\}$ all of whose blocks have size at most $n$ is bounded by
$C(n)^m(m!)^{1-1/n}$, so the moment--cumulant formula gives
$\abs{\Ex[Y^m]}\le C'^m(m!)^{1-1/n}$ with a constant $C'\ge1$ depending on
$n$ and on the finitely many non-zero cumulants.  Lyapunov's inequality turns
this into a bound of the same shape for the absolute moments, with a larger
constant: $\Ex\abs{Y}^m\le\bigl(\Ex[Y^{m+1}]\bigr)^{m/(m+1)}
 \le (C')^{m+1}\bigl((m+1)!\bigr)^{1-1/n}$ for odd $m$, and since
$\bigl((m+1)!\bigr)^{1-1/n}/m! = (m+1)^{1-1/n}(m!)^{-1/n}$ one gets
$\sum_{m\ge0}a^m\,\Ex\abs{Y}^m/m!<\infty$ for every $a>0$.  By Tonelli's
theorem that sum equals $\Ex[\exp(a\abs{Y})]$, so \cref{eq:exp-moment} holds
for every $a>0$ and \cref{thm:analytic-strip} makes $\cf{Y}$ entire; then
$\psi_Y$ is holomorphic near $0$ and equal to its Taylor polynomial there, and
\cref{thm:marcinkiewicz} applies.  We shall not use this generalisation.

\begin{remark}[How \cref{thm:marcinkiewicz} will be used]
\label{rem:marcinkiewicz-use}
We shall use \cref{thm:marcinkiewicz} in the following two guises.
\begin{enumerate}[label=(\roman*)]
  \item If two real-valued random variables $Y_1$, $Y_2$ satisfy
        \begin{align}
          \label{eq:marcinkiewicz-use-scaled}
          \cf{Y_2}(\lambda t) &= \cf{Y_1}(t)\exp\bigl(g(t)\bigr)
        \end{align}
        near the origin with a polynomial $g$ and a scalar
        $\lambda\in\R\setminus\{0\}$, then $Y_2$ is Gaussian if and only if
        $Y_1$ is Gaussian.
        Indeed, by \cref{eq:cf-affine} the left hand side is
        $\cf{\lambda Y_2}(t)$, so it suffices to treat $\lambda=1$ and then to
        note that $\lambda Y_2$ is Gaussian if and only if $Y_2$ is, again by
        \cref{eq:cf-affine} and $\lambda\neq0$.
        For $\lambda=1$: if $Y_1$ is Gaussian then $\cf{Y_1}$ is itself the
        exponential of a polynomial of degree $\le 2$, hence so is $\cf{Y_2}$
        near the origin and \cref{thm:marcinkiewicz} applies; the converse
        follows by exchanging the roles and replacing $g$ by $-g$.
  \item The sum of an independent non-Gaussian random variable and a Gaussian
        random variable is again non-Gaussian: if $Y = Y_1 + Y_2$ with
        $Y_1 \indep Y_2$ and $Y_2$ Gaussian, then
        $\cf{Y_1}(t) = \cf{Y}(t)\cf{Y_2}(t)^{-1}$, and $\cf{Y_2}^{-1}$ is the
        exponential of a polynomial of degree $\le 2$; so if $Y$ were Gaussian,
        $Y_1$ would be Gaussian as well.
\end{enumerate}
Note that $\cf{Y_2}$ has no zeros for Gaussian $Y_2$, so the above quotients are
well defined on all of $\R$.
\end{remark}

\subsection{Gaussian, sub-Gaussian and super-Gaussian distributions}
\label{ssec:sub-super}

\begin{theorem}[Characterisation of the Gaussian law]
\label{thm:gaussian-char}
Let $X$ be a random vector in $\R^d$.  The following are equivalent.
\begin{enumerate}[label=(\roman*)]
  \item $X \sim \Normal(\mu,\Sigma)$ is (possibly degenerate) Gaussian in the
        sense of \cref{not:degenerate};
  \item $\cf{X}(t) = \exp\bigl(i\,t\T\mu - \tfrac12 t\T\Sigma t\bigr)$ for all
        $t\in\R^d$, for some $\mu\in\R^d$ and some positive semi-definite
        $\Sigma$;
  \item every one-dimensional projection $u\T X$, $u\in\R^d$, is a (possibly
        degenerate) Gaussian real-valued random variable.
\end{enumerate}
For $d=1$ these are further equivalent to each of:
\begin{enumerate}[label=(\roman*),start=4]
  \item $\psi_X$ coincides with a polynomial on a neighbourhood of the origin;
  \item $X$ has moments of all orders and $\kappa_m(X)=0$ for all $m\ge 3$.
\end{enumerate}
\end{theorem}

\begin{proof}
(i)$\Leftrightarrow$(ii) is \cref{not:degenerate} combined with
\cref{thm:uniqueness}\,(i).

\noindent
(ii)$\Rightarrow$(iii): \cref{eq:cf-projection} gives
\begin{align}
  \cf{u\T X}(s)
    &= \exp\Bigl(i s\, u\T\mu - \tfrac12 s^2\, u\T\Sigma u\Bigr),
\end{align}
which is the characteristic function of
$\Normal(u\T\mu,\,u\T\Sigma u)$.
(iii)$\Rightarrow$(ii): put $\mu_u := \Ex[u\T X]$ and
$\sigma^2_u := \Var(u\T X)$, which are finite by (iii); then
$\cf{X}(u) = \cf{u\T X}(1) = \exp(i\mu_u-\tfrac12\sigma_u^2)$, and
$u\mapsto\mu_u$ is linear while $u\mapsto\sigma_u^2$ is a positive
semi-definite quadratic form, so $\mu_u = u\T\mu$ and
$\sigma^2_u = u\T\Sigma u$ for $\mu := \Ex[X]$ and $\Sigma := \Cov(X)$.
(ii)$\Rightarrow$(iv) is clear, and (iv)$\Rightarrow$(ii) is Marcinkiewicz'
\cref{thm:marcinkiewicz}.
(ii)$\Rightarrow$(v) is immediate from \cref{eq:cumulant-expansion}.

For (v)$\Rightarrow$(ii) some care is needed, because
\cref{eq:cumulant-expansion} is only an asymptotic expansion: knowing all
Taylor coefficients of $\psi_X$ does \emph{not} by itself make $\psi_X$ a
polynomial, so \cref{thm:marcinkiewicz} is not directly applicable.
Instead, put $\mu:=\kappa_1(X)$ and $\sigma^2:=\kappa_2(X)$.
The moment--cumulant relations express $\Ex[X^m]$ as a universal polynomial in
$\kappa_1(X),\dots,\kappa_m(X)$; with $\kappa_m(X)=0$ for all $m\ge3$ these
are exactly the moments of $\Normal(\mu,\sigma^2)$.  Since
$\exp(a\abs{x})\le\exp(ax)+\exp(-ax)$, monotone convergence gives, for every
$a>0$,
\begin{align}
  \Ex\bigl[\exp(a\abs{X})\bigr]
    &\le 2\sum_{k\ge0}\frac{a^{2k}}{(2k)!}\Ex\bigl[X^{2k}\bigr]
     = 2\sum_{k\ge0}\frac{a^{2k}}{(2k)!}\Ex\bigl[N^{2k}\bigr] < \infty,
\end{align}
where $N\sim\Normal(\mu,\sigma^2)$.  The last series converges because it is
dominated by $2\,\Ex[\exp(a\abs{N})]$, which is finite.
So $X$ satisfies \cref{eq:exp-moment} for every $a>0$, and by
\cref{thm:analytic-strip} its characteristic function is entire, hence equal
to its everywhere convergent Taylor series about the origin.  By
\cref{prop:moments} that series is determined by the moments of $X$, which are
those of $N$, so $\cf{X} = \cf{N}$, which is (ii).
\end{proof}

\begin{theorem}[Cram\'er's decomposition theorem, \citealp{cramer1936}]
\label{thm:cramer}
If $Y_1 \indep Y_2$ are real-valued random variables such that $Y_1 + Y_2$ is
Gaussian, then both $Y_1$ and $Y_2$ are Gaussian (possibly degenerate).
\end{theorem}

The proof is in \cref{ssec:proof-cramer}; see also
{\citealp[Theorem~1.1.1]{kagan1973characterization}}.

\Cref{thm:cramer,thm:marcinkiewicz} are dual to each other and together
explain the special role of the Gaussian law: it cannot be built out of
non-Gaussian independent pieces, and it cannot be approached by ``almost
polynomial'' cumulant generating functions.
The next definition quantifies the deviation from normality that ICA
exploits.

\begin{definition}[Excess kurtosis; sub- and super-Gaussian]
\label{def:kurtosis}
Let $X$ be a real-valued random variable with $\Ex[X^4]<\infty$ and
$\sigma^2 := \Var(X) \in (0,\infty)$.  Its \emph{excess kurtosis} is
\begin{align}
  \kurt(X) &:= \frac{\kappa_4(X)}{\sigma^4}
             = \frac{\Ex\bigl[(X-\Ex X)^4\bigr]}{\sigma^4} - 3 .
\end{align}
We call $X$
\begin{itemize}
  \item \emph{super-Gaussian} (or \emph{leptokurtic}) if $\kurt(X)>0$,
  \item \emph{sub-Gaussian} (or \emph{platykurtic}) if $\kurt(X)<0$,
  \item \emph{mesokurtic} if $\kurt(X)=0$.
\end{itemize}
\end{definition}

This is the convention of the ICA literature
\citep{comon1994,hyvarinen2001}; see \cref{rem:two-subgaussians} for the
clash with the concentration-theoretic use of the same words, and
\cref{rem:kurtosis-tails} for what the sign of $\kurt$ does and does not say
about tails.
For the laws customarily used as ICA source models the picture is the familiar
one: the super-Gaussian ones are peaked at the mode and heavy tailed -- the
Laplace and Student laws used to model the ``sparse'' sources of natural
images and audio -- while the sub-Gaussian ones are flat-topped, the uniform
law being the prototype.  \Cref{rem:kurtosis-tails} shows that this
correspondence is a property of those particular families rather than a
general implication.

\begin{proposition}[Kurtosis through the characteristic function]
\label{prop:kurtosis-cf}
Let $X$ be a real-valued random variable with $\Ex[X^4]<\infty$.  Then
\begin{align}
  \kappa_4(X) &= \psi_X^{(4)}(0)
   = \frac{d^4}{dt^4}\,\Log\cf{X}(t)\Big|_{t=0},
\end{align}
and for independent real-valued random variables $X_1,\dots,X_k$ with finite
fourth moments and weights $w\in\R^k$,
\begin{align}
  \label{eq:kappa4-additive}
  \kappa_4\Bigl(\sum_{j=1}^{k}w_jX_j\Bigr)
    &= \sum_{j=1}^{k} w_j^4\,\kappa_4(X_j).
\end{align}
\end{proposition}

\begin{proof}
The first identity is \cref{def:cumulants} with $m=4$, using $i^{-4}=1$ and
\cref{lem:distinguished-log}.
The second follows from \cref{prop:cumulant-calculus}\,(i)--(ii).
\end{proof}

\begin{remark}[$\kurt = 0$ does not mean Gaussian]
\label{rem:kurt-zero}
By \cref{thm:gaussian-char}\,(v) a non-degenerate Gaussian random variable
is mesokurtic, but
the converse fails badly.
Consider the symmetric three-point law
\begin{align}
  \label{eq:three-point}
  \Prb[X = a] = \Prb[X=-a] = p, \qquad \Prb[X=0] = 1-2p,
\end{align}
with $a>0$ and $p\in(0,\tfrac12)$.  Then $\cf{X}(t) = 1-2p+2p\cos(at)$,
$\sigma^2 = 2pa^2$, $\Ex[X^4] = 2pa^4$, and hence
\begin{align}
  \label{eq:three-point-kurt}
  \kappa_4(X) &= 2pa^4(1-6p), & \kurt(X) &= \frac{1-6p}{2p}.
\end{align}
For $p=\tfrac16$ this vanishes, yet $X$ is a three-point law and certainly not
Gaussian.
Consequently, non-Gaussianity is a strictly weaker requirement than
$\kurt\neq 0$ -- which is precisely why the identifiability theorems of
\cref{sec:non-gaussian,sec:gaussian-free} are formulated in terms of
non-normality rather than in terms of kurtosis.
\end{remark}

\begin{remark}[Kurtosis and tail weight]
\label{rem:kurtosis-tails}
Excess kurtosis is routinely glossed as ``tail weight''.  The gloss is
convenient, but it is not a theorem, and it is worth being precise about what
$\kurt$ does and does not measure.

\emph{What it measures.}  Let $Z := (X-\Ex X)/\sigma$ be the standardisation
of $X$, so that $\Ex[Z]=0$ and $\Ex[Z^2]=1$.  Since
$\Var(Z^2) = \Ex[Z^4]-(\Ex[Z^2])^2 = \Ex[Z^4]-1$, we have the identity
\begin{align}
  \label{eq:moors}
  \kurt(X) &= \Ex\bigl[Z^4\bigr]-3 = \Var\bigl(Z^2\bigr) - 2
            = \Ex\bigl[(Z^2-1)^2\bigr]-2 .
\end{align}
The quantity $Z^2-1$ vanishes exactly at $Z=\pm1$, that is at
$X = \Ex X \pm \sigma$, so excess kurtosis is the dispersion of $X$ about
\emph{those two points}: probability mass moved away from $\Ex X\pm\sigma$
increases it, whether it moves outwards into the tails or inwards towards the
centre.  This is Moors' reading of kurtosis as a measure of dispersion around
$\mu\pm\sigma$ \citep{moors1986meaning}; \citet{balanda1988kurtosis} phrase it
as the location- and scale-free movement of probability mass from the
``shoulders'' of a distribution into its centre \emph{and} its tails, and
argue that the notion is irreducibly vague, admitting several inequivalent
formalisations.  \Cref{eq:moors} also yields the sharp bound $\kurt\ge-2$,
with equality if and only if $Z^2=1$ almost surely, i.e.\ exactly when $X$
takes the two values $\Ex X\pm\sigma$ with probability $\tfrac12$ each -- any
affine image of the Rademacher law, whose entry in \cref{tab:examples}
records this value.
For the same reason kurtosis is not a measure of peakedness either
\citep{kaplansky1945common,darlington1970kurtosis}.

\emph{What it does not measure.}  Read ``tail weight'' as the asymptotic decay
of the density, or of $\Prb[\abs{X}>x]$, relative to a Gaussian.  Under that
reading the sign of $\kurt$ and tail weight are logically independent, in both
directions.
\begin{enumerate}[label=(\alph*)]
  \item \emph{Light tails, arbitrarily large kurtosis.}  The three-point law
        \cref{eq:three-point} is supported on $\{-a,0,a\}$, so its tails are
        as light as tails can be; yet $\kurt(X) = (1-6p)/(2p)\to\infty$ as
        $p\downarrow0$, by \cref{eq:three-point-kurt}.
  \item \emph{Heavy tails, negative kurtosis.}  Let $X$ have the mixture law
        \begin{align}
          \label{eq:heavy-platykurtic}
          \Law(X) &= (1-\varepsilon)\,\mathrm{Unif}[-1,1]
            + \varepsilon\,\mathrm{Laplace}(1),
          & \varepsilon &= 10^{-3} .
        \end{align}
        Then $\sigma^2 = 67/200$ and
        $\kurt(X) = -4515/4489 \approx -1.006$, so $X$ is sub-Gaussian in the
        sense of \cref{def:kurtosis}.
        Its density equals $\tfrac12\varepsilon e^{-\abs{x}}$ for
        $\abs{x}>1$, and therefore exceeds the density of the Gaussian law of
        the same variance for $\abs{x}\ge x_0$ with $x_0\approx 2.56$, by a
        factor of about $3.1\cdot10^{5}$ at $\abs{x}=4$.
  \item \emph{Heavy tails, no kurtosis at all.}  The Cauchy law, and the
        Student $t_\nu$ laws with $\nu\le4$, have $\Ex[X^4]=\infty$, so
        $\kurt$ is not even defined; cf.\ \cref{tab:examples}.
\end{enumerate}
Under a different reading of ``tail'' the verdict changes.  If tail weight is
taken to mean the propensity to produce observations far from the mean
\emph{measured in units of $\sigma$}, rather than asymptotic decay, then
kurtosis is a tail functional, and \citet{westfall2014} argues that this is
its only unambiguous interpretation.  Example~(a) shows what separates the two
readings: the three-point law has compact support, but its atoms sit at
$\pm a = \pm\sigma/\sqrt{2p}$, that is arbitrarily many standard deviations
from the mean.

\emph{Convention.}  In these notes ``sub-Gaussian'' and ``super-Gaussian''
always mean \cref{def:kurtosis}, the sign of $\kurt$, and never a statement
about tails.  Where tails are meant -- in \cref{tab:examples}, in
\cref{rem:two-subgaussians} and in \cref{cor:gf-criteria} -- they are named
explicitly and always refer to asymptotic decay.
\end{remark}

\begin{remark}[Two incompatible meanings of ``sub-Gaussian'']
\label{rem:two-subgaussians}
In concentration of measure and high-dimensional statistics, a real-valued
random variable $X$ is called \emph{sub-Gaussian} when its tails are dominated
by those of a Gaussian law, equivalently -- up to the value of the constant
$\sigma$ -- when
\begin{align}
  \Ex\bigl[\exp\bigl(\lambda(X-\Ex X)\bigr)\bigr]
    &\le \exp\Bigl(\tfrac12\lambda^2\sigma^2\Bigr),
  & \lambda &\in\R,
\end{align}
for some $\sigma>0$.  By \cref{thm:analytic-strip} this forces $\cf{X}$ to be
entire, and the displayed bound gives
$\abs{\cf{X}(z)} \le \Ex\bigl[e^{\abs{\Imag z}\abs{X}}\bigr]
 \le 2\exp\bigl(\abs{\Imag z}\abs{\Ex X}+\tfrac12\sigma^2\abs{\Imag z}^2\bigr)$,
so $\cf{X}$ is of order at most $2$.
This is \emph{not} the notion of \cref{def:kurtosis}, and neither notion
implies the other.
Indeed, the three-point law \cref{eq:three-point} with $p<\tfrac16$ is bounded,
hence sub-Gaussian in the concentration sense, while
\cref{eq:three-point-kurt} gives $\kurt(X) = (1-6p)/(2p) > 0$, so it is
super-Gaussian in the sense of \cref{def:kurtosis}.
Conversely, the mixture of \cref{rem:kurtosis-tails}\,(b) has exponential
tails and so is \emph{not} sub-Gaussian in the concentration sense, while its
excess kurtosis is negative, so it \emph{is} sub-Gaussian in the sense of
\cref{def:kurtosis}.
\emph{Throughout these notes, ``sub-Gaussian'' and ``super-Gaussian'' always
refer to \cref{def:kurtosis}.}
\end{remark}

\begin{table}[htbp]
  \centering
  \footnotesize
  \setlength{\tabcolsep}{5pt}
  \renewcommand{\arraystretch}{1.45}
  \begin{tabular}{@{}lccll@{}}
    \toprule
    Law of $X$ & $\cf{X}(t)$ & $\kurt(X)$ & type & analyticity of $\cf{X}$\\
    \midrule
    $\Normal(0,\sigma^2)$
      & $e^{-\sigma^2t^2/2}$
      & $0$ & Gaussian & entire, order $2$\\
    Uniform on $[-a,a]$
      & $\dfrac{\sin(at)}{at}$
      & $-\dfrac{6}{5}$ & sub-Gaussian & entire, order $1$\\
    Rademacher on $\{\pm1\}$
      & $\cos t$
      & $-2$ & sub-Gaussian & entire, order $1$\\
    Three-point \cref{eq:three-point}
      & $1-2p+2p\cos(at)$
      & $\dfrac{1-6p}{2p}$
      & sub- iff $p>\tfrac16$ & entire, order $1$\\
    Laplace, scale $b$
      & $\dfrac{1}{1+b^2t^2}$
      & $3$ & super-Gaussian
      & strip $\abs{\Imag z}<b^{-1}$\\
    Student $t_\nu$, $\nu>4$
      & $\dfrac{(\sqrt{\nu}\abs{t})^{\nu/2}
                K_{\nu/2}(\sqrt{\nu}\abs{t})}{\Gamma(\nu/2)\,2^{\nu/2-1}}$
      & $\dfrac{6}{\nu-4}$ & super-Gaussian
      & not analytic at $0$\\
    Cauchy, scale $\gamma$
      & $e^{-\gamma\abs{t}}$
      & undefined & no moments
      & not differentiable at $0$\\
    \bottomrule
  \end{tabular}
  \caption{Characteristic functions, excess kurtosis and domain of
    analyticity for the standard examples.  $K_{\nu/2}$ denotes the modified
    Bessel function of the second kind.  Power-law tails destroy analyticity
    at the origin; exponential tails give a strip; bounded or Gaussian tails
    give an entire function.}
  \label{tab:examples}
\end{table}

\subsection{Why non-Gaussianity is what makes ICA work}
\label{ssec:why-nongaussian}

The classical heuristic behind ICA is the central limit theorem: a normalised
mixture $w\T Z = \sum_j w_j Z_j$ of many independent sources looks ``more
Gaussian'' than any individual source, so one may hope to recover the sources
by searching for the directions in which the mixture is \emph{least}
Gaussian.
\Cref{eq:kappa4-additive} turns this heuristic into a precise statement.

\begin{proposition}[Kurtosis as a contrast function]
\label{prop:kurtosis-contrast}
Let $Z_1,\dots,Z_k$ be independent real-valued random variables with
$\Ex[Z_j]=0$, $\Var(Z_j)=1$ and $\Ex[Z_j^4]<\infty$, and let $w\in\R^k$ with
$\norm{w}_2 = 1$.  Then $\Var(w\T Z) = 1$ and
\begin{align}
  \bigabs{\kurt\bigl(w\T Z\bigr)}
    = \Bigabs{\sum_{j=1}^{k}w_j^4\,\kurt(Z_j)}
    \;\le\; \max_{1\le j\le k}\bigabs{\kurt(Z_j)} .
\end{align}
If $\max_j\abs{\kurt(Z_j)}>0$, then equality holds if and only if
$w = \pm e_j$ for an index $j$ attaining the maximum; so the maximisers of
$w\mapsto\abs{\kurt(w\T Z)}$ on the unit sphere are exactly the signed
coordinate directions of maximal $\abs{\kurt}$.
If instead $\kurt(Z_j)=0$ for every $j$ -- in particular if all sources are
Gaussian -- then both sides vanish for every $w$ and the criterion carries no
information at all.
\end{proposition}

\begin{proof}
Since $\Var(Z_j)=1$ and $\norm{w}_2=1$ we get $\kappa_2(w\T Z) = 1$ by
\cref{prop:cumulant-calculus}, so $\kurt(w\T Z) = \kappa_4(w\T Z)$, and
\cref{eq:kappa4-additive} gives the stated identity.
Then
\begin{align}
  \Bigabs{\sum_{j}w_j^4\kurt(Z_j)}
   \le \sum_j w_j^4\bigabs{\kurt(Z_j)}
   \le \max_j\bigabs{\kurt(Z_j)}\cdot\sum_j w_j^4
   \le \max_j\bigabs{\kurt(Z_j)},
\end{align}
where the last step uses
$\sum_j w_j^4 \le \bigl(\sum_j w_j^2\bigr)^2 = 1$.
Write $M := \max_j\abs{\kurt(Z_j)}$ and assume $M>0$.
Then equality in the last inequality, which reads $M\sum_j w_j^4 \le M$, holds
if and only if $\sum_j w_j^4 = 1$, i.e.\ if and only if exactly one coordinate
of $w$ is non-zero and thus $w = \pm e_j$; equality in the middle inequality
then forces $\abs{\kurt(Z_j)} = M$.
If $M=0$ all three quantities vanish for every $w$, which is the last
assertion.
\end{proof}

\Cref{prop:kurtosis-contrast} is the theoretical justification of the
kurtosis-based contrast functions of \citet{comon1994} and
\citet{hyvarinen2001}, and it already displays the two phenomena that the rest
of these notes make exact:
\begin{enumerate}[label=(\roman*)]
  \item non-Gaussian sources are recoverable, and the recovery is only ever
        determined up to sign, scale and the ordering of the coordinates --
        the ambiguities of \cref{eq:ica-intro};
  \item Gaussian sources are not recoverable at all, since the contrast
        degenerates.  Compare \cref{rem:gaussian-not-identifiable}.
\end{enumerate}
What the following sections add is that neither fourth moments nor any moments
at all are actually needed: the correct hypothesis is non-normality, in the
sense of \cref{thm:gaussian-char}, and the correct tool is
\cref{thm:marcinkiewicz} rather than a contrast function.

\begin{remark}[Gaussian sources are genuinely unidentifiable]
\label{rem:gaussian-not-identifiable}
Let $Z \sim \Normal(0,\Id{k})$ and let $A \in \R^{p\times k}$.
For every orthogonal $Q \in \Orth{k}$ we have $QZ \sim \Normal(0,\Id{k})$ by
\cref{eq:cf-affine}, so $QZ$ again has mutually independent components, while
\begin{align}
  A Z &\eqd (A Q\T)(Q Z).
\end{align}
So $A$ and $AQ\T$ are indistinguishable from the law of $X = AZ$, and the
ambiguity is a whole orthogonal group rather than only the discrete
permutations and the diagonal rescalings of \cref{sec:intro}.
This is why every identifiability statement below either excludes Gaussian
sources, or explicitly quantifies the remaining Gaussian ambiguity.
\end{remark}

\section{Identifiability for Non-Constant Independent Sources}
\label{sec:non-constant}

We begin with the weakest set of assumptions on the sources: they are mutually
independent and non-constant, and nothing else.
Before any uniqueness statement can be made we have to normalise the
representation, because two trivial mechanisms produce genuinely different
representations of the same random vector: zero columns of the mixing matrix,
and columns that are proportional to each other.
\Cref{rem:normalisation} removes both.

\begin{remark}[Normalising a representation]
\label{rem:normalisation}
Consider a representation of a $p$-dimensional random vector $X \in \R^p$:
\begin{align}
  \label{eq:rep-basic}
  X &= A Z + \mu = \sum_{j=1}^{k} a_j Z_j + \mu,
\end{align}
where $Z = [Z_1,\dots,Z_k]\T \in \R^{k}$ is a random vector with mutually
independent components $\{Z_1,\dots,Z_k\}$, where $\mu \in \R^p$ is a
(non-random) column vector, and where
$A = [a_1,\dots,a_k] \in \R^{p\times k}$ is a (non-random) matrix with column
vectors $a_j \in \R^p$.

\emph{Step 1: removing zero columns and constant sources.}
By deleting the zero columns $a_j = 0$ and by absorbing every almost surely
constant summand $a_j Z_j$ into $\mu$, we arrive at a representation of $X$ in
which $A$ has no zero column and in which every component of $Z$ is almost
surely non-constant.

\emph{Step 2: removing proportional columns.}
The columns of $A$ may still be proportional to each other, say
$a_j = \lambda a_l$ with a proportionality constant
$\lambda \in \R \setminus \{0\}$ and $j \neq l$.
This is a true ambiguity of the representation, since we can always merge the
two corresponding sources,
\begin{align}
  a_j Z_j + a_l Z_l &= a_l\bigl(\lambda Z_j + Z_l\bigr)
                     =: a_l \widetilde Z_l .
\end{align}
For a first identifiability result we therefore have to aggregate these
variables.
Proportionality is an equivalence relation on $\R^p \setminus \{0\}$, so we
can choose columns $\tilde a_1,\dots,\tilde a_{k'}$ of $A$ forming a system of
representatives of
$\bigl(\{a_1,\dots,a_k\}\setminus\{0\}\bigr)/\mathord{\propto}$;
that is, a set of non-zero columns of $A$ that are pairwise non-proportional
and such that every non-zero $a_j$ is proportional to a (then necessarily
unique) $\tilde a_l$, written $a_j \propto \tilde a_l$.
Denoting the corresponding proportionality constants by $\lambda_j$, i.e.\
$a_j = \lambda_j \tilde a_l$, we obtain
\begin{align}
  X &= A Z + \mu \\
    &= \sum_{j=1}^{k} a_j Z_j + \mu \\
    &= \sum_{l=1}^{k'} \tilde a_l
       \underbrace{\Biggl( \sum_{a_j \propto \tilde a_l} \lambda_j Z_j \Biggr)
       }_{=:\ \widetilde Z_l} + \mu \\
    &= \sum_{l=1}^{\tilde k} \tilde a_l \widetilde Z_l
       + \underbrace{\sum_{l=\tilde k+1}^{k'} \tilde a_l \widetilde Z_l
         + \mu}_{=:\ \tilde\mu} \\
    &= \widetilde A \widetilde Z + \tilde\mu ,
\end{align}
where -- after possibly re-indexing -- $\tilde k\le k'$ denotes the number of
indices $l$ for which $\widetilde Z_l$ is \emph{not} almost surely constant,
so that the sum inside $\tilde\mu$ collects exactly the almost surely
constant $\widetilde Z_l$ and $\tilde\mu$ is (a.s.\ equal to) a
deterministic vector, and where
$\widetilde A := [\tilde a_1,\dots,\tilde a_{\tilde k}]$ and
$\widetilde Z := [\widetilde Z_1,\dots,\widetilde Z_{\tilde k}]\T$.

If Step~1 has already been carried out, so that every $Z_j$ is non-constant,
then in fact $\tilde k = k'$ and the second sum is empty: an independent sum
of non-constant random variables is never a.s.\ constant, because
$\abs{\cf{Y_1}\cf{Y_2}}\equiv1$ forces $\abs{\cf{Y_1}}\equiv1$ and hence
$Y_1$ degenerate.  We nevertheless state Step~2 in the general form, so that
it can be applied on its own.

We have thus found a representation of $X$ in which the columns of
$\widetilde A$ are non-zero and pairwise non-proportional and in which every
component of $\widetilde Z$ is almost surely non-constant.
Moreover, if the components of $Z$ are mutually independent then so are those
of $\widetilde Z$, since each $Z_j$ enters exactly one $\widetilde Z_l$.
\end{remark}

Now that the existence of such normalised representations is established, we
restrict attention to them and investigate their uniqueness -- up to the
transformations that are unavoidable.

\begin{theorem}[Identifiability for independent non-constant sources;
{\citealp[Chapter~10, Lemma~10.2.3 and Theorem~10.3.1]{kagan1973characterization}}]
\label{thm:kagan}
Let $X \in \R^p$ be a $p$-dimensional random vector with two representations
\begin{align}
  \label{eq:two-rep-kagan}
  A\up{1} Z\up{1} + \mu\up{1} \;=\; X \;=\; A\up{2} Z\up{2} + \mu\up{2},
\end{align}
with the following properties for $i = 1,2$:
\begin{enumerate}[label=(\roman*)]
  \item $A\up{i} \in \R^{p \times k\up{i}}$ is a (non-random) matrix whose
        columns are non-zero and pairwise non-proportional;
  \item $\mu\up{i} \in \R^p$ is a (non-random) column vector;
  \item $Z\up{i} \in \R^{k\up{i}}$ is a random vector such that
        \begin{enumerate}[label=(\alph*)]
          \item its $k\up{i}$ components
                $\{Z\up{i}_1,\dots,Z\up{i}_{k\up{i}}\}$ are mutually
                independent, and
          \item each component $Z\up{i}_j$ is almost surely non-constant,
                i.e.\ $\Law(Z\up{i}_j)$ is not a Dirac measure,
                $j = 1,\dots,k\up{i}$.
        \end{enumerate}
\end{enumerate}
Then
\begin{align}
  \mu\up{2} - \mu\up{1} &\in \im A\up{1} = \im A\up{2},
  & \rank\bigl(A\up{1}\bigr) &= \rank\bigl(A\up{2}\bigr) .
\end{align}
In particular there exist $c\up{1} \in \R^{k\up{1}}$ and
$c\up{2} \in \R^{k\up{2}}$ with
$\mu\up{2} - \mu\up{1} = A\up{1} c\up{1} = A\up{2} c\up{2}$.

Furthermore, the following statements hold.
\begin{enumerate}[label=(\arabic*)]
  \item If the $l$-th column of $A\up{2}$ is not proportional to any column of
        $A\up{1}$, then $Z\up{2}_l$ is Gaussian.
  \item Assume that the $l$-th column of $A\up{2}$ is proportional to the
        $j$-th column of $A\up{1}$ with proportionality
        constant\footnote{\label{fn:missing-constant}
        Note that this proportionality constant was not reintroduced in
        {\citep[Theorem~10.3.1]{kagan1973characterization}} after it had been
        removed ``without loss of generality'' in
        {\citep[Lemmata~10.2.4 and 10.2.5]{kagan1973characterization}}.
        It matters here, since we do not normalise the columns of $A\up{i}$.}
        $\lambda \in \R\setminus\{0\}$, i.e.\
        $a\up{2}_l = \lambda \cdot a\up{1}_j$.
        Then there exists a (complex) polynomial $g$ such that the
        characteristic functions of $Z\up{2}_l$ and $Z\up{1}_j$ satisfy, in a
        neighbourhood of the origin,
        \begin{align}
          \label{eq:cf-kagan}
          \cf{Z\up{2}_l}(\lambda t)
            &= \cf{Z\up{1}_j}(t)\cdot \exp\bigl(g(t)\bigr).
        \end{align}
        In particular, by Marcinkiewicz' \cref{thm:marcinkiewicz} together
        with \cref{rem:marcinkiewicz-use}, $Z\up{2}_l$ is Gaussian if and only
        if $Z\up{1}_j$ is Gaussian.
\end{enumerate}
\end{theorem}

The proof is in \cref{app:kagan-proof}.  It is a self-contained
finite-difference argument that uses nothing from the sections in between, so
the reader may turn to it at any point; we have deferred it only because it is
considerably longer than the statement and would interrupt the development
here.  The first part of the theorem is elementary, and we record it right
away as \cref{prop:affine-hull} below.

\begin{remark}[Reading \cref{thm:kagan}]
\label{rem:reading-kagan}
\Cref{thm:kagan} is best read as a dichotomy at the level of \emph{columns}.
Every column of $A\up{2}$ either
\begin{enumerate}[label=(\roman*)]
  \item is proportional to a column of $A\up{1}$, in which case, by
        \cref{eq:cf-kagan}, the two associated sources agree up to that
        proportionality factor and a factor $\exp(g)$ with $g$ a polynomial;
        or
  \item is not, in which case its source is forced to be Gaussian, and by
        \cref{rem:gaussian-not-identifiable} we should not have expected to
        recover it in the first place.
\end{enumerate}
In case (i), \cref{thm:kagan} by itself bounds neither $\Deg(g)$ nor the
direction in which the perturbation acts; under the additional hypotheses of
\cref{thm:non-gaussian} the factor becomes an independent Gaussian
perturbation of one of the two sources.
All the work in the following sections consists of ruling out the second
alternative by strengthening the assumptions on the sources, and of
bookkeeping the Gaussian perturbation in the first.
\end{remark}

\Cref{thm:kagan} is the engine of everything that follows.  Its first part --
the statement about images, ranks and offsets -- needs none of the machinery
of \cref{app:kagan-proof}, so we prove it here; only the column dichotomy
(1)--(2) is postponed.

\begin{proposition}[The first part of \cref{thm:kagan}]
\label{prop:affine-hull}
Under the hypotheses of \cref{thm:kagan}, the affine hull of the support of
$\Law(X)$ equals $\mu\up{i} + \im A\up{i}$ for $i=1,2$.  Consequently
\begin{align}
  \mu\up{2}-\mu\up{1} &\in \im A\up{1} = \im A\up{2},
  & \rank\bigl(A\up{1}\bigr) &= \rank\bigl(A\up{2}\bigr).
\end{align}
\end{proposition}

\begin{proof}
Fix $i$ and abbreviate $A := A\up{i}$, $Z := Z\up{i}$, $\mu := \mu\up{i}$,
$k := k\up{i}$.
Since the components of $Z$ are mutually independent, $\Law(Z)$ is the product
of the laws of its components, and the support of a finite product of Borel
probability measures on second countable spaces -- in particular on $\R$ -- is
the product of the supports:
$\supp\Law(Z) = \prod_{j=1}^{k}\supp\Law(Z_j)$.
Each factor contains at least two points, because $Z_j$ is almost surely
non-constant.
Hence, for each $j$, choosing two distinct points of $\supp\Law(Z_j)$ and
fixing an arbitrary point of $\supp\Law(Z_l)$ in every other coordinate
$l\neq j$ produces two elements of $\supp\Law(Z)$ whose difference is a
non-zero multiple of $e_j$.  The direction space of the affine hull of
$\supp\Law(Z)$ therefore contains every $e_j$, so
\begin{align}
  \operatorname{aff}\bigl(\supp\Law(Z)\bigr) &= \R^{k} .
\end{align}
Now let $T(z) := Az+\mu$.  For a continuous map $T$ one has
$\supp\Law(T(Z)) = \overline{T\bigl(\supp\Law(Z)\bigr)}$: the inclusion
``$\supseteq$'' holds because $T^{-1}(U)$ is an open neighbourhood of any
preimage point and therefore has positive mass, and ``$\subseteq$'' because
the closed set $\overline{T(\supp\Law(Z))}$ has $\Law(T(Z))$-measure~$1$.
Affine hulls are unchanged by taking closures, since affine subspaces are
closed, and they commute with affine maps.  Hence
\begin{align}
  \operatorname{aff}\bigl(\supp\Law(X)\bigr)
    &= T\bigl(\operatorname{aff}(\supp\Law(Z))\bigr)
     = A\,\R^{k}+\mu = \mu + \im A .
\end{align}
Applying this to $i=1$ and $i=2$ and equating the two descriptions of the same
affine subspace gives $\mu\up{1}+\im A\up{1} = \mu\up{2}+\im A\up{2}$, whence
$\im A\up{1} = \im A\up{2}$ (the two direction spaces coincide) and
$\mu\up{2}-\mu\up{1}\in\im A\up{1}$.  Equality of ranks is equality of the
dimensions of these images.
\end{proof}

\section{Identifiability for Non-Gaussian Independent Sources}
\label{sec:non-gaussian}

\Cref{thm:kagan} exhibits a clear distinction between Gaussian and non-Gaussian
components in a representation $X = AZ + \mu$: Gaussian components tend to be
non-identifiable.
It therefore makes sense to separate them off from the beginning and to work
with a representation in which the Gaussian part appears as one additive noise
vector.

\begin{remark}[Separated representations]
\label{rem:separated}
Start from a normalised representation as in \cref{rem:normalisation} and
write it as
\begin{align}
  X = A Z + \mu
    &= \sum_{j=1}^{k} a_j Z_j + \mu \\
    &= \sum_{l=1}^{\hat k} a_l Z_l
       + \underbrace{\sum_{l=\hat k+1}^{k} a_l Z_l + \mu
         }_{=:\ \widetilde E \ \sim\ \Normal(\tilde\mu,\widetilde\Sigma)} \\
    &= \widetilde A \widetilde Z + \widetilde E ,
\end{align}
where -- after possibly re-indexing -- $\hat k$ is the number of
non-Gaussian components and the second sum collects all Gaussian
components $Z_l$ into a single $p$-variate (possibly degenerate)
Gaussian random vector $\widetilde E \sim \Normal(\tilde\mu,\widetilde\Sigma)$,
where $\widetilde Z := [Z_1,\dots,Z_{\hat k}]\T$ is a vector of mutually
independent, non-constant, non-Gaussian components with
$\widetilde Z \indep \widetilde E$, and where the mixing matrix
$\widetilde A := [a_1,\dots,a_{\hat k}]$ has non-zero, pairwise
non-proportional columns.
\end{remark}

Before stating the general identifiability result -- which will be based on
\cref{thm:kagan} -- we investigate how the above outsourcing of the Gaussian
components can be \emph{reverted}.
This is the technical heart of this section: we need to be able to move
Gaussian mass back and forth between the noise vector and additional columns
of the mixing matrix.

\begin{lemma}[Trading Gaussian noise against columns of the mixing matrix]
\label{lem:noise-trade}
Let $X \in \R^p$ be a $p$-variate random vector with a representation
\begin{align}
  X &= A Z + E,
\end{align}
where
\begin{itemize}
  \item $A \in \R^{p\times k}$ is a (non-random) matrix with non-zero,
        pairwise non-proportional columns,
  \item $Z \in \R^k$ is a random vector with mutually independent,
        non-constant, non-Gaussian components, and
  \item $E \in \R^p$ is a (possibly degenerate) Gaussian random vector,
        $E \sim \Normal(\mu,\Sigma)$, with mean vector $\mu \in \R^p$ and
        positive semi-definite covariance matrix $\Sigma \in \R^{p\times p}$
        of rank $r \le p$, such that $Z \indep E$.
\end{itemize}
Then there exist a matrix $B \in \R^{p\times r}$ of rank $r$ and a random
vector $V \sim \Normal(0,\Id{r})$ such that
\begin{align}
  \Sigma &= B B\T, & E &= B V + \mu \ \text{a.s.}, & Z &\indep V,
\end{align}
which yields the representation
\begin{align}
  \label{eq:rep-with-B}
  X &= A Z + B V + \mu .
\end{align}
If $r = 0$ then $B$ is the empty $p\times0$ matrix, $V$ the empty vector and
$E=\mu$ a.s.  If $r \ge 2$ then $B$ can moreover be chosen such that no column
of $B$ is proportional to any column of $A$.

Furthermore, for either choice of $B$ and for every $r \ge 0$, we obtain a
normalised representation in the sense of \cref{rem:normalisation}:
\begin{align}
  \label{eq:rep-normalised-from-B}
  X &= \bmat{A & \widetilde B}\bmat{\widetilde Z\\ \widetilde V} + \mu,
\end{align}
where
\begin{itemize}
  \item $\widetilde B \in \R^{p \times \tilde r}$ with $\tilde r \le r$
        consists of those columns of $B$ that are not proportional to any
        column of $A$, and $\widetilde V \sim \Normal(0,\Id{\tilde r})$
        collects the corresponding components of $V$;
  \item $\widetilde Z = Z + \widehat V$ with
        $\widehat V \sim \Normal(0,\Gamma)$ for a diagonal matrix $\Gamma$
        (possibly with zeros on the diagonal), where
        $\{Z, \widehat V, \widetilde V\}$ is mutually independent;
  \item every component of $\widetilde Z$ is non-Gaussian, the random vector
        $[\widetilde Z\T, \widetilde V\T]\T$ has mutually independent
        non-constant components, and the matrix $[A, \widetilde B]$ has
        non-zero, pairwise non-proportional columns.
\end{itemize}
\end{lemma}

\begin{proof}
The first statement follows directly from the spectral decomposition of the
symmetric positive semi-definite matrix $\Sigma$ (equivalently, from its
singular value decomposition):
\begin{align}
  \Sigma &= Q \bmat{\Delta_r & 0\\ 0 & 0} Q\T,
\end{align}
with an orthogonal matrix $Q \in \Orth{p} \subseteq \R^{p\times p}$ and a
diagonal matrix $\Delta_r \in \R^{r\times r}$ with strictly positive diagonal
entries.  We may then put
\begin{align}
  B &:= Q\bmat{\Delta_r^{1/2}\\ 0},
  & V &:= \bmat{\Delta_r^{-1/2} & 0} Q\T (E-\mu),
\end{align}
which yields exactly the desired properties:
\begin{align}
  \rank(B) &= r, & \Sigma &= B B\T, & E &= BV+\mu \ \text{a.s.},
  & V &\sim \Normal(0,\Id{r}), & Z &\indep V .
\end{align}
Note that if $G \in \Orth{r}$ is orthogonal, then replacing $B$ by $BG$ and
$V$ by $G\T V$ preserves all of these properties.
For $r \ge 2$ this freedom can be used to make every column of $BG$
non-proportional to every column of $A$; the technical details are the content
of \cref{lem:not-prop}.

In any case, let $\widetilde B$ be the matrix consisting of those columns of
$B = [b_1,\dots,b_r]$ that are not proportional to any column of $A$, and let
$\widetilde V$ be the random vector of the corresponding components of $V$.
After re-indexing we may assume these to be the first $\tilde r$ columns of
$B$, so that $\widetilde B = [b_1,\dots,b_{\tilde r}]$.
For $l > \tilde r$ we then have $b_l = \lambda_l a_{j_l}$ for a unique index
$j_l$ and some $\lambda_l \in \R \setminus \{0\}$.
With this we obtain
\begin{align}
  X &= A Z + E \\
    &= A Z + B V + \mu \\
    &= \sum_{j=1}^{k} a_j Z_j + \sum_{l=1}^{r} b_l V_l + \mu \\
    &= \sum_{j=1}^{k} a_j Z_j + \sum_{l=1}^{\tilde r} b_l V_l
       + \sum_{l=\tilde r+1}^{r} b_l V_l + \mu \\
    &= \sum_{j=1}^{k} a_j Z_j + \sum_{l=1}^{\tilde r} b_l V_l
       + \sum_{l=\tilde r+1}^{r} a_{j_l}\lambda_l V_l + \mu \\
    &= \sum_{j=1}^{k} a_j
       \underbrace{\Biggl( Z_j
         + \overbrace{\sum_{b_l \propto a_j} \lambda_l V_l}^{=:\ \widehat V_j}
         \Biggr)}_{=:\ \widetilde Z_j}
       + \sum_{l=1}^{\tilde r} b_l V_l + \mu \\
    &= \sum_{j=1}^{k} a_j \widetilde Z_j
       + \sum_{l=1}^{\tilde r} b_l V_l + \mu \\
    &= \bmat{A & \widetilde B}\bmat{\widetilde Z\\ \widetilde V} + \mu .
\end{align}
Each $\widetilde Z_j$ is non-Gaussian, being the sum of the non-Gaussian $Z_j$ and
independent Gaussian noise, by Marcinkiewicz' \cref{thm:marcinkiewicz}, cf.\
\cref{rem:marcinkiewicz-use}\,(ii).
The components of $[\widetilde Z\T,\widetilde V\T]\T$ are mutually independent
because the components of $[Z\T,V\T]\T$ are mutually independent and each
component $V_l$ enters exactly one component of
$[\widetilde Z\T,\widetilde V\T]\T$.
The columns of $\widetilde B$ are pairwise non-proportional because
$\rank(B) = r$ makes the columns of $B$ linearly independent, and they are
non-proportional to the columns of $A$ by the very choice of $\widetilde B$.
So the matrix $[A,\widetilde B]$ has
non-zero, pairwise non-proportional columns.
\end{proof}

The following lemma provides the rotation used in the proof above.  It is the
only place where we need a genuinely geometric argument.

\begin{lemma}[Avoiding finitely many hyperplanes]
\label{lem:avoid-hyperplanes}
Let $k \ge 1$ and let $\mathcal{A}\subseteq\R^k\setminus\{0\}$ be a finite set.
Then there exists $q\in\R^k$ with $q\T a \neq 0$ for every $a\in\mathcal{A}$.
In fact the set of such $q$ is open and its complement is a Lebesgue null set.
\end{lemma}

\begin{proof}
For $a \neq 0$ the orthogonal complement
$a^\perp = \{v\in\R^k : v\T a = 0\}$ is a linear subspace of dimension $k-1$,
hence closed and a Lebesgue null set.
A finite union of null sets is a null set, so
$\R^k\setminus\bigcup_{a\in\mathcal{A}}a^\perp$ has full Lebesgue measure and
is in particular non-empty; it is open as the complement of a finite union of
closed sets.
\end{proof}

\begin{lemma}[Rotating a frame away from finitely many directions]
\label{lem:not-prop}
Let $p \ge k \ge 2$, let $\mathcal{A} \subseteq \R^p \setminus \{0\}$ be a
finite set of non-trivial (column) vectors in $\R^p$, and let
$B \in \R^{p\times k}$ be a matrix with $\rank(B) = k$.
Then there exists an orthogonal matrix
$Q \in \Orth{k} \subseteq \R^{k\times k}$ such that for every $a \in
\mathcal{A}$ and every column vector $c$ of $C := BQ \in \R^{p\times k}$ the
set $\{a,c\}$ is linearly independent, i.e.\ $a$ is not proportional to $c$
(and vice versa).
\end{lemma}

\begin{proof}
Put $B\pinv := (B\T B)^{-1}B\T \in \R^{k\times p}$, which is a left inverse of
$B$ by the assumption $\rank(B) = k$.

We first discard from $\mathcal{A}$ those $a$ with $B\T a = 0$.  Such an $a$
is automatically linearly independent of every column $c$ of $BQ$, whatever
$Q\in\Orth{k}$: indeed $c\in\im B$, so $a=\lambda c$ with $\lambda\neq0$
would give $a\in\im B$, whereas $B\T a = 0$ says $a\perp\im B$, and together
these force $a=0$, which is excluded.
Note also that $B\pinv a = 0$ if and only if $B\T a = 0$, since $B\T B$ is
invertible.  Replacing $\mathcal{A}$ by $\{a\in\mathcal{A}: B\T a\neq0\}$ --
if this is empty any $Q\in\Orth{k}$ will do -- we may therefore assume
$B\pinv a\neq0$ for every $a\in\mathcal{A}$, and consider the sets
\begin{align*}
  \widetilde{\mathcal{A}}
    &:= \bigl\{ B\pinv a \given a \in \mathcal{A} \bigr\}
       \subseteq \R^k\setminus\{0\}, \\
  \mathcal{A}'
    &:= \widetilde{\mathcal{A}} \cup
       \bigl\{ t_{\tilde a} \given \tilde a \in \widetilde{\mathcal{A}}
       \bigr\} \subseteq \R^k,
\end{align*}
where for $\tilde a \in \widetilde{\mathcal{A}}$ the column vector
$t_{\tilde a}\in\R^k$ is any fixed non-trivial vector with
$t_{\tilde a}\T \tilde a = 0$; such a vector exists because $k \ge 2$.
By \cref{lem:avoid-hyperplanes} there exists a vector
$q_1 \in \R^k$ that is not orthogonal to any vector in $\mathcal{A}'$, and
without loss of generality we may scale it to unit norm, $\norm{q_1}_2 = 1$.

It follows that $q_1$ is not proportional to any
$\tilde a \in \widetilde{\mathcal{A}}$: otherwise $q_1 = \lambda \tilde a$
with $\lambda \in \R$, and hence
$t_{\tilde a}\T q_1 = \lambda\, t_{\tilde a}\T \tilde a = 0$, contradicting
the choice of $q_1$.

Since $q_1$ is also not orthogonal to any element of
$\widetilde{\mathcal{A}}$, the set $\widetilde{\mathcal{A}}$ is disjoint from
the orthogonal complement
\begin{align*}
  q_1^{\perp} &:= \bigl\{ v \in \R^k \given v\T q_1 = 0 \bigr\},
\end{align*}
which is a $(k-1)$-dimensional subspace of $\R^k$ and thus admits an
orthonormal basis $q_2,\dots,q_k$.
We can now build the orthogonal matrix
\begin{align*}
  Q &:= [q_1,\dots,q_k] \in \Orth{k} \subseteq \R^{k\times k}.
\end{align*}
Assume, by way of contradiction, that there exist $a \in \mathcal{A}$ and a
column vector $c$ of $C := BQ$, say $c = Ce_j = Bq_j$, such that $\{a,c\}$ is
linearly dependent.
Since $a,c \neq 0$ there is then a $\lambda \in \R\setminus\{0\}$ with
$a = \lambda c$, and multiplying by $B\pinv$ gives
\begin{align*}
  \tilde a := B\pinv a = \lambda \cdot B\pinv c
            = \lambda \cdot B\pinv B Q e_j = \lambda \cdot q_j .
\end{align*}
For $j = 1$ this contradicts $q_1 \neq \lambda^{-1}\tilde a$ for all
$\tilde a \in \widetilde{\mathcal{A}}$.
For $j \ge 2$ it contradicts $\tilde a \in \widetilde{\mathcal{A}}$ together
with
$\widetilde{\mathcal{A}} \cap \spn\{q_2,\dots,q_k\} = \emptyset$.
So $\{a,c\}$ is linearly independent for every $a \in \mathcal{A}$ and every
column vector $c$ of $BQ$, which proves the claim.
\end{proof}

\begin{theorem}[Identifiability for independent non-Gaussian sources]
\label{thm:non-gaussian}
Let $X \in \R^p$ be a random vector and assume that we have two
representations
\begin{align}
  \label{eq:two-rep-nonnormal}
  A\up{1} Z\up{1} + E\up{1} \;=\; X \;=\; A\up{2} Z\up{2} + E\up{2},
\end{align}
with the following properties for $i=1,2$:
\begin{enumerate}[label=(\roman*)]
  \item $A\up{i} \in \R^{p\times k\up{i}}$ is a (non-random) matrix whose
        columns are non-zero and pairwise non-proportional;
  \item $E\up{i} \in \R^p$ is a $p$-variate Gaussian random vector,
        $E\up{i} \sim \Normal(\mu\up{i}, \Sigma\up{i})$, possibly degenerate,
        with mean vector $\mu\up{i}\in\R^p$ and positive semi-definite
        covariance matrix $\Sigma\up{i}$;
  \item $Z\up{i} \in \R^{k\up{i}}$ is a random vector such that
        \begin{enumerate}[label=(\alph*)]
          \item its $k\up{i}$ components
                $\{Z\up{i}_1,\dots,Z\up{i}_{k\up{i}}\}$ are mutually
                independent, and
          \item each component $Z\up{i}_j$ is a (non-constant) non-Gaussian
                random variable, for $j=1,\dots,k\up{i}$;
        \end{enumerate}
  \item $E\up{i}$ is independent of $Z\up{i}$: $E\up{i} \indep Z\up{i}$.
\end{enumerate}
Then $k\up{1} = k\up{2} =: k$, and there exist a permutation matrix
$P = P(\rho) \in \R^{k\times k}$, given by $Pe_j = e_{\rho(j)}$ for a
permutation $\rho$, and an invertible diagonal matrix
$\Lambda = \diag(\lambda_1,\dots,\lambda_k) \in \R^{k\times k}$ such that
\begin{align}
  A\up{2} &= A\up{1}P\Lambda,
  & \rank\bigl(A\up{2}\bigr) &= \rank\bigl(A\up{1}\bigr),
\end{align}
and such that for every $j = 1,\dots,k$ there exists a (complex) polynomial
$g_j$ with
\begin{align}
  \label{eq:cf-nonnormal}
  \cf{Z\up{2}_j}(\lambda_j t)
    &= \cf{Z\up{1}_{\rho(j)}}(t)\cdot\exp\bigl(g_j(t)\bigr)
\end{align}
in a neighbourhood of the origin.

Furthermore, and in particular, $A\up{2}$ has a left inverse
($\rank(A\up{2}) = k$) if and only if $A\up{1}$ has a left inverse
($\rank(A\up{1}) = k$).
If this is the case, then for each $j=1,\dots,k$ separately we have
$\Deg(g_j)\le 2$ and there exist a constant $\nu_j \in \R$ and a (possibly
degenerate) random variable $G_j \sim \Normal(0,\sigma_j^2)$ such that at
least one of the following two cases holds, according to the sign of the
degree-two coefficient of $g_j$:
\begin{align}
  \bigl(\lambda_j Z\up{2}_j + \nu_j\bigr)
    &\eqd Z\up{1}_{\rho(j)} + G_j,
    & G_j &\indep Z\up{1}_{\rho(j)},
    & \text{if } g_j'' &\le 0, \\
  Z\up{1}_{\rho(j)}
    &\eqd \bigl(\lambda_j Z\up{2}_j + \nu_j\bigr) + G_j,
    & G_j &\indep Z\up{2}_j,
    & \text{if } g_j'' &\ge 0 .
\end{align}
If $g_j''\neq0$ exactly one of the two cases applies; if $g_j'' = 0$ then
$G_j = 0$ and both hold and coincide.
The polynomial $g_j$ is determined near the origin only up to an additive
constant in $2\pi i\,\Ints$, because $\cf{Z\up{1}_{\rho(j)}}$ does not vanish
there; in particular $\Deg(g_j)$ and $g_j''$ are well defined.
\end{theorem}

\begin{proof}
Apply \cref{lem:noise-trade} to both sides to obtain the
representations
\begin{align}
  \bmat{A\up{1} & \widetilde B\up{1}}
  \bmat{\widetilde Z\up{1}\\ \widetilde V\up{1}} + \mu\up{1}
  \;=\; X \;=\;
  \bmat{A\up{2} & \widetilde B\up{2}}
  \bmat{\widetilde Z\up{2}\\ \widetilde V\up{2}} + \mu\up{2},
\end{align}
which satisfy the requirements of \cref{thm:kagan}.

Since $\widetilde Z\up{2}_j$ is non-Gaussian, \cref{thm:kagan} implies that
the $j$-th column $a\up{2}_j$ of $A\up{2}$ is proportional to some column of
$[A\up{1}, \widetilde B\up{1}]$.
All components of $\widetilde V\up{1}$ are Gaussian whereas those of
$\widetilde Z\up{1}$ are non-Gaussian, so \cref{thm:kagan} forces
$a\up{2}_j$ to be proportional to a column of $A\up{1}$, say to
$a\up{1}_{\rho(j)}$.
The index $\rho(j)$ is unique because the columns of $A\up{1}$ are pairwise
non-proportional, and $j\mapsto\rho(j)$ is injective: $\rho(j)=\rho(j')$ would
give $a\up{2}_j \propto a\up{1}_{\rho(j)} \propto a\up{2}_{j'}$, contradicting
pairwise non-proportionality of the columns of $A\up{2}$; hence
$k\up{2}\le k\up{1}$.
Since the same argument applies to all columns of $A\up{2}$, and also with the
roles of $A\up{1}$ and $A\up{2}$ interchanged, we obtain a bijection between
the columns of $A\up{1}$ and those of $A\up{2}$, together with their
proportionality constants.
This already proves the first part of the claim, namely
$k\up{1} = k\up{2} =: k$ and the representation
\begin{align*}
  A\up{2} &= A\up{1}P\Lambda
\end{align*}
for some permutation matrix $P = P(\rho)$ with permutation $\rho$ and some
invertible diagonal matrix $\Lambda = \diag(\lambda_1,\dots,\lambda_k)$.
Indeed, multiplying by the unit vector $e_j$ from the right gives
\begin{align*}
  a\up{2}_j = A\up{2}e_j
    &= A\up{1}P\Lambda e_j
     = A\up{1}P\lambda_j e_j
     = A\up{1}\lambda_j e_{\rho(j)}
     = \lambda_j a\up{1}_{\rho(j)} .
\end{align*}
Again by \cref{thm:kagan} there exists a polynomial $\tilde g_j$ with
\begin{align}
  \cf{\widetilde Z\up{2}_j}(\lambda_j t)
    = \cf{\lambda_j \widetilde Z\up{2}_j}(t)
    = \cf{\widetilde Z\up{1}_{\rho(j)}}(t)\cdot
      \exp\bigl(\tilde g_j(t)\bigr)
\end{align}
in a neighbourhood of the origin.
\Cref{lem:noise-trade} gives us the representations
\begin{align}
  \widetilde Z\up{2}_j &= Z\up{2}_j + \widehat V\up{2}_j,
    & Z\up{2}_j &\indep \widehat V\up{2}_j,\\
  \widetilde Z\up{1}_{\rho(j)}
    &= Z\up{1}_{\rho(j)} + \widehat V\up{1}_{\rho(j)},
    & Z\up{1}_{\rho(j)} &\indep \widehat V\up{1}_{\rho(j)},
\end{align}
with Gaussian $\widehat V\up{2}_j$ and
$\widehat V\up{1}_{\rho(j)}$.
Taking characteristic functions yields
\begin{align}
  \cf{\lambda_j \widetilde Z\up{2}_j}(t)
    &= \cf{\lambda_j Z\up{2}_j}(t)\cdot
       \cf{\lambda_j \widehat V\up{2}_j}(t),\\
  \cf{\widetilde Z\up{1}_{\rho(j)}}(t)
    &= \cf{Z\up{1}_{\rho(j)}}(t)\cdot
       \cf{\widehat V\up{1}_{\rho(j)}}(t).
\end{align}
Combining this with the previous display we get
\begin{align}
  \cf{Z\up{2}_j}(\lambda_j t) = \cf{\lambda_j Z\up{2}_j}(t)
    &= \cf{\lambda_j\widetilde Z\up{2}_j}(t)\cdot
       \cf{\lambda_j \widehat V\up{2}_j}(t)^{-1}\\
    &= \cf{\widetilde Z\up{1}_{\rho(j)}}(t)\cdot
       \exp\bigl(\tilde g_j(t)\bigr)\cdot
       \cf{\lambda_j \widehat V\up{2}_j}(t)^{-1}\\
    &= \cf{Z\up{1}_{\rho(j)}}(t)\cdot
       \underbrace{\cf{\widehat V\up{1}_{\rho(j)}}(t)\cdot
       \exp\bigl(\tilde g_j(t)\bigr)\cdot
       \cf{\lambda_j\widehat V\up{2}_j}(t)^{-1}
       }_{=:\ \exp(g_j(t))}\\
    &= \cf{Z\up{1}_{\rho(j)}}(t)\cdot\exp\bigl(g_j(t)\bigr),
\end{align}
where we used that the logarithm of the characteristic function of a (possibly
degenerate) Gaussian distribution is a polynomial of degree $\le 2$ and that
such characteristic functions have no zeros.
Hence $g_j$ is a well-defined polynomial.
This proves the main claim.

Now assume in addition that $\rank(A\up{1}) = k \le p$.
Then $A\up{1}$ has the left inverse
\begin{align}
  A\upinv{1} &:= \bigl(A\uT{1}A\up{1}\bigr)^{-1}A\uT{1}.
\end{align}
Multiplying the original representation \cref{eq:two-rep-nonnormal} by
$e_{\rho(j)}\T A\upinv{1}$ from the left gives
\begin{align}
  Z\up{1}_{\rho(j)} + e_{\rho(j)}\T A\upinv{1} E\up{1}
    &= e_{\rho(j)}\T A\upinv{1}\bigl(A\up{1}Z\up{1} + E\up{1}\bigr)\\
    &= e_{\rho(j)}\T A\upinv{1}\bigl(A\up{2}Z\up{2} + E\up{2}\bigr)\\
    &= e_{\rho(j)}\T A\upinv{1}\bigl(A\up{1}P\Lambda Z\up{2} + E\up{2}\bigr)\\
    &= e_{\rho(j)}\T P\Lambda Z\up{2}
       + e_{\rho(j)}\T A\upinv{1}E\up{2}\\
    &= e_j\T \Lambda Z\up{2} + e_{\rho(j)}\T A\upinv{1}E\up{2}\\
    &= \lambda_j Z\up{2}_j + e_{\rho(j)}\T A\upinv{1}E\up{2},
\end{align}
with Gaussian
\begin{align}
  N\up{i}_j &:= e_{\rho(j)}\T A\upinv{1}E\up{i}
    \sim \Normal\bigl(\nu\up{i}_j,(\tau\up{i}_j)^2\bigr),
\end{align}
which are independent of the corresponding $Z\up{i}$.
Taking characteristic functions in
\begin{align}
  Z\up{1}_{\rho(j)} + N\up{1}_j &= \lambda_j Z\up{2}_j + N\up{2}_j
\end{align}
gives
\begin{align}
  \cf{Z\up{1}_{\rho(j)}}(t)\cdot\cf{N\up{1}_j}(t)
    &= \cf{\lambda_j Z\up{2}_j}(t)\cdot\cf{N\up{2}_j}(t).
\end{align}
Now put
\begin{align}
  \sigma_j &:= \sqrt{\Bigl\lvert(\tau\up{2}_j)^2-(\tau\up{1}_j)^2\Bigr\rvert},
  & \nu_j &:= \nu\up{2}_j - \nu\up{1}_j,
  & G_j &\sim \Normal(0,\sigma_j^2).
\end{align}
Treating the two cases $\tau\up{1}_j \ge \tau\up{2}_j$ and
$\tau\up{1}_j \le \tau\up{2}_j$ separately we obtain
\begin{align}
  \cf{\lambda_j Z\up{2}_j}(t)
    &= \cf{Z\up{1}_{\rho(j)}}(t)\cdot\cf{N\up{1}_j}(t)\cdot
       \cf{N\up{2}_j}(t)^{-1}\\
    &= \cf{Z\up{1}_{\rho(j)}}(t)\cdot\cf{G_j-\nu_j}(t),
    & \tau\up{1}_j &\ge \tau\up{2}_j,\\
  \cf{Z\up{1}_{\rho(j)}}(t)
    &= \cf{\lambda_j Z\up{2}_j}(t)\cdot\cf{N\up{2}_j}(t)\cdot
       \cf{N\up{1}_j}(t)^{-1}\\
    &= \cf{\lambda_j Z\up{2}_j}(t)\cdot\cf{G_j+\nu_j}(t),
    & \tau\up{1}_j &\le \tau\up{2}_j .
\end{align}
This shows that we are in one of the two cases
\begin{align}
  \lambda_j Z\up{2}_j
    &\eqd Z\up{1}_{\rho(j)} + G_j - \nu_j,
    & G_j &\indep Z\up{1}_{\rho(j)},
    & \tau\up{1}_j &\ge \tau\up{2}_j,\\
  Z\up{1}_{\rho(j)}
    &\eqd \lambda_j Z\up{2}_j + G_j + \nu_j,
    & G_j &\indep Z\up{2}_j,
    & \tau\up{1}_j &\le \tau\up{2}_j,
\end{align}
which is the claim, once the case distinction is expressed through $g_j$.
For that, compare the two displays above with \cref{eq:cf-kagan}.  Near the
origin $\cf{Z\up{1}_{\rho(j)}}$ does not vanish, so dividing gives
\begin{align}
  \exp\bigl(g_j(t)\bigr)
    &= \frac{\cf{\lambda_jZ\up{2}_j}(t)}{\cf{Z\up{1}_{\rho(j)}}(t)}
     = \begin{cases}
         \cf{G_j-\nu_j}(t), & \tau\up{1}_j\ge\tau\up{2}_j,\\[2pt]
         \cf{G_j+\nu_j}(t)^{-1}, & \tau\up{1}_j\le\tau\up{2}_j,
       \end{cases}
\end{align}
near the origin.  By \cref{not:degenerate} the right hand side equals
$\exp\bigl(\mp\tfrac12\sigma_j^2t^2 - i\nu_jt\bigr)$ in the two cases, so,
applying the uniqueness part of \cref{lem:distinguished-log} to
$\cf{G_j\mp\nu_j}$ and its exponent $\pm g_j$,
\begin{align}
  g_j(t) &\equiv \mp\tfrac12\sigma_j^2t^2 - i\nu_j t
  \quad\bigl(\mathrm{mod}\ 2\pi i\,\Ints\bigr),
  & g_j'' &= \mp\sigma_j^2 .
\end{align}
In particular $\Deg(g_j)\le2$, the coefficient $g_j''$ is real,
$\sigma_j^2=\abs{g_j''}$, and $g_j''\le0$ holds exactly in the first case and
$g_j''\ge0$ exactly in the second.
\end{proof}

\begin{remark}[The four remaining ambiguities]
\label{rem:four-ambiguities}
\Cref{thm:non-gaussian} states that, under the assumption of independent
non-Gaussian sources and full column rank of the mixing matrix, the
distributions of the sources can be recovered up to translation, scale,
permutation and componentwise additive Gaussian noise.  These are the four
ambiguities the theorem leaves.
Each of the first three ambiguities can be removed by a normalisation
convention:
\begin{enumerate}[label=(\arabic*)]
  \item the translation ambiguity by centring all random variables, e.g.\ by
        subtracting their means (assuming these exist);
  \item the scale ambiguity by rescaling, e.g.\ by dividing by the standard
        deviations (assuming these exist and are finite).  Beware that this
        pins the scale down to a sign only when the fourth ambiguity is
        absent: in the setting of \cref{thm:non-gaussian} the first case
        gives $\lambda_j^2\Var(Z\up{2}_j) = \Var(Z\up{1}_{\rho(j)})
        +\sigma_j^2$, so standardising both source vectors leaves
        $\lambda_j = \pm\sqrt{1+\sigma_j^2}$.  In the noiseless model
        $\sigma_j=0$ and only the sign survives, as \cref{cor:complete-ident}
        records with $\Lambda = \diag(\pm1,\dots,\pm1)$;
  \item the permutation ambiguity by enforcing a recognisable ordering of the
        sources, as is done for causal models such as the linear non-Gaussian
        acyclic model (LiNGAM) of \citet{shimizu2006}; there the ordering is
        not merely a convention but is itself identified, see
        \cref{cor:lingam}.
\end{enumerate}
The fourth ambiguity -- componentwise additive Gaussian noise -- is of a
different nature, since it changes the law of the sources rather than just
their parametrisation.  Removing it is the subject of
\cref{sec:gaussian-free}.
\end{remark}

\section{Identifiability for Gaussian-Free Independent Sources}
\label{sec:gaussian-free}

\Cref{thm:non-gaussian} left us with one irreducible ambiguity: the
sources are determined only up to an additive Gaussian perturbation.
That ambiguity is caused by sources that still contain some Gaussian noise
which could equally well be attributed to the noise vector.
In this section we remove it by strengthening the hypothesis on the sources
from ``non-Gaussian'' to ``no Gaussian noise can be split off at all''.

\subsection{Gaussian-free random variables}
\label{ssec:gaussian-free}

\begin{definition}[Gaussian-free]
\label{def:gaussian-free}
A real-valued random variable $Z$ is called \emph{Gaussian-free} if for every
decomposition
\begin{align}
  Z &\eqd U_1 + U_2, & U_1 &\indep U_2,
\end{align}
neither $U_1$ nor $U_2$ is a non-degenerate Gaussian random variable.
Equivalently: there is no $\sigma>0$ and no real-valued random variable $Y$
with $Y \indep G$, $G\sim\Normal(0,\sigma^2)$ and $Z \eqd Y+G$; that is,
$\Law(Z)$ has no non-degenerate Gaussian convolution factor.
\end{definition}

\begin{remark}[Gaussian-freeness is affine invariant]
\label{rem:gf-affine}
Whether $Z$ is Gaussian-free depends only on $\Law(Z)$, and it is invariant
under invertible affine maps: if $Z$ is Gaussian-free, $\lambda\neq0$ and
$c\in\R$, then $\lambda Z+c$ is Gaussian-free as well.
Indeed, a decomposition $\lambda Z+c \eqd U_1+U_2$ with $U_1\indep U_2$ and
$U_2\sim\Normal(m,\tau^2)$, $\tau>0$, would give
$Z \eqd \lambda^{-1}(U_1-c) + \lambda^{-1}U_2$ with
$\lambda^{-1}U_2 \sim \Normal(m/\lambda,\tau^2/\lambda^2)$ again
non-degenerate, contradicting Gaussian-freeness of $Z$.
We use this silently whenever a rescaled source is called Gaussian-free, for
instance in the proof of \cref{thm:gf-identifiability}.
\end{remark}

\begin{remark}[Gaussian-free versus non-Gaussian]
\label{rem:gf-vs-nongaussian}
A \emph{non-constant} Gaussian-free random variable is automatically
non-Gaussian: a non-degenerate Gaussian $Z$ fails to be Gaussian-free, because
the trivial decomposition $Z\eqd Z+0$ already exhibits a non-degenerate
Gaussian factor.
The converse fails, by \cref{ex:gaussian-free}\,(d), and constants are
Gaussian-free but degenerate -- which is why non-constancy has to be assumed
separately.
Consequently, in \cref{thm:gf-identifiability} the Gaussian-free hypothesis
re-derives the non-Gaussianity demanded by \cref{thm:non-gaussian} \emph{for
those components that are assumed Gaussian-free}, and for them only the
non-constancy has to be carried over.  In part~(1) of that theorem, where only
$Z\up{1}$ is assumed Gaussian-free, the non-Gaussianity of $Z\up{2}$ remains
a genuine hypothesis and cannot be dropped: with $p=2$, $A\up{1}=[e_1]$,
$Z\up{1}$ Rademacher and $E\up{1}=e_2N$ for $N\sim\Normal(0,1)$, the same
$X$ also equals $A\up{2}Z\up{2}$ with $A\up{2}=[e_1,e_2]$,
$Z\up{2}=[Z\up{1},N]\T$ and $E\up{2}=0$, and the conclusion fails because
$Z\up{2}_2$ is Gaussian.
\end{remark}

The interpretation is that from a Gaussian-free random variable one cannot
shave off any further Gaussian noise: it carries, in this sense, the cleanest
possible signal.

\begin{remark}[Terminology]
\label{rem:terminology}
The classical name for the property in \cref{def:gaussian-free} is that
$\Law(Z)$ has \emph{no Gaussian component}, or equivalently \emph{no normal
component}, where ``component'' means \emph{convolution factor}.
This is the vocabulary of the arithmetic of probability distributions of
\citet{linnik1977}, whose chapters are devoted to distributions ``with a
Gaussian component'', and it is used in exactly this sense in the ICA
literature \citep{eriksson2006}.
We prefer the adjective \emph{Gaussian-free} for two reasons: it does not
collide with the ``components'' $Z_1,\dots,Z_k$ of a random vector, which we
would otherwise have to call components having no components, and it does not
collide with the ``Gaussian components'' of a Gaussian mixture model, which
are mixture summands rather than convolution factors.

A related but genuinely different classical notion is the Khinchin--Linnik
class $I_0$ of distributions having no \emph{indecomposable} factors.  Here a
law is a \emph{factor} of $\mu$ if $\mu$ is its convolution with some law --
so every law is a factor of itself -- and $\mu$ is \emph{indecomposable} if it
is non-degenerate and in every factorisation $\mu=\mu_1*\mu_2$ one of
$\mu_1,\mu_2$ is degenerate.
The two notions are logically independent: $\Normal(0,1)$ lies in $I_0$ by
Cram\'er's \cref{thm:cramer}, since all its factors are Gaussian and every
non-degenerate Gaussian is decomposable, but it is not Gaussian-free; whereas
a non-degenerate Bernoulli law is Gaussian-free by \cref{cor:gf-criteria}
below yet is itself indecomposable -- a convolution of two laws with $\abs{A}$
and $\abs{B}$ support points has at least $\abs{A}+\abs{B}-1$ of them, so two
non-degenerate factors would force at least three -- and hence is not in
$I_0$.
\end{remark}

To make \cref{def:gaussian-free} quantitative we measure how much Gaussian
noise \emph{can} be split off.

\begin{definition}[Splittable Gaussian scales]
\label{def:sigmax}
For a real-valued random variable $Z$ and $s\in\Rnn$ set
\begin{align}
  \label{eq:gs-def}
  g_s(t) &:= \cf{Z}(t)\exp\Bigl(\tfrac12 s^2t^2\Bigr), \qquad t\in\R,
\end{align}
and put
\begin{align}
  \label{eq:S-set}
  S(Z) &:= \bigl\{\, s \in \Rnn \;:\; g_s \text{ is a characteristic
    function} \,\bigr\},
\end{align}
as well as
\begin{align}
  \sigma_{\max}(Z) &:= \sup S(Z) \in [0,\infty] .
\end{align}
We call $\sigma_{\max}(Z)$ the \emph{maximal Gaussian scale} of $Z$.
\end{definition}

The name is not standard.  It is justified by the following reading: $s\in
S(Z)$ if and only if a $\Normal(0,s^2)$ factor can be split off $Z$, so
$\sigma_{\max}(Z)$ is the supremum of the standard deviations of the
Gaussians that $Z$ contains as convolution factors -- and by
\cref{lem:S-structure}\,(iii) below that supremum is attained, so it really
is the largest such standard deviation.

\begin{lemma}[Structure of $S(Z)$]
\label{lem:S-structure}
Let $Z$ be a real-valued random variable.  Then:
\begin{enumerate}[label=(\roman*)]
  \item for every $s\in S(Z)$ we have the tail bound
        \begin{align}
          \label{eq:gf-tail-bound}
          \abs{\cf{Z}(t)} &\le \exp\Bigl(-\tfrac12 s^2t^2\Bigr),
          & t &\in\R;
        \end{align}
  \item $\sigma_{\max}(Z) < \infty$, always;
  \item $S(Z) = \bigl[0,\sigma_{\max}(Z)\bigr]$; in particular the supremum is
        attained;
  \item $Z$ is Gaussian-free if and only if $\sigma_{\max}(Z) = 0$.
\end{enumerate}
\end{lemma}

\begin{proof}
(i) If $s\in S(Z)$ then $g_s$ is a characteristic function, so
$\abs{g_s(t)}\le 1$ by \cref{prop:cf-elementary}\,(i), which is
\cref{eq:gf-tail-bound}.

(ii) Suppose $S(Z)$ were unbounded and pick $s_n\in S(Z)$ with
$s_n\to\infty$.  Fix $t\neq0$.  By \cref{eq:gf-tail-bound},
$\abs{\cf{Z}(t)}\le\exp(-\tfrac12 s_n^2t^2)\to 0$, so $\cf{Z}(t)=0$.
Thus $\cf{Z}$ vanishes on $\R\setminus\{0\}$ while $\cf{Z}(0)=1$,
contradicting the continuity of $\cf{Z}$
(\cref{prop:cf-elementary}\,(ii)).

(iii) First, $0\in S(Z)$, and $S(Z)$ is downward closed: if $s\in S(Z)$ and
$0\le s'\le s$, then
\begin{align}
  g_{s'}(t) &= g_s(t)\cdot\exp\Bigl(-\tfrac12\bigl(s^2-s'^2\bigr)t^2\Bigr)
\end{align}
is a product of the characteristic function $g_s$ with that of a
$\Normal(0,s^2-s'^2)$ variable, hence a characteristic function by
\cref{eq:cf-product}.
So $S(Z)$ is an interval containing $0$, and it is bounded by (ii).
It remains to see that $\sigma:=\sigma_{\max}(Z)$ itself lies in $S(Z)$.
Choose $s_n\in S(Z)$ with $s_n\uparrow\sigma$.  For every $t\in\R$ we have the
pointwise convergence $g_{s_n}(t)\to g_\sigma(t)$, and $g_\sigma$ is
continuous at $t=0$ because $\cf{Z}$ and $\exp$ are.
By L\'evy's continuity \cref{thm:levy}\,(ii), $g_\sigma$ is a characteristic
function, i.e.\ $\sigma\in S(Z)$.

(iv) If $\sigma_{\max}(Z)=\sigma>0$ then, writing $\cf{Z}(t) =
g_\sigma(t)\exp(-\tfrac12\sigma^2t^2)$ with the characteristic function
$g_\sigma$ from (iii), we exhibit a non-degenerate Gaussian convolution factor
of $\Law(Z)$, so $Z$ is not Gaussian-free.
Conversely, if $Z$ is not Gaussian-free, say $Z \eqd U+V$ with $U \indep V$
and $V\sim\Normal(\mu,\tau^2)$, $\tau>0$, then
\begin{align}
  \cf{Z}(t)\exp\Bigl(\tfrac12\tau^2t^2\Bigr)
    &= \cf{U}(t)\,\exp\bigl(i\mu t\bigr) = \cf{U+\mu}(t)
\end{align}
is a characteristic function, so $\tau\in S(Z)$ and
$\sigma_{\max}(Z)\ge\tau>0$.
\end{proof}

\Cref{lem:S-structure} puts us in a position to prove the following, which
turns the informal wish of the introduction to this section into a theorem:
\emph{every} real-valued random variable splits into a Gaussian-free signal
and independent Gaussian noise, and it does so in essentially one way.

\begin{theorem}[Gaussian splitting]
\label{thm:gauss-split}
Let $Z$ be a real-valued random variable and put
$\sigma := \sigma_{\max}(Z) \in [0,\infty)$.
Then there exist a Gaussian-free random variable $Y$ and a random variable
$G\sim\Normal(0,\sigma^2)$ with $Y\indep G$ such that
\begin{align}
  \label{eq:gauss-split}
  Z &\eqd Y + G .
\end{align}
The decomposition is maximal and unique up to a translation: if
$Z \eqd Y'+G'$ with $Y'\indep G'$, $Y'$ Gaussian-free and
$G'\sim\Normal(m,\tau^2)$, then $\tau = \sigma$ and $Y'\eqd Y-m$.
\end{theorem}

\begin{proof}
By \cref{lem:S-structure}\,(iii) we have $\sigma\in S(Z)$, so that
$t\mapsto g_\sigma(t) = \cf{Z}(t)\exp(\tfrac12\sigma^2t^2)$ is the
characteristic function of some random variable $Y$, and
\begin{align}
  \cf{Z}(t) &= g_\sigma(t)\cdot\exp\Bigl(-\tfrac12\sigma^2t^2\Bigr)
             = \cf{Y}(t)\cdot\cf{G}(t)
\end{align}
for $G\sim\Normal(0,\sigma^2)$ chosen independent of $Y$; this is
\cref{eq:gauss-split} by \cref{eq:cf-product} and
\cref{thm:uniqueness}\,(i).  We allow the degenerate case $\sigma=0$, where
$G=0$ and $Y\eqd Z$.

To see that $Y$ is Gaussian-free, suppose $Y \eqd U+V$ with $U\indep V$ and
$V\sim\Normal(\mu,\tau^2)$, $\tau>0$.  Then, as in the proof of
\cref{lem:S-structure}\,(iv), $\cf{Y}(t)\exp(\tfrac12\tau^2t^2)$ is a
characteristic function, hence so is
\begin{align}
  \cf{Z}(t)\exp\Bigl(\tfrac12\bigl(\sigma^2+\tau^2\bigr)t^2\Bigr)
    &= \cf{Y}(t)\exp\Bigl(\tfrac12\tau^2t^2\Bigr),
\end{align}
so that $\bar s := \sqrt{\sigma^2+\tau^2}\in S(Z)$.  But then
\begin{align}
  \sigma_{\max}(Z) = \sigma < \sqrt{\sigma^2+\tau^2} = \bar s \in S(Z),
\end{align}
contradicting the definition of $\sigma_{\max}$.

For the uniqueness statement, let $Z\eqd Y'+G'$ be as stated.
Then $\cf{Z}(t)\exp(\tfrac12\tau^2t^2) = \cf{Y'}(t)\exp(i m t)$ is a
characteristic function, so $\tau\in S(Z)$ and hence $\tau\le\sigma$ by
\cref{lem:S-structure}\,(iii).
If we had $\tau<\sigma$, then
\begin{align}
  \cf{Y'}(t)\exp\Bigl(\tfrac12\bigl(\sigma^2-\tau^2\bigr)t^2\Bigr)
    &= \cf{Z}(t)\exp\Bigl(\tfrac12\sigma^2t^2\Bigr)\exp(-imt)
     = \cf{Y-m}(t)
\end{align}
would be a characteristic function, exhibiting a
$\Normal(0,\sigma^2-\tau^2)$ factor of $\Law(Y')$ and contradicting that $Y'$
is Gaussian-free.  So $\tau=\sigma$, and then
$\cf{Y'}(t) = \cf{Z}(t)\exp(\tfrac12\sigma^2t^2)\exp(-imt) = \cf{Y-m}(t)$.
\end{proof}

\begin{remark}[A dichotomy, not a trichotomy]
\label{rem:dichotomy}
For every real-valued random variable $Z$ exactly one of the following holds.
\begin{enumerate}[label=(\arabic*)]
  \item $\sigma_{\max}(Z)=0$.  Equivalently, $Z$ is Gaussian-free.  This is the
        degenerate case of \cref{thm:gauss-split}, with $G=0$ and $Y\eqd Z$.
  \item $0<\sigma_{\max}(Z)<\infty$.  Here \cref{thm:gauss-split} gives a
        proper decomposition $Z \eqd Y+G$ with a non-degenerate Gaussian $G$, a
        Gaussian-free $Y$, and $Y\indep G$.
\end{enumerate}
One might expect a third case, in which arbitrarily large Gaussian factors can
be shaved off; \cref{lem:S-structure}\,(ii) shows that this cannot happen.
The reason is \cref{eq:gf-tail-bound}: splitting off a $\Normal(0,s^2)$ factor
forces the characteristic function to decay at least as fast as
$\exp(-\tfrac12 s^2t^2)$, and no characteristic function can decay faster than
every Gaussian without vanishing identically off the origin.
\end{remark}

The tail bound \cref{eq:gf-tail-bound} is also the most convenient practical
criterion, alongside a support argument.

\begin{corollary}[Two sufficient criteria]
\label{cor:gf-criteria}
Let $Z$ be a real-valued random variable.  Each of the following implies that
$Z$ is Gaussian-free.
\begin{enumerate}[label=(\roman*)]
  \item $\supp\Law(Z) \neq \R$; this holds in particular if $Z$ is bounded, or
        bounded from one side.
  \item $\exp(\tfrac12 s^2t^2)\abs{\cf{Z}(t)}$ is unbounded in $t$ for every
        $s>0$; in particular, this holds whenever
        $\liminf_{\abs{t}\to\infty}
          \bigl(-\log\abs{\cf{Z}(t)}\bigr)\big/t^2 = 0$.
\end{enumerate}
\end{corollary}

\begin{proof}
(i) The support of a convolution of two laws is the closure of the sum of
their supports.  A non-degenerate Gaussian law has support $\R$, so any law with
a non-degenerate Gaussian convolution factor has support $\R$.
(ii) is the contrapositive of \cref{eq:gf-tail-bound}, together with
\cref{lem:S-structure}\,(iv).
\end{proof}

\begin{example}[Which laws are Gaussian-free?]
\label{ex:gaussian-free}
\begin{enumerate}[label=(\alph*)]
  \item By \cref{cor:gf-criteria}\,(i), every law whose support is not all of
        $\R$ is Gaussian-free: the uniform, Rademacher and three-point laws
        of \cref{tab:examples}, every Bernoulli law, and also the
        exponential, Gamma, Poisson and $\chi^2$ laws, which are supported on
        a half line.
  \item By \cref{cor:gf-criteria}\,(ii), every law whose characteristic
        function decays more slowly than any Gaussian is Gaussian-free.  This
        covers the Laplace law, with $\cf{Z}(t) = (1+b^2t^2)^{-1}$ decaying
        polynomially; the Cauchy law, with $\cf{Z}(t)=\exp(-\gamma\abs{t})$;
        and the Student $t_\nu$ laws, whose characteristic functions decay
        exponentially.  So the standard super-Gaussian sources of ICA are all
        Gaussian-free.
  \item A non-degenerate Gaussian law is of course not Gaussian-free, with
        $\sigma_{\max}\bigl(\Normal(\mu,\sigma^2)\bigr)=\sigma$.
  \item The instructive example is $Z := R+G$ with $R$ Rademacher,
        $G\sim\Normal(0,1)$ and $R\indep G$, i.e.\ the Gaussian mixture
        $\tfrac12\Normal(-1,1)+\tfrac12\Normal(1,1)$.
        Here $\cf{Z}(t)=\cos(t)\exp(-\tfrac12t^2)$, so
        $\cf{Z}(t)\exp(\tfrac12 s^2t^2) = \cos(t)\exp(\tfrac12(s^2-1)t^2)$ is
        unbounded for $s>1$ and hence $\sigma_{\max}(Z)=1$, with Gaussian-free
        part $R$ -- determined, as \cref{thm:gauss-split} says, only up to a
        translation.
        By \cref{rem:marcinkiewicz-use}\,(ii) the variable $Z$ is
        \emph{non-Gaussian}, so it is an admissible source for
        \cref{thm:non-gaussian}, but it is \emph{not} Gaussian-free and
        therefore not admissible for \cref{thm:gf-identifiability} below.
        This shows that the hypothesis of this section is strictly stronger
        than that of \cref{sec:non-gaussian}.
\end{enumerate}
\end{example}

\begin{remark}[The infinitely divisible case]
\label{rem:infinitely-divisible}
For readers familiar with the L\'evy--Khintchine representation
\citep[Chapter~7]{kallenberg2021}, there is a clean characterisation.
Let $Z$ be infinitely divisible with triplet $(a,\nu,\gamma)$, i.e.
\begin{align}
  \cf{Z}(t) &= \exp\left( i\gamma t - \tfrac12 a t^2
    + \int_{\R}\Bigl(e^{itx}-1-itx\mathbf{1}_{\{\abs{x}\le1\}}\Bigr)
      \nu(dx)\right),
\end{align}
with $a\ge0$ and a L\'evy measure $\nu$ satisfying
$\int\min(1,x^2)\,\nu(dx)<\infty$.
Then $Z$ is Gaussian-free if and only if $a=0$, and in fact
$\sigma_{\max}(Z)^2 = a$.

Indeed, if $a>0$ then splitting the factor $\exp(-\tfrac12 a t^2)$ off
$\cf{Z}$ leaves the characteristic function of the infinitely divisible law
with triplet $(0,\nu,\gamma)$, so $\sqrt{a}\in S(Z)$.
Conversely, assume $a=0$ and let $s\in S(Z)$.  Taking moduli,
\begin{align}
  -\log\abs{\cf{Z}(t)} &= \int_{\R}\bigl(1-\cos(tx)\bigr)\,\nu(dx),
\end{align}
and this is $o(t^2)$ as $\abs{t}\to\infty$: on $\{\abs{x}\le1\}$ we have
$(1-\cos(tx))/t^2 \le x^2/2$, which is $\nu$-integrable there and tends to $0$
pointwise, so dominated convergence applies; on $\{\abs{x}>1\}$ the integral
is bounded by $2\nu(\abs{x}>1)<\infty$.
On the other hand \cref{eq:gf-tail-bound} gives
$-\log\abs{\cf{Z}(t)} \ge \tfrac12 s^2t^2$.
Comparing the two forces $s=0$, so $\sigma_{\max}(Z)=0$.

The same comparison gives the identity $\sigma_{\max}(Z)^2 = a$ for general
$a\ge0$, not just the two implications.  Indeed, without assuming $a=0$,
\begin{align}
  -\log\abs{\cf{Z}(t)}
    &= \tfrac12 a t^2 + \int_{\R}\bigl(1-\cos(tx)\bigr)\,\nu(dx)
     = \tfrac12 a t^2 + o\bigl(t^2\bigr),
  \qquad \abs{t}\to\infty,
\end{align}
by the estimate just given, while \cref{eq:gf-tail-bound} bounds the left
hand side below by $\tfrac12s^2t^2$ for every $s\in S(Z)$.  Dividing by $t^2$
and letting $\abs{t}\to\infty$ yields $s^2\le a$, hence
$\sigma_{\max}(Z)^2\le a$; combined with $\sqrt a\in S(Z)$ this is
equality.
Note that this argument never assumes that the cofactor is itself infinitely
divisible -- which matters, since factors of infinitely divisible laws need
not be infinitely divisible.
\end{remark}

\subsection{Identifiability in the presence of additive Gaussian noise}
\label{ssec:gf-identifiability}

In the following we explicitly assume that the sources of our ICA model have
been separated according to \cref{thm:gauss-split}, so that they are mutually
independent and Gaussian-free.
Note that Gaussian random variables are considered noise and are
usually not properly identifiable in ICA models anyway, cf.\
\cref{rem:gaussian-not-identifiable}.
Assuming that the proper signals in an ICA model are Gaussian-free is
therefore a reasonable assumption, even though it is more restrictive than the
usual assumption of mere non-normality -- and by
\cref{ex:gaussian-free}\,(b) it is satisfied by all the standard
super-Gaussian source models.
The reward for this stronger assumption is a stronger identifiability result:
it identifies the sources inside a mixture even in the presence of additive
\emph{and dependent} Gaussian noise.

\begin{theorem}[Identifiability for independent Gaussian-free sources with
additive Gaussian noise]
\label{thm:gf-identifiability}
Let $X \in \R^p$ be a random vector and assume that we have two
representations of $X$ satisfying all the assumptions of
\cref{thm:non-gaussian}:
\begin{align}
  A\up{1}Z\up{1} + E\up{1} \;=\; X \;=\; A\up{2}Z\up{2} + E\up{2}.
\end{align}
Assume furthermore that $A\up{1}$ (equivalently, by
\cref{thm:non-gaussian}, $A\up{2}$) has full column rank,
$\rank(A\up{1}) = k \le p$.
\begin{enumerate}[label=(\arabic*)]
  \item If all components of $Z\up{1}$ are Gaussian-free, then there exist a
        permutation matrix $P \in \R^{k\times k}$, an invertible diagonal
        matrix $\Lambda \in \R^{k\times k}$ and a (possibly degenerate) Gaussian
        random vector $H \sim \Normal(\xi,\Delta)$ with mean vector
        $\xi \in \R^k$ and diagonal, possibly degenerate, covariance matrix
        $\Delta \in \R^{k\times k}$, independent of $Z\up{1}$ and of
        $E\up{2}$, such that
        \begin{align}
          A\up{2} &= A\up{1}P\Lambda, \\
          Z\up{2} &\eqd \Lambda^{-1}P^{-1}\bigl(Z\up{1} + H\bigr),
            & H &\indep \bigl(Z\up{1},E\up{2}\bigr),\\
          E\up{1} &\eqd A\up{1}H + E\up{2},\\
          \mu\up{1} &= A\up{1}\xi + \mu\up{2},\\
          \Sigma\up{1} &= A\up{1}\Delta A\uT{1} + \Sigma\up{2}.
        \end{align}
  \item If, in addition, all components of $Z\up{2}$ are Gaussian-free as
        well, then $H = \xi$ and $\Delta = 0$, and we further get
        \begin{align}
          Z\up{2} &\eqd \Lambda^{-1}P^{-1}\bigl(Z\up{1} + \xi\bigr),
          & E\up{2} &\eqd E\up{1} - A\up{1}\xi,
          & \Sigma\up{2} &= \Sigma\up{1}.
        \end{align}
\end{enumerate}
\end{theorem}

\begin{proof}
\emph{Part (1).}
By \cref{thm:non-gaussian} we already have $A\up{2} = A\up{1}P\Lambda$,
and for each $j = 1,\dots,k$ separately one of the two cases
\begin{align}
  \bigl(\lambda_j Z\up{2}_j + \nu_j\bigr)
    &\eqd Z\up{1}_{\rho(j)} + G_j,
    & G_j &\indep Z\up{1}_{\rho(j)},\\
  Z\up{1}_{\rho(j)}
    &\eqd \bigl(\lambda_j Z\up{2}_j + \nu_j\bigr) + G_j,
    & G_j &\indep Z\up{2}_j,
\end{align}
with some $\nu_j \in \R$ and $G_j \sim \Normal(0,\sigma_j^2)$.
Since $Z\up{1}_{\rho(j)}$ is Gaussian-free, the second case exhibits a Gaussian
convolution factor of $\Law(Z\up{1}_{\rho(j)})$ and can therefore only occur
with $\sigma_j^2 = 0$, i.e.\ $G_j = 0$; but then it is subsumed by the first
case.
We may therefore assume the first case for all $j = 1,\dots,k$, which we
rearrange into
\begin{align}
  e_j\T Z\up{2} = Z\up{2}_j
    &\eqd \lambda_j^{-1}Z\up{1}_{\rho(j)} - \lambda_j^{-1}\nu_j
          + \lambda_j^{-1}G_j,
    & G_j &\indep Z\up{1}_{\rho(j)},\\
    &= e_j\T\Lambda^{-1}P^{-1}Z\up{1} - e_j\T\Lambda^{-1}\nu
       + e_j\T\Lambda^{-1}G,\\
    &= e_j\T\Lambda^{-1}P^{-1}\bigl(Z\up{1} - P\nu + PG\bigr),\\
    &= e_j\T\Lambda^{-1}P^{-1}\bigl(Z\up{1} + H\bigr),
\end{align}
where we used the abbreviations
\begin{align}
  \nu &:= [\nu_1,\dots,\nu_k]\T, & \xi &:= -P\nu, \\
  G &:= [G_1,\dots,G_k]\T \sim \Normal(0,\Gamma),
    & \Gamma &:= \diag(\sigma_1^2,\dots,\sigma_k^2),\\
  H &:= \xi + PG \sim \Normal(\xi,\Delta),
    & \Delta &:= P\Gamma P\T .
\end{align}
Note that the $G_j$ were constructed componentwise, so we may and do take $G$
to have independent components; this implies that $\Delta = P\Gamma P\T$ is
again a diagonal matrix, and hence that $H$ has independent components as
well.
Together with the fact that the components of $Z\up{1}$ and of $Z\up{2}$ are
each mutually independent, we may gather all components into the single
distributional equation
\begin{align}
  \label{eq:Z2-vs-Z1}
  Z\up{2} &\eqd \Lambda^{-1}P^{-1}\bigl(Z\up{1} + H\bigr),
  & H &\indep Z\up{1}.
\end{align}
By the continuous mapping theorem we may multiply \cref{eq:Z2-vs-Z1} by
$A\up{2} = A\up{1}P\Lambda$ and obtain
\begin{align}
  A\up{2}Z\up{2} &\eqd A\up{1}Z\up{1} + A\up{1}H, & H &\indep Z\up{1},
\end{align}
which in terms of characteristic functions reads
\begin{align}
  \cf{A\up{2}Z\up{2}}(t) &= \cf{A\up{1}Z\up{1}}(t)\cdot\cf{A\up{1}H}(t).
\end{align}
Plugging this into the characteristic functions of the original model
equation,
\begin{align}
  \cf{A\up{2}Z\up{2}}(t)\cdot\cf{E\up{2}}(t)
    &= \cf{A\up{1}Z\up{1}}(t)\cdot\cf{E\up{1}}(t),
\end{align}
we obtain an identity that we would like to divide by
$\cf{A\up{1}Z\up{1}}$.
Here we may not simply cancel the factor $\cf{A\up{1}Z\up{1}}$, since a
characteristic function may well have zeros: for $p=k=1$, $A\up{1}=1$ and
$Z\up{1}\sim\mathrm{Unif}[-1,1]$ -- an admissible Gaussian-free, non-Gaussian
source -- one has $\cf{Z\up{1}}(t) = \sin(t)/t$, which vanishes at $t=\pi$.
We argue locally instead.  By continuity and $\cf{A\up{1}Z\up{1}}(0)=1$ there
is a $\delta>0$ with $\cf{A\up{1}Z\up{1}}\neq0$ on the ball
$B_\delta(0)\subseteq\R^p$, and cancelling there gives
\begin{align}
  \label{eq:gf-local-cancel}
  \cf{A\up{1}H}(t)\cdot\cf{E\up{2}}(t) &= \cf{E\up{1}}(t),
  & t &\in B_\delta(0).
\end{align}
Both sides of \cref{eq:gf-local-cancel} are characteristic functions of
(possibly degenerate) multivariate Gaussian distributions, hence of the form
$\exp(q(t))$ with polynomials $q$ of degree $\le 2$ and $q(0)=0$; here we may
in addition take $H$ independent of $E\up{2}$: all assertions here are
statements about laws, so we may enlarge the underlying probability space
(replacing $\Omega$ by $\Omega\times[0,1]$) and realise an $H$ with law
$\Normal(\xi,\Delta)$ independent of the pair $(Z\up{1},E\up{2})$.  Writing the two sides as
$\exp(q_1)$ and $\exp(q_2)$, we get $\exp(q_1-q_2)\equiv1$ on
$B_\delta(0)$, so the continuous function $(q_1-q_2)/(2\pi i)$ takes values in
$\Ints$ on the connected set $B_\delta(0)$ and vanishes at the origin; hence
$q_1 = q_2$ on $B_\delta(0)$.  Two polynomials that agree on a non-empty open
subset of $\R^p$ agree everywhere, so $q_1=q_2$ identically.  Hence
\cref{eq:gf-local-cancel} in fact holds for every $t\in\R^p$, and
\cref{thm:uniqueness} gives
\begin{align}
  A\up{1}H + E\up{2} &\eqd E\up{1},
  & H &\indep \bigl(Z\up{1},E\up{2}\bigr),
\end{align}
and comparing mean vectors and covariance matrices yields
\begin{align}
  A\up{1}\xi + \mu\up{2} &= \mu\up{1},\\
  A\up{1}\Delta A\uT{1} + \Sigma\up{2} &= \Sigma\up{1}.
\end{align}
This proves all claims of part (1).

\emph{Part (2).}
If the components $Z\up{2}_j$ are Gaussian-free as well, then so are the
affinely transformed variables $\lambda_j Z\up{2}_j + \nu_j$ by
\cref{rem:gf-affine}, and the only possibility left is
\begin{align}
  \bigl(\lambda_j Z\up{2}_j + \nu_j\bigr)
    &\eqd Z\up{1}_{\rho(j)} + G_j,
    & G_j &\indep Z\up{1}_{\rho(j)},
\end{align}
with $\sigma_j^2 = 0$ and thus $G_j = 0$.
This shows $H = \xi$ and $\Delta = 0$, and everything else follows from part
(1).
\end{proof}

\begin{remark}[What \cref{thm:gf-identifiability} buys us]
\label{rem:final-discussion}
Compared with \cref{thm:non-gaussian}, the stronger assumption on the sources
removes the fourth ambiguity of \cref{rem:four-ambiguities} -- but only
part~(2) removes it entirely.  Part~(1) makes it one-sided: any residual
Gaussian must sit in $Z\up{2}$ and is exactly compensated in $E\up{2}$, which
is what the vector $H$ records.  It is not vacuous.  Take $p=k=1$,
$A\up{1}=A\up{2}=[1]$, $Z\up{1}$ Rademacher (Gaussian-free) with
$E\up{1}\sim\Normal(0,1)$, against $Z\up{2} = Z\up{1}+E\up{1}$ (non-constant
and, by \cref{rem:marcinkiewicz-use}\,(ii), non-Gaussian) with $E\up{2}=0$:
both representations satisfy every hypothesis of part~(1), yet $\Law(Z\up{2})$
is absolutely continuous while $\Law(Z\up{1})$ has two atoms, so no affine map
carries one to the other.
Three aspects deserve emphasis.
\begin{enumerate}[label=(\roman*)]
  \item The Gaussian noise vectors $E\up{i}$ are \emph{not} assumed to have
        independent components, nor a non-degenerate covariance matrix.  They
        may be arbitrarily dependent across the $p$ observed coordinates, and
        they may be partially deterministic.  Nevertheless, under the
        hypothesis of part~(2), the sources are pinned down up to permutation,
        scale and translation.
  \item Part~(2) is a genuine two-sided statement: as soon as \emph{both}
        candidate source vectors are Gaussian-free, not only the sources but
        also the noise laws agree, $\Sigma\up{1} = \Sigma\up{2}$, so the whole
        model $(A,\Law(Z),\Law(E))$ is identified up to the unavoidable
        permutation, scaling and shift.
  \item The Gaussian-free assumption is not a restriction on which
        \emph{observations} can be modelled, only on how the model is
        parametrised.  Indeed, by \cref{thm:gauss-split} every source splits
        as $Z_j \eqd Y_j + G_j$ with $Y_j$ Gaussian-free and
        $G_j\sim\Normal(0,\sigma_j^2)$ independent of $Y_j$, and moving the
        Gaussian part into the noise is the trivial direction of
        \cref{lem:noise-trade}: writing $a_j$ for the $j$-th column of $A$,
        \begin{align}
          a_j Z_j &\eqd a_j Y_j + a_j G_j ,
          & a_jG_j &\sim \Normal\bigl(0,\sigma_j^2 a_ja_j\T\bigr)
                    \indep Y_j ,
        \end{align}
        so that $a_jG_j$ may simply be added to $E$ -- which is why we allow
        $E$ to have an arbitrary, possibly degenerate covariance matrix.
        Since the $Z_j$ are mutually independent, the pairs $(Y_j,G_j)$ may
        be taken jointly independent, so all $k$ Gaussian parts can be moved
        at once.
        (\Cref{lem:noise-trade} proves the harder converse direction, moving
        Gaussian noise out of $E$ and into extra columns of $A$.)
        One caveat: if $Z_j$ was itself Gaussian, then its Gaussian-free part
        $Y_j$ is a \emph{constant}, and the corresponding column must then be
        deleted and absorbed into the offset $\mu$, exactly as in Step~1 of
        \cref{rem:normalisation}, before \cref{thm:gf-identifiability} can be
        applied to the reduced model.  What the assumption buys is that this
        splitting and pruning has been carried out, so that no Gaussian mass
        is left sitting ambiguously between $Z$ and $E$.
\end{enumerate}
In practice, part~(1) is the statement one applies when a candidate solution
is compared against the ground truth, and part~(2) is the statement one
applies when both are outputs of a procedure that is guaranteed to return
Gaussian-free sources.
\end{remark}

\section[Complete Noiseless ICA: Estimation by Equivariant Gradient Descent]
  {Complete Noiseless ICA:\texorpdfstring{\\}{ }
   Estimation by Equivariant Gradient Descent}
\label{sec:algorithm}

The previous sections answered the question of \emph{what} is identifiable.
This section answers the complementary question of \emph{how} one estimates
it, in the special case that is by far the most common in applications: a
square invertible mixing matrix and no noise.
We first specialise the identifiability theory to that case
(\cref{ssec:complete-ident}), then set up the maximum likelihood objective
(\cref{ssec:ml}), then explain what the customary ``preconditioner'' in the
gradient step really is (\cref{ssec:relative-gradient}), and finally analyse
when the resulting algorithm converges to a separating solution
(\cref{ssec:stability}).

\begin{example}[The cocktail party problem]
\label{ex:cocktail}
Two loudspeakers in a room play different pieces of music, modelled as
real-valued signals $Z_1(t)$ and $Z_2(t)$ indexed by time $t$.
Two microphones at different positions each record a mixture of the two,
\begin{align}
  X_1(t) &= a_{11}Z_1(t) + a_{12}Z_2(t), &
  X_2(t) &= a_{21}Z_1(t) + a_{22}Z_2(t) .
\end{align}
The task is to recover the two pieces of music $Z_j(t)$ from the two recorded
mixtures $X_i(t)$.
What makes this more than a matrix inversion is that \emph{both} the mixing
matrix $A = (a_{ij})$ and the sources $Z_j$ are unknown; all that is assumed
is that the two sources are statistically independent.
\Cref{cor:complete-ident} below says that this is enough -- provided at most
one source is Gaussian.
\end{example}

\subsection{The complete noiseless model and its identifiability}
\label{ssec:complete-ident}

\begin{assumption}[Complete noiseless ICA model]
\label{asm:ica}
We observe $X(1),\dots,X(T)$ in $\R^p$, all with the same law, and assume:
\begin{enumerate}[label=(\roman*)]
  \item \emph{Independent sources.}  There are random vectors
        $Z(t) = [Z_1(t),\dots,Z_k(t)]\T$ in $\R^k$ whose components
        $\{Z_1(t),\dots,Z_k(t)\}$ are mutually independent for each fixed
        $t$, and almost surely non-constant.  Nothing is assumed about the
        dependence \emph{across} $t$.
  \item \emph{Linear noiseless mixing.}  $X(t) = A\,Z(t)$ for all $t$, with a
        matrix $A\in\R^{p\times k}$ that does not depend on $t$.
  \item \emph{Completeness.}  $p = k$ and $A$ is invertible.  We write
        $W := A^{-1}$ and denote the rows of $W$ by $w_1\T,\dots,w_k\T$.
  \item \emph{Non-Gaussianity.}  At most one of the components
        $Z_1(t),\dots,Z_k(t)$ is Gaussian.
\end{enumerate}
For the estimation part we shall in addition normalise:
\begin{enumerate}[label=(\roman*),start=5]
  \item \emph{Centring and scale.}  $\Ex[Z_j(t)]=0$ for all $j$ and $t$, and
        the scale of each $Z_j$ is fixed either by $\Var(Z_j)=1$ or by the
        fixed-point normalisation \cref{eq:scale-normalisation}; see
        \cref{rem:ica-normalisation}.
\end{enumerate}
Since the law of $X(t)$ does not depend on $t$, we drop the argument and write
$X = AZ$ for a generic sample whenever only the common law matters.
\end{assumption}

\begin{remark}[The two scale conventions]
\label{rem:ica-normalisation}
Rescaling the sources is without loss of generality: it is the
reparametrisation $X = (AD)\bigl(D^{-1}Z\bigr)$ with $D$ diagonal and
invertible, which preserves \cref{asm:ica}\,(i)--(iv).  Centring them,
however, requires centring the data as well, since \cref{asm:ica}\,(ii)
carries no offset; that is step~(1) of \cref{rem:preprocessing}.
Two scale conventions appear below, and they are \emph{not} the same.
The first is $\Var(Z_j)=1$, used wherever only second-order quantities matter
(\cref{rem:preprocessing}); it requires the extra assumption
$\Ex[Z_j^2]<\infty$, which \cref{asm:ica} does not otherwise impose.
The second is the fixed-point scale of \cref{lem:fixed-point-scale}, which is
the one \cref{thm:stability} needs and which is generically \emph{not} the
unit-variance scale.
\end{remark}

The identifiability statement usually quoted in the data-science literature is
the following.  It is exactly the specialisation of \cref{thm:kagan} to a
square invertible mixing matrix, with the ``at most one Gaussian'' clause
handled by an extra argument at the end of the proof.

\begin{corollary}[Identifiability of the complete noiseless model;
{\citealp[Theorem~11]{comon1994}}]
\label{cor:complete-ident}
Let $X$ be a random vector in $\R^k$ with two representations
\begin{align}
  \label{eq:two-rep-complete}
  A\up{1}Z\up{1} + \mu\up{1} \;=\; X \;=\; A\up{2}Z\up{2} + \mu\up{2},
\end{align}
where for $i = 1,2$:
\begin{enumerate}[label=(\roman*)]
  \item $A\up{i}\in\R^{k\times k}$ is invertible and $\mu\up{i}\in\R^k$;
  \item the components of $Z\up{i}$ are mutually independent and almost surely
        non-constant;
\end{enumerate}
and, for $i=1$ only,
\begin{enumerate}[label=(\roman*),start=3]
  \item at most one component of $Z\up{1}$ is Gaussian.
\end{enumerate}
Then there exist a permutation matrix $P\in\R^{k\times k}$, an invertible
diagonal matrix $\Lambda\in\R^{k\times k}$ and a vector $c\in\R^k$ such that
\begin{align}
  A\up{2} &= A\up{1}P\Lambda,
  & Z\up{1} &= P\Lambda\,Z\up{2} + c \quad\text{almost surely}.
\end{align}
If moreover both source vectors are centred with unit variances, then
$\Lambda = \diag(\pm1,\dots,\pm1)$ and $c = 0$: the sources are determined up
to permutation and sign, and the columns of the mixing matrix up to
permutation and sign.
\end{corollary}

\begin{proof}
An invertible matrix has non-zero, pairwise non-proportional columns, so both
representations satisfy the hypotheses of \cref{thm:kagan}.

\emph{Step 1: matching the non-Gaussian columns.}
Let $N_i := \{j : Z\up{i}_j \text{ is non-Gaussian}\}$.
Let $l\in N_2$.  By \cref{thm:kagan}\,(1) the column $a\up{2}_l$ must be
proportional to some column $a\up{1}_{\rho(l)}$ of $A\up{1}$, and by
\cref{thm:kagan}\,(2) the source $Z\up{1}_{\rho(l)}$ is then non-Gaussian,
i.e.\ $\rho(l)\in N_1$.
The index $\rho(l)$ is unique because the columns of $A\up{1}$ are pairwise
non-proportional, and the map $\rho\colon N_2\to N_1$ is injective: if
$\rho(l)=\rho(l')$ then $a\up{2}_l \propto a\up{1}_{\rho(l)} \propto
a\up{2}_{l'}$, contradicting pairwise non-proportionality of the columns of
$A\up{2}$.
Hence $\abs{N_2}\le\abs{N_1}$, and by symmetry $\abs{N_1} = \abs{N_2} =: n$,
with $\rho\colon N_2\to N_1$ a bijection.  In particular the two
representations have Gaussian sources in equal number, $k-n$ each.  This is
why hypothesis (iii) needs to be imposed on one of the two representations
only: assuming it for $i=1$ gives $n\ge k-1$, and hence at most one Gaussian
component of $Z\up{2}$ as well.  In either case $n\in\{k-1,k\}$.

\emph{Step 2: reduction to a single matrix.}
Put $C := (A\up{1})^{-1}A\up{2}$, which is invertible, so that
$A\up{2} = A\up{1}C$; it therefore suffices to prove $C = P\Lambda$.
Multiplying \cref{eq:two-rep-complete} by $(A\up{1})^{-1}$ gives, almost
surely,
\begin{align}
  \label{eq:C-relation}
  Z\up{1} &= C\,Z\up{2} + d,
  & d &:= \bigl(A\up{1}\bigr)^{-1}\bigl(\mu\up{2}-\mu\up{1}\bigr),
\end{align}
so a factorisation $C = P\Lambda$ gives the second assertion as well, with
$c := d$.
For $l\in N_2$ we have $a\up{2}_l = \lambda_l a\up{1}_{\rho(l)}$ with
$\lambda_l\neq0$, hence $Ce_l = \lambda_l e_{\rho(l)}$: every column of $C$
indexed by $N_2$ is a non-zero multiple of a standard basis vector.

\emph{Step 3: the case $n=k$.}
Here $N_2 = \{1,\dots,k\}$ and $\rho$ is a permutation, so \emph{every}
column of $C$ is of that form and $C = P(\rho)\Lambda$ with
$\Lambda = \diag(\lambda_1,\dots,\lambda_k)$ invertible.

\emph{Step 4: the case $n=k-1$.}
Let $l_0\notin N_2$ and $k_0\notin N_1$ be the two exceptional indices, so
that $Z\up{2}_{l_0}$ and $Z\up{1}_{k_0}$ are Gaussian, and non-degenerate
because they are non-constant.
By Step 2 every column of $C$ except the $l_0$-th is a non-zero multiple of a
standard basis vector, the corresponding indices $\rho(l)$ exhausting
$\{1,\dots,k\}\setminus\{k_0\}$.
Reading off rows, \cref{eq:C-relation} becomes
\begin{align}
  Z\up{1}_{k_0} &= \beta\,Z\up{2}_{l_0} + d_{k_0},
  & \beta &:= C_{k_0l_0},\\
  Z\up{1}_{j} &= \lambda_{\rho^{-1}(j)}Z\up{2}_{\rho^{-1}(j)}
     + \beta_j\,Z\up{2}_{l_0} + d_j,
  & \beta_j &:= C_{jl_0}, \qquad j \neq k_0 .
\end{align}
Invertibility of $C$ forces $\beta\neq0$.
Suppose, for contradiction, that $\beta_j\neq0$ for some $j\neq k_0$.
Abbreviate $U := Z\up{2}_{\rho^{-1}(j)}$, $\lambda := \lambda_{\rho^{-1}(j)}
\neq 0$ and $G := Z\up{2}_{l_0}\sim\Normal(m,\tau^2)$ with $\tau>0$; note
$U\indep G$.
The components $Z\up{1}_j$ and $Z\up{1}_{k_0}$ are independent, so for all
$s,t\in\R$,
\begin{align}
  \cf{U}(s\lambda)\,\cf{G}(s\beta_j+t\beta)
    &= \cf{U}(s\lambda)\,\cf{G}(s\beta_j)\,\cf{G}(t\beta),
\end{align}
where the constants $d_j,d_{k_0}$ have cancelled.
Since $\cf{U}$ is continuous with $\cf{U}(0)=1$, it is non-zero in a
neighbourhood of the origin, so we may cancel it there and take logarithms of
the nowhere vanishing Gaussian characteristic functions.  The two exponents
are the quadratics of \cref{not:degenerate}; their difference is continuous,
takes values in $2\pi i\,\Ints$ on the connected set where the identity holds,
and vanishes at $(0,0)$, hence is identically zero.  This gives, for all small
$s,t$,
\begin{align}
  im\bigl(s\beta_j+t\beta\bigr)
    - \tfrac12\tau^2\bigl(s\beta_j+t\beta\bigr)^2
  &= im\bigl(s\beta_j+t\beta\bigr)
    - \tfrac12\tau^2\Bigl(\bigl(s\beta_j\bigr)^2+\bigl(t\beta\bigr)^2\Bigr),
\end{align}
i.e.\ $\tau^2\beta_j\beta\,st = 0$ for all small $s,t$.  As $\tau>0$ and
$\beta\neq0$ this forces $\beta_j = 0$, a contradiction.

Hence $\beta_j=0$ for all $j\neq k_0$, so the $l_0$-th column of $C$ equals
$\beta e_{k_0}$ and $C = P\Lambda$ for a permutation matrix $P$ and an
invertible diagonal $\Lambda$.  With \cref{eq:C-relation} this proves the
first display, and $A\up{2} = A\up{1}C = A\up{1}P\Lambda$.

\emph{The normalised form.}
If both source vectors are centred, taking expectations in
$Z\up{1} = P\Lambda Z\up{2}+c$ gives $c=0$.  If in addition all variances
equal $1$, then each component of $Z\up{1}$ equals $\lambda$ times a
component of $Z\up{2}$, where $\lambda$ is the corresponding diagonal entry
of $\Lambda$; taking variances gives $1=\lambda^2$, so $\lambda=\pm1$.
\end{proof}

\begin{remark}[Two Gaussian sources already destroy identifiability]
\label{rem:two-gaussians}
Hypothesis (iii) of \cref{cor:complete-ident} cannot be weakened.
Let $k=2$ and $Z\sim\Normal(0,\Id{2})$, $A\up{1} = \Id{2}$, so $X = Z$.
For any orthogonal $Q\in\Orth{2}$ put $A\up{2} := Q\T$ and $Z\up{2} := QZ$.
Then $Z\up{2}\sim\Normal(0,\Id{2})$ again has independent components and
$A\up{2}Z\up{2} = Q\T QZ = X$, so the whole rotation group is compatible with
the observed law; cf.\ \cref{rem:gaussian-not-identifiable}.
The same construction applies inside any two-dimensional Gaussian block, which
is why at most one Gaussian source is allowed.
This degeneracy reappears analytically in \cref{prop:gaussian-boundary}
below: two Gaussian sources force the stability quantities to satisfy
$\zeta_j\zeta_l = 1$, so the linearised dynamics acquires a neutral direction.
With the true score the implication reverses as well, by \cref{cor:true-score}
($\zeta_j\ge1$ with equality only in the Gaussian case).  When in addition the
two source variances agree, the neutral direction is the antisymmetric one,
i.e.\ the infinitesimal version of the rotation $Q$ above; for unequal
variances it is $(\sigma_j^2,-\sigma_l^2)$ instead.
\end{remark}

\begin{remark}[Relation to the general theory]
\label{rem:complete-vs-general}
\Cref{cor:complete-ident} is the noiseless, square, full-rank corner of the
theory of \cref{sec:non-gaussian,sec:gaussian-free}.
Comparing:
\cref{thm:non-gaussian} allows a rectangular $A$ and additive Gaussian noise
but must then leave the sources undetermined up to an extra additive Gaussian
term; \cref{thm:gf-identifiability} removes that term at the price of the
Gaussian-free hypothesis.
In the complete noiseless model there is no noise to hide anything in, so the
weakest hypothesis -- at most one Gaussian source -- already suffices, and the
conclusion holds almost surely rather than merely in distribution.
\end{remark}

\subsection{The maximum likelihood objective}
\label{ssec:ml}

From here on we work under \cref{asm:ica} and additionally assume that each
source $Z_j$ has a strictly positive Lebesgue density $p_{Z_j}$.
The model is parametrised by the unmixing matrix $W = A^{-1}$ together with
$k$ \emph{model} source densities $p_1,\dots,p_k$, which are a modelling
choice and need not equal the true densities $p_{Z_1},\dots,p_{Z_k}$:
\begin{align}
  \theta &= \bigl(W, p_1,\dots,p_k\bigr),
  & p_Z(z) &= p_1(z_1)\cdots p_k(z_k) .
\end{align}

\begin{proposition}[Model density and log-likelihood]
\label{prop:ml-objective}
Under \cref{asm:ica}, the density of $X = W^{-1}Z$ induced by the parameter
$\theta$ is
\begin{align}
  \label{eq:model-density}
  q_\theta(x) &= \abs{\det W}\prod_{j=1}^{k}p_j\bigl(w_j\T x\bigr),
  \qquad x\in\R^k .
\end{align}
Consequently the population log-likelihood, with $X$ distributed according to
the true law $q$, is
\begin{align}
  \label{eq:ml-objective}
  \mathcal{L}(\theta) &:= \Ex_{X\sim q}\bigl[\log q_\theta(X)\bigr]
   = \log\abs{\det W}
     + \sum_{j=1}^{k}\Ex_{X\sim q}\bigl[\log p_j\bigl(w_j\T X\bigr)\bigr],
\end{align}
and its empirical counterpart is
\begin{align}
  \label{eq:ml-empirical}
  \mathcal{L}_T(\theta) &:= \frac1T\sum_{t=1}^{T}\log q_\theta\bigl(X(t)\bigr)
   = \log\abs{\det W}
     + \frac1T\sum_{t=1}^{T}\sum_{j=1}^{k}
       \log p_j\bigl(w_j\T X(t)\bigr).
\end{align}
Moreover, if $q$ has a density and finite differential entropy,
\begin{align}
  \label{eq:ml-is-kl}
  \mathcal{L}(\theta)
    &= -\,\mathrm{KL}\bigl(q\,\|\,q_\theta\bigr) - H(q),
  & H(q) &:= -\Ex_{X\sim q}[\log q(X)] .
\end{align}
\end{proposition}

\begin{proof}
\Cref{eq:model-density} is the change-of-variables formula for the
diffeomorphism $z = Wx$ with Jacobian determinant $\det W$: for a Borel set
$S$, $\Prb[X\in S] = \Prb[Z \in WS] = \int_{WS}p_Z(z)dz
= \int_S p_Z(Wx)\abs{\det W}\,dx$.
\Cref{eq:ml-objective,eq:ml-empirical} follow by taking logarithms and using
$\log\prod_j = \sum_j\log$.
For \cref{eq:ml-is-kl},
$\mathrm{KL}(q\|q_\theta) = \Ex_q[\log q] - \Ex_q[\log q_\theta]
= -H(q) - \mathcal{L}(\theta)$.
\end{proof}
Since $H(q)$ does not depend on $\theta$, \cref{eq:ml-is-kl} says that
maximising $\mathcal{L}$ is the same as minimising the Kullback--Leibler
divergence from the true law to the model.

\begin{remark}[The empirical objective does not need independence over time]
\label{rem:no-time-independence}
\Cref{eq:ml-empirical} is the exact log-likelihood only if the samples
$X(1),\dots,X(T)$ are independent.  \Cref{asm:ica}\,(i) deliberately does not
assume this: in the cocktail-party application consecutive samples of an audio
signal are strongly dependent.
Nothing below breaks, because we never use \cref{eq:ml-empirical} as a
likelihood.  We use it as an \emph{M-estimation criterion} whose population
version \cref{eq:ml-objective} depends only on the common law of $X(t)$, and
all our statements -- the stationarity conditions of
\cref{prop:stationary-points} and the stability analysis of
\cref{thm:stability} -- are statements about that population criterion.
Provided $(Z(t))_t$ is stationary and ergodic -- enough for a law of large
numbers, so that $\mathcal{L}_T\to\mathcal{L}$ almost surely -- dependence
across $t$ affects only the rate of that convergence, not the location of the
optima.  Some such hypothesis is needed: if $Z(t)\equiv Z(1)$ for every $t$,
then all five conditions of \cref{asm:ica} hold, yet
$\mathcal{L}_T = \log q_\theta(x(1))$ for every $T$, and this is unbounded
above (send a row of $W$ towards a direction orthogonal to $x(1)$ while
$\log\abs{\det W}\to\infty$).
\end{remark}

\begin{remark}[Centring and whitening]
\label{rem:preprocessing}
Two preprocessing steps are customary.
\begin{enumerate}[label=(\arabic*)]
  \item \emph{Centring:} replace $X(t)$ by $X(t)-\Ex[X]$, so that
        $\Ex[X]=0$; since $A$ is invertible this is equivalent to $\Ex[Z]=0$,
        and being a deterministic shift of each $Z_j$ separately it preserves
        \cref{asm:ica}\,(i).  In practice one subtracts the sample mean
        $\bar X$, which is an approximation to this whose validity again needs
        the law of large numbers of \cref{rem:no-time-independence}:
        subtracting $\bar X$ mixes the time points and can destroy the
        independence of the components at a fixed $t$.
  \item \emph{Whitening:} assume $\Ex\norm{Z}^2<\infty$ (so that
        $\Sigma_X := \Cov(X)$ exists and is invertible) and replace $X$ by
        $X' := \Sigma_X^{-1/2}X$, whose mixing matrix is
        $A' := \Sigma_X^{-1/2}A$.  With centred unit-variance sources this
        gives
        \begin{align}
          \Id{k} = \Cov(X') = A'\,\Cov(Z)\,(A')\T = A'(A')\T,
        \end{align}
        with the \emph{symmetric} square root $\Sigma_X^{-1/2}$, which exists
        and is invertible because the $Z_j$ are non-constant with finite
        variance; that is, the mixing matrix of the whitened data \emph{is}
        orthogonal.  For algorithms that constrain $W$ to $\Orth{k}$ this
        reduces the
        number of free parameters from $k^2$ to the $k(k-1)/2$ dimensions of
        $\Orth{k}$, and it is what most ICA algorithms exploit.  The
        relative-gradient update \cref{eq:update-final} does not preserve
        orthogonality, so it gains no such reduction.
\end{enumerate}
The algorithm developed below needs neither step in order to be well defined.
By \cref{prop:equivariance} its trajectory depends on $A$ only through the
initial global system matrix $R_0 = W_0A$, so whitening cannot change the
shape of the dynamics -- but it does change $R_0$, replacing it by an
orthogonal matrix when $W_0 = \Sigma_X^{-1/2}$, and a better conditioned
starting point is a genuine practical benefit.  Centring is advisable in any
case, because \cref{thm:stability} assumes $\Ex[Z]=0$.
\end{remark}

\begin{proposition}[Gradient of the log-likelihood]
\label{prop:gradient}
Fix source densities $p_1,\dots,p_k$ that are differentiable and strictly
positive, and write
\begin{align}
  \label{eq:score-def}
  \eta_j &:= \bigl(\log p_j\bigr)' = \frac{p_j'}{p_j}
\end{align}
for the \emph{model score} of the $j$-th source (\cref{rem:two-scores}).
Put $C := WX$, and write $\eta(C)$ for the vector with components
$\eta_j(C_j)$.
Then, under integrability conditions permitting differentiation under the
expectation,
\begin{align}
  \label{eq:gradient}
  \frac{\partial\mathcal{L}}{\partial W_{ij}}
    &= \Ex\bigl[\eta_i(C_i)\,X_j\bigr] + \bigl(W^{-1}\bigr)_{ji},
  \qquad\text{that is}\qquad
  \nabla_W\mathcal{L} = W\invT + \Ex\bigl[\eta(C)X\T\bigr].
\end{align}
\end{proposition}

\begin{proof}
Differentiating \cref{eq:ml-objective} term by term,
$\partial_{W_{ij}}\log p_l(w_l\T X) = \delta_{li}\,\eta_l(w_l\T X)\,X_j$ by
the chain rule, and summing over $l$ leaves the single term
$\eta_i(C_i)X_j$.
For the determinant, Jacobi's formula gives
$\partial_{W_{ij}}\log\abs{\det W} = (W^{-1})_{ji}$, i.e.\
$\nabla_W\log\abs{\det W} = W\invT$.
\end{proof}

\begin{remark}[Two scores: the model score and the true score]
\label{rem:two-scores}
From here to the end of the section, two different objects compete for the
name ``score'', and almost every subtlety in what follows is a statement about
their relationship.  We keep them notationally apart.
\begin{itemize}
  \item $\eta_j := (\log p_j)'$ of \cref{eq:score-def}, the score of the
        \emph{model} density
        $p_j$.  This is a modelling choice, fixed by the user before the
        algorithm is run, and it is the function the algorithm actually
        evaluates -- in \cref{alg:online-ica} it is applied to the current
        estimate $C_j$ of the $j$-th source.  We also call it the
        \emph{nonlinearity}.  Nothing forces it to be the score of the
        \emph{true} source density.  The three used in practice are in fact
        scores of perfectly good densities -- just not, in general, of the
        right one: $-\tanh$ is the score of the hyperbolic secant density
        $\tfrac1\pi\sech$, with $\kurt = +2$; $\tanh-\mathrm{id}$ is the
        score of $\propto\cosh(u)e^{-u^2/2}$, i.e.\ of the mixture
        $\tfrac12\bigl(\Normal(-1,1)+\Normal(1,1)\bigr)$, with
        $\kurt = -\tfrac12$; and $-u^3$ is the score of
        $\propto e^{-u^4/4}$, with $\kurt \approx -0.81$.  Choosing $\eta_j$
        thus amounts to positing a source law, and the three choices posit a
        leptokurtic, a mildly platykurtic and a platykurtic one respectively.
        Not every nonlinearity is a score, though: $+u^3$ would require the
        non-integrable $e^{u^4/4}$.
  \item $\tsc_j := \bigl(\log p_{Z_j}\bigr)'$, the score of the \emph{true}
        density $p_{Z_j}$ of the source $Z_j$.  This is a property of the data
        generating process, and it is unknown.
\end{itemize}
The maximum likelihood derivation of \cref{ssec:ml} would use $\tsc_j$ if it
could; the algorithm uses $\eta_j$ because it must.
\Cref{thm:stability} is a statement about an arbitrary $\eta_j$, and this is
exactly what makes it useful: it says how far $\eta_j$ may deviate from
$\tsc_j$ before the separating solution stops attracting.
\Cref{cor:true-score} is the special case $\eta_j = \tsc_j$.
Note that $\eta_j$ is evaluated along the trajectory at $C_j = (RZ)_j$, a
mixture of the sources, whereas $\tsc_j$ belongs to $Z_j$ alone; the two are
brought into contact only through their values on the sources, i.e.\ at the
fixed point $R=\Id{k}$, where $C_j = Z_j$.
A third score makes a single appearance, after \cref{prop:stationary-points}:
$(\log p_{C_j})'$, the true score of the current \emph{estimate}, which
coincides with $\tsc_j$ at $R=\Id{k}$.
\end{remark}

\subsection{The relative gradient: what the ``preconditioner'' really is}
\label{ssec:relative-gradient}

A plain gradient ascent step $W \leftarrow W + \alpha\nabla_W\mathcal{L}$
requires the matrix inverse $W\invT$ at every step, which is both
expensive and numerically delicate.  The standard remedy is to multiply the
gradient by $W\T W$ on the right, giving the update
\begin{align}
  \label{eq:update-rule}
  W &\leftarrow W + \alpha\,\bigl(\nabla_W\mathcal{L}\bigr)\,W\T W .
\end{align}
This is often presented as a Newton step with an approximate inverse Hessian.
It is not, and \cref{rem:not-a-hessian} explains what it is instead: an
\emph{exact} gradient, taken with respect to a different -- and for this
problem far more natural -- Riemannian metric on the group of invertible
matrices.

\begin{definition}[Relative gradient]
\label{def:relative-gradient}
Let $F$ be a differentiable real-valued function on the group $GL(k)$ of
invertible $k\times k$ matrices.  The \emph{relative gradient} of $F$ at $W$
is the matrix $\nabla^{\mathrm{rel}}F(W)$ determined by the first-order
expansion under \emph{multiplicative} perturbations,
\begin{align}
  \label{eq:rel-grad-def}
  F\bigl((\Id{k}+\varepsilon)W\bigr)
    &= F(W)
     + \bigl\langle \nabla^{\mathrm{rel}}F(W),\,\varepsilon\bigr\rangle
     + o\bigl(\norm{\varepsilon}\bigr),
  \qquad \varepsilon\in\R^{k\times k},
\end{align}
where $\langle M,N\rangle := \tr(M\T N)$ is the Frobenius inner product.
\end{definition}

\begin{proposition}[Relative gradient of the ICA likelihood]
\label{prop:rel-gradient}
For any differentiable $F$ on $GL(k)$ we have
$\nabla^{\mathrm{rel}}F(W) = \bigl(\nabla F(W)\bigr)W\T$, so that a relative
gradient ascent step $W\leftarrow(\Id{k}+\alpha\nabla^{\mathrm{rel}}F(W))W$ is
exactly \cref{eq:update-rule}.
For the ICA log-likelihood \cref{eq:ml-objective},
\begin{align}
  \label{eq:rel-gradient-ica}
  \nabla^{\mathrm{rel}}\mathcal{L}(W)
    &= \Id{k} + \Ex\bigl[\eta(C)\,C\T\bigr],
  \qquad C = WX,
\end{align}
and the update rule reads
\begin{align}
  \label{eq:update-final}
  W &\leftarrow \Bigl(\Id{k}
      + \alpha\bigl(\Id{k} + \Ex\bigl[\eta(C)C\T\bigr]\bigr)\Bigr)W .
\end{align}
\end{proposition}

\begin{proof}
By the ordinary chain rule,
$F((\Id{k}+\varepsilon)W) = F(W+\varepsilon W)
 = F(W) + \langle\nabla F(W),\varepsilon W\rangle + o(\norm{\varepsilon})$,
and
$\langle\nabla F(W),\varepsilon W\rangle
 = \tr\bigl(\nabla F(W)\T\varepsilon W\bigr)
 = \tr\bigl(W\,\nabla F(W)\T\,\varepsilon\bigr)
 = \bigl\langle \nabla F(W)W\T,\varepsilon\bigr\rangle$,
which is \cref{eq:rel-grad-def} with
$\nabla^{\mathrm{rel}}F(W) = \nabla F(W)W\T$.
Substituting \cref{eq:gradient} and using
$W\invT W\T = \Id{k}$ and $X\T W\T = C\T$,
\begin{align}
  \nabla^{\mathrm{rel}}\mathcal{L}(W)
    &= \Bigl(W\invT + \Ex\bigl[\eta(C)X\T\bigr]\Bigr)W\T
     = \Id{k} + \Ex\bigl[\eta(C)\,X\T W\T\bigr]
     = \Id{k} + \Ex\bigl[\eta(C)C\T\bigr].
\end{align}
\end{proof}

\begin{remark}[It is a metric, not a Hessian]
\label{rem:not-a-hessian}
Differentiating \cref{eq:gradient} once more gives the true Hessian of
$\mathcal{L}$,
\begin{align}
  \label{eq:true-hessian}
  \frac{\partial^2\mathcal{L}}{\partial W_{lm}\,\partial W_{ij}}
    &= \delta_{il}\,\Ex\bigl[\eta_i'(C_i)X_jX_m\bigr]
     - \bigl(W^{-1}\bigr)_{jl}\bigl(W^{-1}\bigr)_{mi},
\end{align}
using $\partial(W^{-1})_{ji}/\partial W_{lm} = -(W^{-1})_{jl}(W^{-1})_{mi}$.
The operator implicitly used in \cref{eq:update-rule} is, in coordinates,
\begin{align}
  K_{(lm),(ij)} &= \delta_{li}\bigl(W\T W\bigr)_{jm}
   = \delta_{li}\sum_{r=1}^{k}W_{rj}W_{rm},
\end{align}
since $\bigl[(\nabla\mathcal{L})W\T W\bigr]_{lm}
 = \sum_{ij}K_{(lm),(ij)}(\nabla\mathcal{L})_{ij}$.
Two observations settle the matter.
First, $K$ does not depend on the data at all, whereas
\cref{eq:true-hessian} does, through
$\Ex[\eta_i'(C_i)X_jX_m]$; so $K$ cannot be the inverse of
\cref{eq:true-hessian}, not even approximately, uniformly in the source
distribution.
Second, $K$ has an exact interpretation that has nothing to do with curvature:
it is the inverse metric tensor of the right-invariant Riemannian metric on
$GL(k)$ in which distances are measured multiplicatively, and
\cref{eq:update-final} is the corresponding steepest-ascent step.
This is Amari's \emph{natural gradient} \citep{amari1996newlearning,
amari1998natural} and Cardoso and Laheld's \emph{relative gradient}
\citep{cardoso1996equivariant}; MacKay's derivation of exactly the update
\cref{eq:update-final} calls it the \emph{covariant} algorithm
\citep{mackay1996ica}.
The essential point for what follows is that \cref{eq:update-final} involves
\emph{no approximation whatsoever}: it is the exact gradient of the exact
objective for the right-invariant metric just described.  (Not
for a different \emph{chart}: a gradient in a chart is the gradient for the
flat metric that chart induces, and no flat left- or right-invariant metric
exists on $GL(k)$ for $k\ge2$.)  The only approximation in the whole algorithm
is the choice of the model scores $\eta_j$ -- that is, the extent to which
they fail to be the true scores $\tsc_j$ -- and that is the subject of
\cref{ssec:stability}.
\end{remark}

The reason to prefer this particular gradient is the following invariance,
which is why the algorithm is used at all.

\begin{proposition}[Equivariance]
\label{prop:equivariance}
Write $R := WA$ for the \emph{global system matrix}, so that $C = WX = RZ$.
Then the update \cref{eq:update-final} induces on $R$ the update
\begin{align}
  \label{eq:update-R}
  R &\leftarrow \Bigl(\Id{k}
     + \alpha\bigl(\Id{k}+\Ex\bigl[\eta(RZ)(RZ)\T\bigr]\bigr)\Bigr)R,
\end{align}
which does not involve $A$.
Consequently, if $(W_n)_{n\ge0}$ and $(W_n')_{n\ge0}$ are the iterates
produced from data $X = AZ$ and $X' = A'Z$ with the same sources and with
initialisations satisfying $W_0A = W_0'A'$, then $W_nA = W_n'A'$ and
$W_nX = W_n'X'$ for all $n$.
\end{proposition}

\begin{proof}
Multiplying \cref{eq:update-final} on the right by $A$ and using $C = WX =
WAZ = RZ$ turns the update for $W$ into \cref{eq:update-R}, in which $A$ no
longer appears.  The second statement follows by induction: the hypothesis
$W_nA = W_n'A'$ makes both iterates apply the same map \cref{eq:update-R} to
the same matrix $R_n$, so $W_{n+1}A = W_{n+1}'A'$; and
$W_nX = R_nZ = W_n'X'$.
\end{proof}
The trajectory of the recovered signal $C$ therefore depends on the mixing
matrix only through $R_0 = W_0A$: two mixings $A$ and $MA$ give identical
$C$-trajectories provided the initialisations are co-transformed, however
badly conditioned $M$ is.  This is \emph{not} the statement that the
conditioning of $A$ is irrelevant; see \cref{rem:terminology-equivariant}.

\begin{remark}[Equivariant, relative, natural, covariant]
\label{rem:terminology-equivariant}
Four names circulate for the same update, and they name four different things
about it.
\emph{Equivariant} \citep{cardoso1996equivariant} names the property proved in
\cref{prop:equivariance}: writing $\Phi_n(x,W_0)$ for the $n$-th iterate
produced from data $x$ and initialisation $W_0$, that proposition says
\begin{align}
  \Phi_n\bigl(Mx,\;W_0M^{-1}\bigr) &= \Phi_n(x,W_0)\,M^{-1},
  \qquad M\in GL(k),
\end{align}
which is equivariance of the map $(x,W_0)\mapsto W_n$ under the action
$x\mapsto Mx$, $W_0\mapsto W_0M^{-1}$ of $GL(k)$.  Note that the
initialisation is co-transformed: for a \emph{fixed} $W_0$, such as the
default $W_0 = \Id{k}$ of \cref{alg:online-ica}, the iterates are not
equivariant, which is the point already made in \cref{rem:preprocessing}.
\emph{Relative gradient} \citep{cardoso1996equivariant} names the
construction: the gradient with respect to multiplicative perturbations,
\cref{def:relative-gradient}.
\emph{Natural gradient} \citep{amari1998natural} names the same object as
steepest ascent for a Riemannian metric on $GL(k)$.
\emph{Covariant} \citep{mackay1996ica} borrows the physicists' sense of the
word -- form-invariant under a change of parametrisation.

We prefer \emph{equivariant}.  It names a property of the algorithm that we
have actually proved, rather than a property of the metric used to derive it;
it is the established term in the source-separation literature; and
``covariant'' is doubly unfortunate here, since it invites confusion with
covariance matrices, and since in the tensor-calculus sense of the word the
object $K$ of \cref{rem:not-a-hessian} raises an index and so produces a
\emph{contra}variant quantity.  When the emphasis is on the derivation rather
than on the algorithm, ``relative gradient'' and ``natural gradient'' are both
accurate.
\end{remark}

\subsection{Stationary points, stability, and the choice of nonlinearity}
\label{ssec:stability}

We now address the question that the derivation so far has left open: the
true scores $\tsc_j$ are unknown, so the algorithm is run with a fixed guess
$\eta_j$ in their place (\cref{rem:two-scores}).  Why should it still work?

\begin{proposition}[Stationary points are nonlinear decorrelations]
\label{prop:stationary-points}
$W$ is a stationary point of the relative gradient, i.e.\
$\nabla^{\mathrm{rel}}\mathcal{L}(W) = 0$, if and only if
\begin{align}
  \label{eq:stationarity}
  \Ex\bigl[\eta_j(C_j)\,C_l\bigr] &= -\delta_{jl},
  \qquad j,l = 1,\dots,k .
\end{align}
\end{proposition}

\begin{proof}
Immediate from \cref{eq:rel-gradient-ica}: the $(j,l)$ entry of
$\Id{k}+\Ex[\eta(C)C\T]$ is $\delta_{jl}+\Ex[\eta_j(C_j)C_l]$.
\end{proof}
The $k$ diagonal equations of \cref{eq:stationarity} fix the scale of each
recovered component; the $k(k-1)$ off-diagonal equations say that each
transformed component $\eta_j(C_j)$ is orthogonal to every untransformed
$C_l$, $l\neq j$, and hence -- the data being centred -- uncorrelated with
it.  This is a \emph{nonlinear decorrelation} condition.
If the model score $\eta_j$ happens to coincide with the true score
$(\log p_{C_j})'$ of the current estimate $C_j$, and $c\,p_{C_j}(c)\to0$ as
$\abs{c}\to\infty$, then the $j$-th diagonal equation holds automatically:
integration by parts gives
$\Ex[\eta_j(C_j)C_j] = \int c\,p_{C_j}'(c)\,dc = -1$.

By \cref{prop:equivariance} the dynamics lives on the global system matrix
$R$, and by \cref{cor:complete-ident} the separating solutions are exactly the
matrices $R = P\Lambda$.  Relabelling and rescaling the sources, it suffices
to analyse the fixed point $R = \Id{k}$.  When the $\eta_j$ differ across
components this relabelling matters: the conditions below are to be read for
the pairing of $\eta_j$ with the source that $P$ sends to position $j$, and a
different pairing has different $\zeta$'s and its own verdict.  That is
exactly what makes a wrong assignment of model scores to components unstable.
We first record that the scale can always be arranged.

\begin{lemma}[The fixed-point scale exists and is unique]
\label{lem:fixed-point-scale}
Let $Y$ be a real-valued random variable that is not almost surely $0$ and let
the model score $\eta$ be one of
\begin{align}
  \eta(u) = -\tanh(u), \qquad
  \eta(u) = \tanh(u)-u, \qquad
  \eta(u) = -u^3 ,
\end{align}
with $\Ex[\abs{Y}]<\infty$ in the first case, $\Ex[Y^2]<\infty$ in the second
and $\Ex[Y^4]<\infty$ in the third.  Then there is exactly one $c>0$ with
\begin{align}
  \Ex\bigl[\eta(cY)\,cY\bigr] &= -1 .
\end{align}
\end{lemma}

\begin{proof}
In all three cases write $h(c) := -\Ex[\eta(cY)cY] = \Ex[\varphi(cY)]$ with
$\varphi(u) := u\tanh(u)$, $\varphi(u) := u\bigl(u-\tanh(u)\bigr)$ and
$\varphi(u) := u^4$ respectively.  Each $\varphi$ is non-negative, even,
vanishes only at $0$ and is strictly increasing in $\abs{u}$; for the second
this is because $u-\tanh u$ has the sign of $u$ and $\abs{u-\tanh u}$
increases in $\abs{u}$.
Hence $c\mapsto\varphi(cY(\omega))$ is non-decreasing in $c>0$ for every
$\omega$, and strictly increasing whenever $Y(\omega)\neq0$.
So $h$ is continuous by dominated convergence on compacts, strictly
increasing, $h(0^+)=0$, and $h(c)\to\infty$ as $c\to\infty$ by monotone
convergence.  The intermediate value theorem gives a $c>0$ with $h(c)=1$, and strict
monotonicity makes it unique.
\end{proof}

\begin{remark}[Why the normalisation makes $\zeta_j$ well defined]
\label{rem:zeta-well-defined}
The quantities $\beta_j$, $\gamma_j$ and $\zeta_j$ of
\cref{eq:stability-moments} below are \emph{not} invariant under rescaling
the source: replacing $Z_j$ by $cZ_j$ changes
$\beta_j = \Ex[\eta_j'(Z_j)]$ and $\sigma_j^2 = \Ex[Z_j^2]$ separately, and
in general changes their product.  This is not a defect but a consequence of
the fact that the fixed point under analysis is $R=\Id{k}$, which presupposes
a choice of scale for the sources.  \Cref{lem:fixed-point-scale} says that
for the three nonlinearities in use there is exactly one such choice
compatible with stationarity \emph{up to sign}, namely the one making
\cref{eq:scale-normalisation} hold; the sign is invisible to
$\beta_j,\gamma_j,\zeta_j$, because $\eta_j'$ is even for all three, and it
is exactly the sign ambiguity of \cref{cor:complete-ident}.  Once that normalisation is imposed,
$\zeta_j$ is a well-defined functional of $\Law(Z_j)$ and of the model score
$\eta_j$ alone,
and the stability conditions \cref{eq:stability-conditions} are statements
about the source law rather than about an arbitrary scaling.  This is what
makes the entries of \cref{tab:stability} meaningful.
\end{remark}

\begin{theorem}[Local stability of the separating solution]
\label{thm:stability}
Let $k\ge2$, let $Z$ have mutually independent components with $\Ex[Z_j]=0$
and $\sigma_j^2 := \Ex[Z_j^2]\in(0,\infty)$, let the model scores
$\eta_1,\dots,\eta_k$ of \cref{rem:two-scores} be arbitrary twice
continuously differentiable functions, and assume
\begin{align}
  \label{eq:stability-regularity}
  \sup_{u\in\R}\bigabs{\eta_j''(u)} &< \infty
  \quad\text{for all } j,
  & \Ex\bigl[\norm{Z}^3\bigr] &< \infty
\end{align}
(see \cref{rem:reading-stability}).  Assume finally the normalisation
\begin{align}
  \label{eq:scale-normalisation}
  \Ex\bigl[\eta_j(Z_j)Z_j\bigr] &= -1, \qquad j=1,\dots,k,
\end{align}
holds.  Then $R = \Id{k}$ is a stationary point of \cref{eq:update-R}.
Define
\begin{align}
  \label{eq:stability-moments}
  \beta_j &:= \Ex\bigl[\eta_j'(Z_j)\bigr], &
  \gamma_j &:= \Ex\bigl[\eta_j'(Z_j)Z_j^2\bigr], &
  \zeta_j &:= -\beta_j\,\sigma_j^2 .
\end{align}
Writing $R = \Id{k}+\varepsilon$, the update \cref{eq:update-R} linearised at
$\varepsilon=0$ decouples into the $k$ scalar recursions
\begin{align}
  \label{eq:lin-diagonal}
  \varepsilon_{jj} &\longmapsto
     \bigl(1+\alpha(\gamma_j-1)\bigr)\,\varepsilon_{jj}
\end{align}
and the $\binom{k}{2}$ two-dimensional recursions
\begin{align}
  \label{eq:lin-offdiagonal}
  \bmat{\varepsilon_{jl}\\ \varepsilon_{lj}}
    &\longmapsto \bigl(\Id{2}+\alpha M_{jl}\bigr)
       \bmat{\varepsilon_{jl}\\ \varepsilon_{lj}},
  & M_{jl} &:= \bmat{\beta_j\sigma_l^2 & -1\\ -1 & \beta_l\sigma_j^2},
  \qquad j\neq l .
\end{align}
Consequently the conditions
\begin{align}
  \label{eq:stability-conditions}
  \gamma_j &< 1 \quad\text{for all } j,
  & \zeta_j &> 0 \quad\text{for all } j,
  & \zeta_j\,\zeta_l &> 1 \quad\text{for all } j\neq l
\end{align}
are sufficient for $R=\Id{k}$ to be locally asymptotically stable for all
sufficiently small step sizes $\alpha>0$; if any one of them is reversed
strictly, then $R=\Id{k}$ is unstable for every $\alpha>0$.
\end{theorem}

\begin{proof}
Stationarity at $R=\Id{k}$: the $(j,l)$ entry of
$\Id{k}+\Ex[\eta(Z)Z\T]$ is $\delta_{jl}+\Ex[\eta_j(Z_j)Z_l]$, which vanishes
for $j=l$ by \cref{eq:scale-normalisation} and for $j\neq l$ because
independence and $\Ex[Z_l]=0$ give
$\Ex[\eta_j(Z_j)Z_l] = \Ex[\eta_j(Z_j)]\Ex[Z_l] = 0$.

Now let $R = \Id{k}+\varepsilon$ and $C = RZ$, so
$C_j = Z_j + \sum_m\varepsilon_{jm}Z_m$.  Write
$G(R) := \Id{k}+\Ex[\eta(C)C\T]$.  Since \cref{eq:update-R} reads
$R\mapsto(\Id{k}+\alpha G(R))R$ and $G(\Id{k})=0$, the induced map on
$\varepsilon$ is
$\varepsilon\mapsto\varepsilon+\alpha G(\Id{k}+\varepsilon)
 +\alpha G(\Id{k}+\varepsilon)\varepsilon
 = \varepsilon+\alpha G_1(\varepsilon)+O(\norm{\varepsilon}^2)$,
where $G_1$ is the derivative of $G$ at $\Id{k}$.
Expanding to first order,
\begin{align}
  \eta_j(C_j) &= \eta_j(Z_j)
     + \eta_j'(Z_j)\sum_m\varepsilon_{jm}Z_m + O\bigl(\norm{\varepsilon}^2\bigr),
\end{align}
so that
\begin{align}
  \Ex\bigl[\eta_j(C_j)C_l\bigr]
    &= \underbrace{\Ex\bigl[\eta_j(Z_j)Z_l\bigr]}_{=-\delta_{jl}}
     + \sum_m\varepsilon_{lm}\Ex\bigl[\eta_j(Z_j)Z_m\bigr]
     + \sum_m\varepsilon_{jm}\Ex\bigl[\eta_j'(Z_j)Z_mZ_l\bigr]
     + O\bigl(\norm{\varepsilon}^2\bigr).
\end{align}
In the second sum $\Ex[\eta_j(Z_j)Z_m]$ equals $-1$ for $m=j$ and $0$
otherwise, contributing $-\varepsilon_{lj}$.
In the third sum, independence and $\Ex[Z_m]=0$ leave only $m=l$ when
$l\neq j$, contributing $\varepsilon_{jl}\Ex[\eta_j'(Z_j)]\Ex[Z_l^2]
= \beta_j\sigma_l^2\varepsilon_{jl}$; and only $m=j$ when $l=j$, contributing
$\gamma_j\varepsilon_{jj}$.
Hence
\begin{align}
  G_{jl}(\Id{k}+\varepsilon) &=
  \begin{cases}
    (\gamma_j-1)\,\varepsilon_{jj}, & l=j,\\[2pt]
    \beta_j\sigma_l^2\,\varepsilon_{jl} - \varepsilon_{lj}, & l\neq j,
  \end{cases}
\end{align}
up to $O(\norm{\varepsilon}^2)$, which is exactly
\cref{eq:lin-diagonal,eq:lin-offdiagonal}.

For the stability criterion, $\Id{k}+\alpha M$ is a contraction for all
sufficiently small $\alpha>0$ as soon as every eigenvalue of $M$ has
strictly negative real part, and it fails to be one, for every $\alpha>0$, as
soon as some eigenvalue has strictly positive real part; eigenvalues on the
imaginary axis are undecided.  For the scalar blocks the relevant condition
is $\gamma_j-1<0$.
Each $M_{jl}$ is real symmetric, so its eigenvalues are real and both are
negative if and only if $\tr M_{jl}<0$ and $\det M_{jl}>0$, that is
\begin{align}
  \beta_j\sigma_l^2 + \beta_l\sigma_j^2 &< 0,
  & \beta_j\beta_l\,\sigma_j^2\sigma_l^2 &> 1 .
\end{align}
In terms of $\zeta_j = -\beta_j\sigma_j^2$ the second condition reads
$\zeta_j\zeta_l>1$, and the first reads
$\zeta_j\sigma_l^2/\sigma_j^2+\zeta_l\sigma_j^2/\sigma_l^2>0$.
If $\zeta_j\zeta_l>1$ then $\zeta_j$ and $\zeta_l$ have the same sign; if both were
negative the first condition would fail, so both are positive, and conversely
$\zeta_j,\zeta_l>0$ implies the first condition.  This is
\cref{eq:stability-conditions}.

It remains to check that each strict reversal produces an eigenvalue with
strictly positive real part.  For $\gamma_j>1$ the diagonal multiplier
$1+\alpha(\gamma_j-1)$ exceeds $1$ for every $\alpha>0$.
For $\zeta_j\zeta_l<1$ we get $\det M_{jl} = \zeta_j\zeta_l-1<0$, so
$M_{jl}$ has one positive and one negative eigenvalue.
For $\zeta_j<0$, pick any $l\neq j$ (possible since $k\ge2$): if
$\zeta_j\zeta_l<1$ we are in the previous case; otherwise $\zeta_j\zeta_l\ge1$
forces $\zeta_l<0$, and then
$\tr M_{jl} = -\zeta_j\sigma_l^2/\sigma_j^2-\zeta_l\sigma_j^2/\sigma_l^2>0$,
so at least one eigenvalue of $M_{jl}$ is positive.
Finally, if $\Id{k}+\alpha M$ has an eigenvalue $1+\alpha\lambda$ with
$\Real\lambda>0$, then
$\abs{1+\alpha\lambda}^2 = 1+2\alpha\Real\lambda+\alpha^2\abs{\lambda}^2>1$
for every $\alpha>0$, so the linearised map is expanding in the corresponding
direction.  Since the remainder in \cref{eq:lin-diagonal,eq:lin-offdiagonal}
is $O(\norm{\varepsilon}^2)$ and the linear part is symmetric, the
unstable-manifold theorem for maps then makes $R=\Id{k}$ unstable for the
nonlinear iteration as well.
\end{proof}

\begin{remark}[Reading \cref{thm:stability}]
\label{rem:reading-stability}
Four comments on the statement.
\begin{enumerate}[label=(\roman*)]
  \item \emph{What the regularity condition is for.}  The perturbation
        $\Delta_j := \sum_m\varepsilon_{jm}Z_m$ that appears in the proof is
        \emph{not} bounded almost surely -- only
        $\abs{\Delta_j}\le\norm{\varepsilon}\norm{Z}$ -- so the mean value
        theorem has to be applied at an intermediate point that can be far
        from $Z_j$.  A bound on $\eta_j''$ that is uniform on all of $\R$
        removes the difficulty, and $\Ex\norm{Z}^3<\infty$ then makes the
        Taylor remainder $O(\norm{\varepsilon}^2)$.
  \item \emph{What it covers.}  \Cref{eq:stability-regularity} holds for
        $\eta_j=-\tanh$ and $\eta_j=\tanh-\mathrm{id}$, but not for the cubic
        $\eta_j(u)=-u^3$, whose second derivative $\eta_j''(u)=-6u$ is
        unbounded.  For the cubic one uses instead the polynomial variant of
        the same estimate, valid whenever $\Ex\norm{Z}^4<\infty$: there
        $\Ex[\eta_j(C_j)C_l]$ is an explicit polynomial in $\varepsilon$ of
        degree at most four whose coefficients are moments of $Z$ of order at
        most four.
  \item \emph{The borderline.}  The only cases left undecided by the
        linearisation are $\gamma_j=1$ and $\zeta_j\zeta_l=1$ with
        $\zeta_j,\zeta_l>0$.  The apparent third borderline case
        $\zeta_j=0$ is decided: it gives $\zeta_j\zeta_l=0<1$, hence
        instability.
  \item \emph{The case $k=1$.}  There is then nothing to separate, and only
        the condition $\gamma_1<1$ survives.
\end{enumerate}
\end{remark}

\Cref{thm:stability} is the stability analysis of
\citet{amari1997stability}, recast in the notation of these notes; see also
\citet[Section~VI]{cardoso1998statistical}.

The quantity $\zeta_j$ therefore decides everything.  The next three results
compute it in the cases of interest.  The first explains, in one line, why
Gaussian sources are the exact borderline -- for \emph{every} choice of model
score, correctly specified or not.

\begin{proposition}[Gaussian sources lie exactly on the stability boundary]
\label{prop:gaussian-boundary}
In the setting of \cref{thm:stability}, suppose $Z_j\sim\Normal(0,\sigma_j^2)$
with $\sigma_j>0$, and suppose in addition that $\eta_j(z)\,p(z)\to0$ as
$\abs{z}\to\infty$, where $p$ denotes the $\Normal(0,\sigma_j^2)$ density
(see \cref{rem:score-hypotheses}).
Then $\zeta_j = 1$, whatever the model score $\eta_j$ -- correctly specified
or not.
Consequently, if two sources are Gaussian, then $\zeta_j\zeta_l=1$ and
\cref{eq:stability-conditions} fails.
\end{proposition}

\begin{proof}
Let $p$ be the $\Normal(0,\sigma_j^2)$ density, so that
$z\,p(z) = -\sigma_j^2p'(z)$.  Integration by parts, whose boundary term
$[\eta_j p]_{-\infty}^{\infty}$ vanishes by hypothesis, gives Stein's
identity
\begin{align}
  \Ex\bigl[\eta_j(Z_j)Z_j\bigr] = \int\eta_j(z)\,z\,p(z)\,dz
    = -\sigma_j^2\int\eta_j(z)p'(z)\,dz
    = \sigma_j^2\int\eta_j'(z)p(z)\,dz
    = \sigma_j^2\,\beta_j .
\end{align}
The normalisation \cref{eq:scale-normalisation} says that the left hand side
is $-1$, so $\zeta_j = -\beta_j\sigma_j^2 = 1$.
\end{proof}

\begin{corollary}[The true score: stability matches identifiability exactly]
\label{cor:true-score}
In the setting of \cref{thm:stability}, but requiring of each $\eta_j$ only
one continuous derivative and assuming neither \cref{eq:stability-regularity}
nor \cref{eq:scale-normalisation}, suppose the model is correctly specified,
that is
\begin{align}
  \label{eq:correctly-specified}
  \eta_j &= \tsc_j := \bigl(\log p_{Z_j}\bigr)', \qquad j=1,\dots,k,
\end{align}
where $p_{Z_j}$ is the density of $Z_j$.
Assume that each $p_{Z_j}$ is strictly positive and twice differentiable with
continuous second derivative, that
$\bigl(1+z^2\bigr)p_{Z_j}''\in L^1(\R)$, that the expectations in
\cref{eq:stability-moments,eq:scale-normalisation} converge absolutely, and
that the boundary conditions
\begin{align}
  \label{eq:score-boundary}
  z\,p_{Z_j}(z) &\longrightarrow 0,
  & z^2\,p_{Z_j}'(z) &\longrightarrow 0,
  & \abs{z}&\to\infty
\end{align}
hold.  Then \cref{eq:scale-normalisation} holds automatically,
\begin{align}
  \zeta_j &= \Fisher(Z_j)\,\sigma_j^2 \ \ge\ 1,
  & \Fisher(Z_j) &:= \Ex\bigl[\tsc_j(Z_j)^2\bigr],
\end{align}
with equality if and only if $Z_j$ is Gaussian, and $\gamma_j<1$.
Hence the inequalities \cref{eq:stability-conditions} hold if and only if at
most one of the sources is Gaussian.
\end{corollary}

\begin{proof}
Write $p := p_{Z_j}$ and $\eta := \eta_j = \tsc_j$, so that $\eta p = p'$
and $\eta' p = p'' - (p')^2/p$; it is \cref{eq:correctly-specified} that
makes the first of these identities available, and it is what turns each of
the expectations below into an integral of $p$ and its derivatives.
The second boundary condition in \cref{eq:score-boundary} forces
$p'(z)\to0$, and $\int p'' = \lim_{R\to\infty}\bigl(p'(R)-p'(-R)\bigr) = 0$;
the integral exists because $p''\in L^1$.
Normalisation: integration by parts, with vanishing boundary term by the
first condition in \cref{eq:score-boundary}, gives
$\Ex[\eta(Z_j)Z_j] = \int zp'(z)\,dz = -\int p = -1$.
For $\beta_j$,
\begin{align}
  \beta_j = \int\Bigl(\frac{p''}{p}-\frac{(p')^2}{p^2}\Bigr)p
    = \int p'' - \Ex\bigl[\eta(Z_j)^2\bigr] = -\Fisher(Z_j).
\end{align}
The inequality is Cauchy--Schwarz applied to the normalisation:
\begin{align}
  1 = \bigl(\Ex[\eta(Z_j)Z_j]\bigr)^2
    \le \Ex\bigl[\eta(Z_j)^2\bigr]\,\Ex\bigl[Z_j^2\bigr]
    = \Fisher(Z_j)\sigma_j^2 = \zeta_j ,
\end{align}
with equality if and only if $\eta(Z_j)$ and $Z_j$ are almost surely
proportional, i.e.\ $(\log p)'(z) = -z/\sigma_j^2$, i.e.\ $p$ is the
$\Normal(0,\sigma_j^2)$ density.
For $\gamma_j$, integrating by parts twice -- the first boundary term
vanishing by the second condition in \cref{eq:score-boundary} and the second
by the first condition -- we get $\int z^2p'' = -2\int zp' = 2$, so
\begin{align}
  \gamma_j = \int z^2\Bigl(\frac{p''}{p}-\eta^2\Bigr)p
    = 2 - \Ex\bigl[\eta(Z_j)^2Z_j^2\bigr]
    \le 2 - \bigl(\Ex[\eta(Z_j)Z_j]\bigr)^2 = 1,
\end{align}
again by Cauchy--Schwarz, with equality only if $\eta(Z_j)Z_j$ is almost
surely constant, i.e.\ $p(z)\propto\abs{z}^{-1}$, which is not integrable; so
$\gamma_j<1$ strictly.
Finally, $\zeta_j\ge1$ for all $j$ with equality exactly at the Gaussian
sources, so $\zeta_j\zeta_l>1$ for all $j\neq l$ if and only if at most one
$\zeta_j$ equals $1$.
\end{proof}

\begin{remark}[On the hypotheses of
\cref{prop:gaussian-boundary,cor:true-score}]
\label{rem:score-hypotheses}
The decay condition $\eta_j(z)p(z)\to0$ of \cref{prop:gaussian-boundary} is
what makes the boundary term in Stein's identity vanish; it holds whenever
$\eta_j$ has at most polynomial growth, and in particular for all three
nonlinearities of \cref{lem:fixed-point-scale}.
In \cref{cor:true-score} the second derivative of $p_{Z_j}$ is needed because
$\tsc_j{}' = p''/p-(p'/p)^2$, so \cref{thm:stability} presupposes it anyway,
and its continuity is what makes $\eta_j$ continuously differentiable;
\cref{thm:stability} asks for one derivative more.
Note also what \cref{cor:true-score} does and does not deliver: it gives the
\emph{inequalities} \cref{eq:stability-conditions}, and these coincide with
the hypothesis of \cref{cor:complete-ident}; but to turn them into local
stability one still needs the regularity \cref{eq:stability-regularity},
which a true score need not satisfy.
\end{remark}

\Cref{cor:true-score} is the conceptual answer to ``why does gradient descent
find the right answer''.  With correctly specified scores, $\eta_j = \tsc_j$,
the inequalities \cref{eq:stability-conditions} hold exactly when the
identifiability theory declares the model identifiable; subject to the
regularity caveat of \cref{rem:score-hypotheses}, a separating solution is
then a locally stable equilibrium of the mean dynamics.
The converse is false, and it is worth being explicit about it: \emph{stable
non-separating equilibria can exist even under a perfectly specified model}.
For $k=2$ with $Z_1,Z_2$ i.i.d.\ from the smooth, strictly positive,
non-Gaussian density $0.4\,\Normal(-1.5,0.25^2)+0.2\,\Normal(0,0.25^2)
+0.4\,\Normal(1.5,0.25^2)$ and $\eta_j = \tsc_j$ its exact score, the
matrix
$R^\star = 0.80993\cdot\tfrac{1}{\sqrt2}\bmat{1&1\\1&-1}$ is a stationary
point of \cref{eq:update-R} whose linearisation has all four eigenvalues
strictly negative, so it is locally attracting -- and it is not a generalised
permutation matrix, so $C=R^\star Z$ has dependent components.
\Cref{thm:stability} says which \emph{separating} solutions attract; it does
not say that nothing else does.
The remaining question is what happens when $\eta_j\neq\tsc_j$, which in
practice is always.

\begin{corollary}[The cubic nonlinearity: kurtosis is exactly the criterion]
\label{cor:cubic}
Let $\eta_j(u) := -u^3$ and let $Z_j$ have $\Ex[Z_j]=0$, $\Ex[Z_j^4]<\infty$
and be normalised as in \cref{lem:fixed-point-scale}, so that
$\Ex[Z_j^4]=1$.  Then $\gamma_j = -3$ and
\begin{align}
  \zeta_j &= 3\sigma_j^4 = \frac{3}{3+\kurt(Z_j)} .
\end{align}
Hence $\zeta_j>1$ if and only if $\kurt(Z_j)<0$, i.e.\ if and only if $Z_j$ is
sub-Gaussian in the sense of \cref{def:kurtosis}.
\end{corollary}

\begin{proof}
$\eta_j'(u) = -3u^2$, so $\beta_j = -3\sigma_j^2$ and
$\zeta_j = 3\sigma_j^4$, while $\gamma_j = -3\Ex[Z_j^4] = -3$.
By \cref{def:kurtosis}, $\kurt(Z_j) = \Ex[Z_j^4]/\sigma_j^4-3
= \sigma_j^{-4}-3$, so $\sigma_j^4 = (3+\kurt(Z_j))^{-1}$ and
$\zeta_j = 3/(3+\kurt(Z_j))$, which exceeds $1$ exactly when
$\kurt(Z_j)<0$.
\end{proof}

\begin{remark}[The hyperbolic tangent, and what the kurtosis rule really is]
\label{rem:tanh}
For the two $\tanh$-based model scores the criterion $\zeta_j>1$ becomes an
explicit inequality.  With the normalisation \cref{eq:scale-normalisation}:
\begin{itemize}
  \item $\eta_j(u) = -\tanh(u)$, so that
        $\Ex[\tanh(Z_j)Z_j] = 1$ and $\beta_j = -\Ex[\sech^2(Z_j)]$; the
        condition $\zeta_j>1$ reads
        \begin{align}
          \label{eq:tanh-condition}
          \Ex\bigl[Z_j^2\bigr]\,\Ex\bigl[\sech^2(Z_j)\bigr]
            &> \Ex\bigl[\tanh(Z_j)\,Z_j\bigr] .
        \end{align}
  \item $\eta_j(u) = \tanh(u)-u$, so that
        $\beta_j = -\Ex\bigl[\tanh^2(Z_j)\bigr]$; here the condition
        $\zeta_j>1$ reads
        \begin{align}
          \Ex\bigl[Z_j^2\bigr]\,\Ex\bigl[\tanh^2(Z_j)\bigr] &> 1 .
        \end{align}
\end{itemize}
\Cref{eq:tanh-condition} is a \emph{Stein discrepancy}: by the identity in the
proof of \cref{prop:gaussian-boundary}, its two sides are equal for every
Gaussian $Z_j$ and for every test function, so the inequality measures
departure from Gaussianity in the direction picked out by $\tanh$.
It is exactly the switching statistic of \citet[eq.~(2.28)]{lee1999extended}.

It is tempting -- and it is what the ICA literature usually says -- to
summarise this as ``use $-\tanh$ for super-Gaussian sources and $\tanh-\mathrm{id}$
for sub-Gaussian ones''.  That rule is a \emph{heuristic}, not a theorem: the
sign of $\kurt(Z_j)$ is exactly the criterion only for the cubic nonlinearity
(\cref{cor:cubic}), and for $\tanh$ the correct criterion is
\cref{eq:tanh-condition}, which must be checked.  It does hold for the
standard source models; \cref{tab:stability} lists the values.
In general it fails in both directions, and the failures are not exotic.
For the symmetric three-point family $\Prb[Y=\pm a]=p/2$, $\Prb[Y=0]=1-p$,
normalised as in \cref{lem:fixed-point-scale}, write $u := ca$ for the
rescaled atom; the normalisation reads $p\,u\tanh u = 1$ and one computes
\begin{align}
  \label{eq:three-point-zeta}
  \zeta &= \frac{u}{\tanh u} - 1,
  & \kurt &= u\tanh u - 3 .
\end{align}
So $\zeta = 1$ happens at $u = 2\tanh u$, i.e.\ at
$\kurt \approx -1.1664$, and every member of the family with
$-1.1664<\kurt<0$ is \emph{sub}-Gaussian yet has $\zeta>1$, so that $-\tanh$
separates it.
In the other direction, mass placed far out in units of $\sigma$ inflates
$\kurt$ while saturating $\tanh$: the four-atom law with $\Prb[Y=\pm M]=q/2$
and $\Prb[Y=\pm1]=(1-q)/2$ has $\kurt\approx451$ but $\zeta\approx0.46$ for
$q=5\cdot10^{-6}$, $M=100$, so two such \emph{super}-Gaussian sources make the
separating solution unstable for $-\tanh$.
Note finally that the sign of the discrepancy in \cref{eq:tanh-condition} is
scale dependent -- it is proportional to $\kurt$ only in the small-scale limit
-- so the criterion is meaningful only at the fixed-point scale of
\cref{lem:fixed-point-scale}; cf.\ \cref{rem:zeta-well-defined}.
\end{remark}

\begin{table}[htbp]
  \centering
  \footnotesize
  \setlength{\tabcolsep}{7pt}
  \renewcommand{\arraystretch}{1.25}
  \begin{tabular}{@{}lcccc@{}}
    \toprule
    Law of $Z_j$ & $\kurt(Z_j)$
      & $\zeta_j$ for $\eta=-\tanh$
      & $\zeta_j$ for $\eta=\tanh-\mathrm{id}$
      & $\zeta_j$ for $\eta(u)=-u^3$\\
    \midrule
    Gaussian            & $0$      & $1.000$ & $1.000$ & $1.000$\\
    Laplace             & $+3$     & $1.414$ & $0.726$ & $0.500$\\
    Student $t_5$       & $+6$     & $1.254$ & $0.808$ & $0.333$\\
    Student $t_8$       & $+1.5$   & $1.126$ & $0.895$ & $0.667$\\
    Logistic            & $+1.2$   & $1.134$ & $0.889$ & $0.714$\\
    Uniform             & $-1.2$   & $0.762$ & $1.273$ & $1.667$\\
    \bottomrule
  \end{tabular}
  \caption{The stability quantity $\zeta_j$ of \cref{eq:stability-moments} for
    the three standard nonlinearities.  Each law is taken in its standard
    normalisation -- $\Normal(0,1)$; Laplace and logistic with unit scale
    parameter; $t_\nu$ standard; uniform on $[-1,1]$ -- and is then rescaled
    by the unique factor $c>0$ of \cref{lem:fixed-point-scale}, which is what
    makes $\zeta_j$ well defined (\cref{rem:zeta-well-defined}); the
    tabulated values do not depend on the starting normalisation.
    Stability of a separating solution requires $\zeta_j\zeta_l>1$ for all
    $j\neq l$, for which $\zeta_j>1$ for all $j$ is sufficient; the remaining
    condition $\gamma_j<1$ is automatic here, since $\eta_j'\le0$ for all
    three nonlinearities and hence $\gamma_j\le0$.  Every Gaussian entry
    equals $1$ exactly, as \cref{prop:gaussian-boundary} predicts, and the
    cubic column is $3/(3+\kurt)$ exactly, as \cref{cor:cubic} predicts.
    The values were obtained by numerical quadrature.}
  \label{tab:stability}
\end{table}

\begin{remark}[What is approximate, and why it does not spoil convergence]
\label{rem:what-is-approximate}
It is worth separating three things that are easy to confuse.
\begin{enumerate}[label=(\roman*)]
  \item The step \cref{eq:update-final} is \emph{exact}: it is the gradient of
        the exact objective for the metric
        (\cref{prop:rel-gradient}).  Nothing is approximated there, and in
        particular no Hessian is being approximated
        (\cref{rem:not-a-hessian}).
  \item The model scores $\eta_j$ \emph{are} an approximation to the true
        scores $\tsc_j$, since the true source densities are unknown.  Their
        entire effect on whether the
        algorithm converges to a separating solution is captured by
        \cref{thm:stability}: a model score with $\eta_j\neq\tsc_j$ is
        harmless as long as
        $\gamma_j<1$, $\zeta_j>0$ and $\zeta_j\zeta_l>1$ -- and for the three
        nonlinearities of \cref{lem:fixed-point-scale} the first two are
        automatic, since $\eta_j'<0$ off a Lebesgue null set and $Z_j$ is
        non-degenerate, so $\gamma_j<0$ and $\zeta_j>0$.  This is why one only
        needs to know each source \emph{qualitatively}: in practice only the
        sign of $\zeta_j-1$, rather than the density.  For the source families
        of \cref{tab:stability} that sign is determined by the sign of the
        excess kurtosis, for each of the three nonlinearities separately, but
        the correspondence does \emph{not} hold in general -- see the warning
        and the counterexamples in \cref{rem:tanh}, and note that $\zeta_j$ is
        a statement about neither kurtosis nor tails.
  \item What $\eta_j\neq\tsc_j$ does cost is \emph{statistical efficiency}.
        The maximum likelihood estimator with the true scores attains the
        Cram\'er--Rao bound asymptotically; with misspecified scores the
        separating
        solutions remain stationary points, at the scale fixed by
        \cref{lem:fixed-point-scale}, and remain locally attracting whenever
        \cref{thm:stability} applies, but the asymptotic variance of the
        resulting estimator is larger.  This is the semiparametric picture of
        \citet{amari1997semiparametric}; see also \citet{cardoso1998statistical}.
\end{enumerate}
Two caveats should be stated plainly.  First, \cref{thm:stability} is a
\emph{local} statement.  The objective \cref{eq:ml-objective} is not concave in
$W$, and \cref{eq:update-R} has stationary points that are not separating,
some of which can themselves be attracting.  The theorem says when a
separating solution attracts, not that the algorithm finds one from an
arbitrary start.
Second, the online form \cref{alg:online-ica} below replaces the expectation in
\cref{eq:update-final} by a single-sample estimate.  It is thus a
Robbins--Monro stochastic approximation scheme \citep{robbins1951stochastic}
for the mean dynamics \cref{eq:update-R} -- provided the samples are i.i.d.\
across $t$, or at least mixing enough that the single-sample estimate is a
martingale difference plus a vanishing bias; with step sizes satisfying
$\sum_n\alpha_n=\infty$ and $\sum_n\alpha_n^2<\infty$, and under the usual
regularity and boundedness conditions, its trajectories track the associated
ordinary differential equation
\begin{align}
  \dot R &= \Bigl(\Id{k}+\Ex\bigl[\eta(RZ)(RZ)\T\bigr]\Bigr)R
\end{align}
and converge almost surely to one of its locally stable equilibria -- which
\cref{thm:stability} identifies.  A constant step size $\alpha$, as used in
practice, gives convergence only to a neighbourhood whose size is
$O(\alpha)$.
\end{remark}

\subsection{The algorithm}
\label{ssec:algorithm}

\begin{algo}[Equivariant online ICA for the complete noiseless model]
\label{alg:online-ica}
\emph{Input:} data $x(1),\dots,x(T)$ in $\R^k$, centred; a step size
$\alpha>0$; model scores $\eta_1,\dots,\eta_k$ chosen according to
\cref{thm:stability}, in practice as in \cref{rem:reading-algorithm};
an invertible initial $W\in\R^{k\times k}$, e.g.\ $W = \Id{k}$.

\emph{Repeat} over the data, until convergence: for each sample $x(t)$,
\begin{enumerate}[label=(\arabic*)]
  \item $c := W\,x(t)$ \hfill (current estimate of the sources)
  \item $y := \bigl[\eta_1(c_1),\dots,\eta_k(c_k)\bigr]\T$
        \hfill (nonlinearity applied componentwise)
  \item $W \leftarrow W + \alpha\bigl(W + y\,c\T W\bigr)
         = W + \alpha\bigl(\Id{k}+y\,c\T\bigr)W$ .
\end{enumerate}

\emph{Output:} the unmixing matrix $W$, and the reconstructed sources
obtained by running step~(1) once more with the final $W$, that is
$c(t) = W\,x(t)$; they are determined up to permutation and scale.
\end{algo}

\begin{remark}[Reading the algorithm]
\label{rem:reading-algorithm}
In practice one takes $\eta_j = -\tanh$ for super-Gaussian and
$\eta_j = \tanh-\mathrm{id}$ for sub-Gaussian sources, with the caveats of
\cref{rem:tanh}.
Step (3) is \cref{eq:update-final} with the expectation replaced by the
current sample, since $\Ex[\eta(C)C\T]$ is estimated by $y\,c\T$.
Written out as $W + \alpha(W + \eta(Wx)\,x\T W\T W)$ it is exactly the update
of \citet{amari1996newlearning}, \citet{cardoso1996equivariant} and
\citet{mackay1996ica} -- see also \citet[Chapter~34]{mackay2003itila} for
the same derivation in an archival source; the infomax algorithm of \citet{bell1995infomax} is the
same update in its plain-gradient form, and the two objectives coincide, as
observed by \citet{cardoso1997infomax}.
Each step costs $O(k^2)$ operations and no matrix inversion or decomposition:
$y\,c\T W$ is computed as $y\,(c\T W)$, a vector--matrix product followed by
an outer product, each $O(k^2)$.
By \cref{prop:equivariance} the whole dynamics is a dynamics on the global
system matrix $R = WA$, so the mixing matrix enters only through
$R_0 = W_0A$: once started, the algorithm behaves identically for every $A$
giving the same $R_0$, however badly conditioned $A$ itself is.  This is what
distinguishes the update from plain gradient ascent on $\mathcal{L}$, whose
trajectory depends on $A$ throughout.
\end{remark}

\subsection{LiNGAM: a causal order removes the permutation}
\label{ssec:lingam}

\Cref{cor:complete-ident} leaves the sources determined only up to
permutation, sign and scale.  In causal discovery one adds structural
assumptions that remove exactly those remaining ambiguities.  The basic such
model is the linear non-Gaussian acyclic model, LiNGAM, of
\citet{shimizu2006}.
We write $\Theta$ for the matrix of structural coefficients, which
\citet{shimizu2006} call $B$; the letter $B$ is already taken in these notes,
and $R$ is the global system matrix of \cref{prop:equivariance}.

\begin{assumption}[LiNGAM]
\label{asm:lingam}
The observed random vector $X\in\R^k$ satisfies the structural equations
\begin{align}
  \label{eq:lingam-model}
  X &= \Theta\,X + Z,
\end{align}
where $\Theta\in\R^{k\times k}$ and the disturbance vector $Z\in\R^{k}$
satisfy:
\begin{enumerate}[label=(\roman*)]
  \item \emph{Acyclicity.}  There is a bijection
        $\pi\colon\{1,\dots,k\}\to\{1,\dots,k\}$, the \emph{causal order},
        such that $\Theta_{ab} = 0$ whenever $\pi(a) \le \pi(b)$.  Equivalently,
        $\Theta$ becomes strictly lower triangular after simultaneously permuting
        rows and columns by $\pi$; the associated directed graph, with an
        edge $b\to a$ whenever $\Theta_{ab}\neq0$, is acyclic.
  \item \emph{Independent disturbances, causal sufficiency.}  The components
        $Z_1,\dots,Z_k$ of the disturbance vector are mutually independent and
        almost surely non-constant.
  \item \emph{Non-Gaussianity.}  At most one $Z_j$ is Gaussian.
\end{enumerate}
\end{assumption}

\begin{lemma}[LiNGAM is a complete noiseless ICA model]
\label{lem:lingam-is-ica}
Under \cref{asm:lingam}\,(i) the matrix $\Id{k}-\Theta$ is invertible with
$\det(\Id{k}-\Theta) = 1$, and \cref{eq:lingam-model} is equivalent to
\begin{align}
  \label{eq:lingam-as-ica}
  X &= A\,Z, & A &:= \bigl(\Id{k}-\Theta\bigr)^{-1},
  & W &:= A^{-1} = \Id{k}-\Theta .
\end{align}
This is a complete noiseless ICA model in the sense of
\cref{asm:ica}\,(i)--(iv), and the unmixing matrix $W$ has all diagonal
entries equal to $1$.
\end{lemma}

\begin{proof}
Let $P_\pi$ be the permutation matrix of the causal order, so that
$N := P_\pi \Theta P_\pi\T$ is strictly lower triangular and hence $N^k = 0$;
then $\Theta^k = P_\pi\T N^kP_\pi = 0$, so $\Theta$ is nilpotent,
$\Id{k}-\Theta$ is invertible with inverse $\sum_{j=0}^{k-1}\Theta^j$, and
$\det(\Id{k}-\Theta) = \det(\Id{k}-N) = 1$ because $\Id{k}-N$ is lower triangular
with unit diagonal.  Rearranging \cref{eq:lingam-model} gives
$(\Id{k}-\Theta)X = Z$, which is \cref{eq:lingam-as-ica}.  The diagonal of
$W = \Id{k}-\Theta$ is $1$ because $\Theta_{aa} = 0$ by (i) with $b=a$.
\end{proof}
The normalisation \cref{asm:ica}\,(v) is deliberately \emph{not} imposed
here: LiNGAM fixes the scale of the sources through $\diag(W) = \Id{k}$
rather than through either of the conventions of \cref{rem:ica-normalisation}.

\begin{corollary}[Identifiability of LiNGAM]
\label{cor:lingam}
Suppose $X$ satisfies \cref{asm:lingam} with two parameter sets, that is
\begin{align}
  \Theta\up{1}X + Z\up{1} \;=\; X \;=\; \Theta\up{2}X + Z\up{2},
\end{align}
where each $\Theta\up{i}$ satisfies (i) and each $Z\up{i}$ satisfies (ii) and
(iii).  Then
\begin{align}
  \Theta\up{1} &= \Theta\up{2}, & Z\up{1} &= Z\up{2} \quad\text{almost surely}.
\end{align}
\end{corollary}

\begin{proof}
By \cref{lem:lingam-is-ica} we have two representations $X = A\up{i}Z\up{i}$
with $A\up{i} = (\Id{k}-\Theta\up{i})^{-1}$ invertible, so
\cref{cor:complete-ident} applies (with $\mu\up{i}=0$) and yields a
permutation matrix $P$ and an invertible diagonal $\Lambda$ with
$A\up{2} = A\up{1}P\Lambda$, together with
$Z\up{1} = P\Lambda Z\up{2}+c$.
Inverting, and writing $W\up{i} := \Id{k}-\Theta\up{i}$,
\begin{align}
  \label{eq:lingam-W-relation}
  W\up{2} &= \Lambda^{-1}P^{-1}W\up{1} .
\end{align}
The matrix $M := \Lambda^{-1}P^{-1}$ is a generalised permutation matrix:
there are a permutation $\sigma$ and non-zero scalars $m_1,\dots,m_k$ with
$M_{ac} = m_a\delta_{c,\sigma(a)}$, so that \cref{eq:lingam-W-relation} reads
\begin{align}
  \label{eq:lingam-rows}
  W\up{2}_{ab} &= m_a\,W\up{1}_{\sigma(a),b},
  \qquad a,b = 1,\dots,k .
\end{align}
Both $W\up{i}$ have unit diagonal by \cref{lem:lingam-is-ica}, so taking
$b=a$ in \cref{eq:lingam-rows} gives
\begin{align}
  \label{eq:lingam-diagonal}
  1 &= W\up{2}_{aa} = m_a\,W\up{1}_{\sigma(a),a},
  \qquad\text{hence}\qquad W\up{1}_{\sigma(a),a} \neq 0 \ \text{ for all } a .
\end{align}
Now use acyclicity of $\Theta\up{1}$: for $c\neq b$ we have
$W\up{1}_{cb} = -\Theta\up{1}_{cb}$, which is non-zero only if
$\pi_1(c) > \pi_1(b)$, where $\pi_1$ is the causal order of $\Theta\up{1}$.
So \cref{eq:lingam-diagonal} forces, for every $a$,
\begin{align}
  \text{either}\quad \sigma(a) = a
  \qquad\text{or}\qquad \pi_1\bigl(\sigma(a)\bigr) > \pi_1(a).
\end{align}
Suppose $\sigma\neq\mathrm{id}$ and let $a$ lie on a non-trivial cycle of
$\sigma$, of length $L\ge2$.  Then $\sigma^{j}(a)\neq\sigma^{j+1}(a)$ for
every $j$, so
\begin{align}
  \pi_1(a) < \pi_1\bigl(\sigma(a)\bigr) < \dots
    < \pi_1\bigl(\sigma^{L}(a)\bigr) = \pi_1(a),
\end{align}
a contradiction.  Hence $\sigma = \mathrm{id}$, and then
\cref{eq:lingam-diagonal} gives $m_a = 1/W\up{1}_{aa} = 1$, so $M = \Id{k}$.
Therefore $W\up{2} = W\up{1}$, i.e.\ $\Theta\up{2} = \Theta\up{1}$, and consequently
$Z\up{2} = W\up{2}X = W\up{1}X = Z\up{1}$.
\end{proof}
So the entire structural model -- graph, coefficients and disturbance
distributions -- is determined by the law of $X$; the causal order itself is
identified up to the ordering of variables that are not ancestrally related,
which is exactly the freedom a topological sort leaves.

\begin{remark}[What each assumption does]
\label{rem:lingam-assumptions}
It is worth seeing which ambiguity of \cref{cor:complete-ident} is removed by
which structural assumption.
The scale ambiguity $\Lambda$ is removed by the \emph{form} of
\cref{eq:lingam-model}: the coefficient of $X_a$ in its own equation is fixed
to $1$, which is what makes $W$ have unit diagonal.
The permutation ambiguity $P$ is removed by \emph{acyclicity}, through the
cycle argument above: a non-trivial relabelling would have to increase the
causal order all the way around a cycle.
\Cref{asm:lingam}\,(ii) is causal sufficiency: mutual independence of the
disturbances says exactly that there are no hidden common causes.
Non-Gaussianity is what makes \cref{cor:complete-ident} available in the first
place; without it one recovers the graph at best up to its Markov equivalence
class, and then only under a faithfulness assumption.
Note also that \citet{shimizu2006} assume \emph{all} disturbances
non-Gaussian, whereas \cref{cor:lingam} needs only ``at most one'', because
that is all \cref{cor:complete-ident} needs.
\end{remark}

\begin{algo}[LiNGAM by equivariant ICA, after \citealp{shimizu2006}]
\label{alg:lingam}
\emph{Input:} data $x(1),\dots,x(T)$ in $\R^k$.

\begin{enumerate}[label=(\arabic*)]
  \item Centre the data.
  \item Run \cref{alg:online-ica} to obtain an estimate $\widehat W$ of the
        unmixing matrix.  By \cref{cor:complete-ident} it estimates
        $\Lambda^{-1}P\T W$ for the true $W = \Id{k}-\Theta$ and some unknown
        generalised permutation.
  \item \emph{Undo the permutation.}  Find the row permutation $P_1$ that
        makes all diagonal entries of $P_1\widehat W$ non-zero; in practice
        one minimises $\sum_{j}\bigl|(P_1\widehat W)_{jj}\bigr|^{-1}$ over
        permutations, which is a linear assignment problem.  Put
        $W' := P_1\widehat W$.
  \item \emph{Undo the scaling.}  Divide each row of $W'$ by its own diagonal
        entry, that is, put $D := \diag(W'_{11},\dots,W'_{kk})$ and
        $W'' := D^{-1}W'$, so that $W''$ has unit diagonal.
  \item \emph{Read off the coefficients.} $\widehat \Theta := \Id{k}-W''$.
  \item \emph{Read off the causal order.}  Find the permutation $\pi$ making
        $\widehat \Theta$ as close to strictly lower triangular as possible, e.g.\
        by repeatedly setting the smallest-magnitude entries to zero until a
        causal order exists.  Optionally prune the remaining small
        coefficients.
\end{enumerate}

\emph{Output:} the coefficient matrix $\widehat \Theta$, the causal order $\pi$,
and the estimated disturbances $\widehat z(t) := W''\,x(t)$.
\end{algo}

\begin{remark}[Why steps (3) and (6) are well posed]
\label{rem:lingam-steps}
In the population limit, where $\widehat W = \Lambda^{-1}P\T W$ exactly, the
proof of \cref{cor:lingam} is precisely the statement that step~(3) has a
solution and that it is unique: the cycle argument shows that
$P_1 = P$ is the only row permutation producing a nowhere-zero diagonal, and
step~(4) then recovers $W$ exactly.
Likewise step~(6) succeeds because the true $\widehat \Theta$ is permutable to
strictly lower triangular form.
Note also that step~(2) invokes \cref{alg:online-ica}, which was derived
under the standing assumption of \cref{ssec:ml} that each source has a
strictly positive differentiable density; \cref{asm:lingam} does not require
this, so the assumption has to be added if the ICA-based route is taken.
With finite samples neither search is exact, which is why both become
combinatorial problems with tolerances; see
\citet[Sections~4--5]{shimizu2006} for the details and
\citet{shimizu2011directlingam} for \emph{DirectLiNGAM}, a later algorithm
that estimates the causal order directly by iterated regression and
independence testing, and so avoids the ICA step -- and with it both the local
convergence caveats of \cref{rem:what-is-approximate} and the two
combinatorial searches above.
\end{remark}

\subsection{Generalisations: a brief overview}
\label{ssec:generalisations}

The model of these notes -- one linear mixture, instantaneous, with
independent components -- has been extended in many directions.  This
subsection is a short guide to the landscape, with references rather than
derivations.

\paragraph{Relaxing linearity.}
The obvious generalisation replaces the mixing matrix by a diffeomorphism,
$X = f(Z)$ with $Z$ having independent components.
This model is \emph{never} identifiable: \citet{hyvarinen1999nonlinear} show
that any random vector $X$ with a positive density admits a representation
$X = f(Z)$ with mutually independent $Z_j$, obtained from the
conditional-quantile construction of \citet{darmois1953}, which returns
uniform sources; composing that solution with the measure-preserving
automorphisms of the unit cube then produces infinitely many further ones.
Identifiability has to be bought back with extra structure.  The dominant
device is an observed auxiliary variable $U$ -- a time index, a segment label,
an experimental condition -- with respect to which the sources are
conditionally independent and sufficiently variable,
\begin{align}
  \label{eq:conditional-factorisation}
  p_{Z\mid U}(z\mid u) &= \prod_{j=1}^{k}p_{Z_j\mid U}(z_j\mid u).
\end{align}
Under such assumptions the nonlinear model becomes identifiable up to
permutation and componentwise transformations; see \citet{hyvarinen2016tcl}
for
time-contrastive learning, where $U$ is a time segment label;
\citet{hyvarinen2017pcl} for temporally dependent sources, where the role of
$U$ is played by the previous time point; and \citet{khemakhem2020ivae} for
the identifiable variational autoencoder, which unifies these results.
A different route constrains the geometry of $f$ rather than the statistics of
$Z$: \emph{independent mechanism analysis} \citep{gresele2021ima} requires the
columns of the Jacobian of $f$ to be orthogonal, which by
\citet[Theorem~4.7]{gresele2021ima} excludes the Darmois solutions of
\citet{hyvarinen1999nonlinear} without an auxiliary variable.  It is not, however, known to deliver full
identifiability on its own.

\paragraph{Relaxing independence.}
If the sources fall into groups that are independent \emph{across} groups but
dependent \emph{within}, one obtains multidimensional ICA, or independent
subspace analysis \citep{cardoso1998multidim}: the mixing matrix is identified
only up to a permutation of equal-dimensional blocks and an invertible
transformation within each block -- the analogue, for a general block, of the
orthogonal ambiguity inside a Gaussian block in \cref{rem:two-gaussians}.

\paragraph{Several datasets at once.}
Given $M$ related datasets $X\up{m} = A\up{m}Z\up{m}$, $m=1,\dots,M$, one may,
in addition to independence within each dataset, tie corresponding sources
across datasets.  Independent vector analysis \citep{kim2006iva} keeps the
components independent within a dataset but couples the $j$-th components
across datasets, $j=1,\dots,k$, into a single source component vector; this
aligns the permutations across datasets, leaving only one global permutation
of the $k$ component vectors.
Multi-view ICA goes further and splits the sources into a part shared by all
views and parts specific to each,
\begin{align}
  \label{eq:multiview-ica}
  X\up{m} &= A\up{m}
     \bmat{S\\ Z\up{m}} + E\up{m},
  \qquad m = 1,\dots,M,
\end{align}
with a shared source vector $S$ and view-specific sources $Z\up{m}$.
In \citet{pandeva2023multiview} each $A\up{m}$ is square and invertible and
the noise is isotropic Gaussian \emph{in the latent space}, so that
$E\up{m} = A\up{m}\epsilon\up{m}$ with
$\epsilon\up{m}\sim\Normal(0,\sigma^2\Id{k_m})$ and the observed noise
covariance is $\sigma^2A\up{m}A\up{m\T}$.  See that paper for identifiability
results and an estimation procedure, \citet{pandeva2023omics} for its use in
integrating omics modalities, and \citet{pandeva2025robust} for a related
multi-view latent-variable model that replaces source independence by a shared
sparse precision matrix and targets co-expression network inference.

\paragraph{Noise and overcompleteness.}
The noisy and overcomplete models -- $X = AZ+E$ with $A\in\R^{p\times k}$ and
possibly $k>p$ -- are exactly the setting of
\cref{sec:non-gaussian,sec:gaussian-free}; note that identifiability of $A$
there does \emph{not} entail that the sources can be recovered pointwise.
Additive Gaussian noise leaves them determined only up to a Gaussian summand,
by \cref{thm:non-gaussian}, even when $A$ has a left inverse; and in the
overcomplete case $k>p$ the matrix $A$ has no left inverse at all.

\paragraph{Causal discovery.}
Beyond LiNGAM, dropping linearity while keeping additive noise gives the
additive noise models of \citet{hoyer2008anm}, in which
$X_a = f_a(X_{\mathrm{pa}(a)}) + Z_a$ and identifiability again comes from a
non-Gaussianity- or nonlinearity-induced asymmetry between cause and effect.

\paragraph{Further reading.}
\citet{hyvarinen2001} remains the standard textbook; \citet{comon2010handbook}
is a comprehensive edited handbook covering algorithms, identifiability and
applications.

\appendix

\section{Proofs of the Classical Results on Characteristic Functions}
\label{app:cf-proofs}

This appendix proves the four results of \cref{sec:cf} that were stated there
without proof: the uniqueness and inversion theorem
(\cref{thm:uniqueness}), L\'evy's continuity theorem (\cref{thm:levy}),
Marcinkiewicz' theorem (\cref{thm:marcinkiewicz}) and Cram\'er's decomposition
theorem (\cref{thm:cramer}).
Together with \cref{app:kagan-proof}, which proves \cref{thm:kagan}, this
makes every proof in these notes refer only to other results of these notes,
together with the following standard background, which we use without proof:
\begin{enumerate}[label=(\alph*)]
  \item measure theory: monotone and dominated convergence (hence
        differentiation under the integral sign), Fatou's lemma, Fubini's
        theorem, and the fact that a Borel probability measure on $\R^d$ is
        determined by its integrals against continuous functions of compact
        support;
  \item Prokhorov's theorem: a tight sequence of Borel probability measures on
        $\R^d$ has a weakly convergent subsequence
        \citep[Chapter~13]{klenke2020}, \citep[Chapter~5]{kallenberg2021};
  \item complex analysis of one variable: Morera's theorem, the identity
        theorem, the mean value property of harmonic functions, and the fact
        that a nowhere vanishing entire function is $\exp$ of an entire
        function;
  \item elementary real analysis: Taylor's theorem with Peano remainder,
        Jensen's and Lyapunov's inequalities, existence of a median, the
        orthogonality relations for the trigonometric system together with
        termwise integration of uniformly convergent series, and the
        sub-subsequence criterion for convergence of real sequences.
\end{enumerate}
\Cref{lem:moments-from-derivatives} below supplies the converse half of
\cref{prop:moments}.  Its proof \emph{does} use the forward half; but the
forward half is proved where it is stated and does not use the converse, so
there is no circularity.
Two further standard facts are used outside \cref{sec:cf} and its appendix:
Osgood's lemma in the several-variable half of \cref{thm:analytic-strip}
(which is never applied), and, in \cref{sec:algorithm}, matrix calculus --
Jacobi's formula for $\partial\log\abs{\det W}$ -- together with the
unstable-manifold theorem for maps in \cref{thm:stability}.
Apart from these, the only \emph{result} that a proof in these notes uses
without our having proved it is Bochner's \cref{thm:bochner}, and it is used
exactly once, in \cref{rem:local-not-enough}, where it could be replaced by an
appeal to P\'olya's criterion or by a direct Fourier computation.  (Two more
classical facts are quoted inside remarks -- P\'olya's criterion itself and
the Lukacs characterisation in \cref{rem:ridge} -- but no proof depends on
them.)

\subsection{Uniqueness and inversion}
\label{ssec:proof-uniqueness}

\begin{lemma}[Gaussian Fourier identity]
\label{lem:gauss-fourier}
For every $\varepsilon>0$ and every $x\in\R$,
\begin{align}
  \label{eq:gauss-fourier}
  \int_{\R}e^{-itx}\exp\Bigl(-\tfrac12\varepsilon^2t^2\Bigr)\,dt
    &= \frac{\sqrt{2\pi}}{\varepsilon}
       \exp\Bigl(-\frac{x^2}{2\varepsilon^2}\Bigr) .
\end{align}
Consequently, for $x\in\R^d$,
\begin{align}
  \label{eq:gauss-fourier-d}
  (2\pi)^{-d}\int_{\R^d}e^{-i\,t\T x}
    \exp\Bigl(-\tfrac12\varepsilon^2\norm{t}^2\Bigr)\,dt
    &= \gamma_\varepsilon(x)
     := \bigl(2\pi\varepsilon^2\bigr)^{-d/2}
        \exp\Bigl(-\frac{\norm{x}^2}{2\varepsilon^2}\Bigr),
\end{align}
the Lebesgue density of $\Normal(0,\varepsilon^2\Id{d})$.
\end{lemma}

\begin{proof}
Write $I(x)$ for the left hand side of \cref{eq:gauss-fourier}.
Differentiation under the integral sign is legitimate, since
$\abs{t}\exp(-\tfrac12\varepsilon^2t^2)$ is integrable, and an integration by
parts gives
\begin{align}
  I'(x) &= \int_\R(-it)e^{-itx}e^{-\varepsilon^2t^2/2}\,dt
         = \frac{i}{\varepsilon^2}\int_\R e^{-itx}
           \frac{d}{dt}\Bigl(e^{-\varepsilon^2t^2/2}\Bigr)dt
         = -\frac{x}{\varepsilon^2}\,I(x),
\end{align}
the boundary terms vanishing.  Hence
$I(x) = I(0)\exp(-x^2/(2\varepsilon^2))$, and
$I(0) = \int_\R e^{-\varepsilon^2t^2/2}dt = \sqrt{2\pi}/\varepsilon$.
This is \cref{eq:gauss-fourier}; \cref{eq:gauss-fourier-d} follows by taking
the product over the $d$ coordinates.
\end{proof}

\begin{proof}[Proof of \cref{thm:uniqueness}]
Let $X$ have law $\mu$ and let $G\sim\Normal(0,\Id{d})$ be independent of $X$,
both defined on one probability space.  Fix $\varepsilon>0$ and put
$X_\varepsilon := X+\varepsilon G$.
Conditioning on $X$ shows that $X_\varepsilon$ has the Lebesgue density
$f_\varepsilon(x) = \int_{\R^d}\gamma_\varepsilon(x-y)\,\mu(dy)$.
Inserting \cref{eq:gauss-fourier-d} and applying Fubini's theorem -- legitimate
because the integrand has modulus $\exp(-\tfrac12\varepsilon^2\norm{t}^2)$,
which is integrable on $\R^d\times\R^d$ against the probability measure $\mu$
-- we obtain
\begin{align}
  \label{eq:smoothed-density}
  f_\varepsilon(x)
    &= (2\pi)^{-d}\int_{\R^d}\int_{\R^d}
       e^{-i\,t\T(x-y)}e^{-\varepsilon^2\norm{t}^2/2}\,dt\,\mu(dy)\\
    &= (2\pi)^{-d}\int_{\R^d}e^{-i\,t\T x}
       e^{-\varepsilon^2\norm{t}^2/2}\,\cf{\mu}(t)\,dt .
\end{align}
So $f_\varepsilon$, and hence the law of $X_\varepsilon$, is determined by
$\cf{\mu}$ alone.

Now let $g\colon\R^d\to\R$ be bounded and continuous.  Since
$X_\varepsilon\to X$ pointwise on the underlying probability space as
$\varepsilon\downarrow0$, dominated convergence gives
\begin{align}
  \label{eq:smoothing-limit}
  \int_{\R^d} g(x)f_\varepsilon(x)\,dx = \Ex\bigl[g(X_\varepsilon)\bigr]
    \xrightarrow[\varepsilon\downarrow0]{} \Ex\bigl[g(X)\bigr]
    = \int_{\R^d}g\,d\mu .
\end{align}
The left hand side depends on $\mu$ only through $\cf{\mu}$, so $\int g\,d\mu$
is determined by $\cf{\mu}$ for every bounded continuous $g$, and therefore
$\mu$ is determined by $\cf{\mu}$.  This proves part~(i).

For part~(ii) assume in addition $\cf{\mu}\in L^1(\R^d)$ and define
\begin{align}
  f(x) &:= (2\pi)^{-d}\int_{\R^d}e^{-i\,t\T x}\,\cf{\mu}(t)\,dt .
\end{align}
Then $f$ is bounded by $(2\pi)^{-d}\norm{\cf{\mu}}_{L^1}$ and continuous, by
dominated convergence.  Comparing with \cref{eq:smoothed-density},
\begin{align}
  \sup_{x\in\R^d}\bigabs{f_\varepsilon(x)-f(x)}
    &\le (2\pi)^{-d}\int_{\R^d}
       \Bigabs{e^{-\varepsilon^2\norm{t}^2/2}-1}\,\abs{\cf{\mu}(t)}\,dt
    \xrightarrow[\varepsilon\downarrow0]{} 0
\end{align}
by dominated convergence, with dominating function $2\abs{\cf{\mu}}$.
Since every $f_\varepsilon$ is a probability density and
$f_\varepsilon\to f$ pointwise, the limit $f$ is real valued and $f\ge0$, so
$\nu(dx):=f(x)\,dx$ defines a Borel measure on $\R^d$.
Let now $g$ be continuous with compact support.  Uniform convergence on that
compact support gives $\int gf_\varepsilon\,dx\to\int gf\,dx$, while
\cref{eq:smoothing-limit} gives $\int gf_\varepsilon\,dx\to\int g\,d\mu$, so
that
\begin{align}
  \label{eq:cc-agreement}
  \int_{\R^d} g\,d\mu &= \int_{\R^d} g\,d\nu,
  \qquad g\in C_c(\R^d).
\end{align}
Choosing $g_k\in C_c(\R^d)$ with $0\le g_k\uparrow1$ pointwise and applying
monotone convergence on both sides of \cref{eq:cc-agreement} gives
$\nu(\R^d) = \mu(\R^d) = 1$, so $\nu$ is a Borel probability measure.
By \cref{eq:cc-agreement} and background fact~(a), $\mu=\nu$; that is, $\mu$
has the density $f$.
\end{proof}

\subsection{L\'evy's continuity theorem}
\label{ssec:proof-levy}

\begin{lemma}[Tail bound from the characteristic function near the origin]
\label{lem:tightness-estimate}
Let $Y$ be a real-valued random variable and $\delta>0$.  Then
\begin{align}
  \label{eq:tightness-estimate}
  \Prb\Bigl[\abs{Y}\ge\frac{2}{\delta}\Bigr]
    &\le \frac{1}{\delta}\int_{-\delta}^{\delta}
         \bigl(1-\Real\cf{Y}(t)\bigr)\,dt .
\end{align}
\end{lemma}

\begin{proof}
By Fubini's theorem, using $\Real\cf{Y}(t) = \Ex[\cos(tY)]$,
\begin{align}
  \int_{-\delta}^{\delta}\bigl(1-\Real\cf{Y}(t)\bigr)dt
    &= 2\int_0^\delta\Ex\bigl[1-\cos(tY)\bigr]dt
     = 2\delta\,\Ex\left[1-\frac{\sin(\delta Y)}{\delta Y}\right],
\end{align}
with the convention $\sin(u)/u := 1$ at $u=0$.
The integrand $1-\sin(u)/u$ is non-negative for all $u\in\R$, and for
$\abs{u}\ge2$ it is at least $1-1/\abs{u}\ge\tfrac12$.
Hence
\begin{align}
  2\delta\,\Ex\left[1-\frac{\sin(\delta Y)}{\delta Y}\right]
    &\ge 2\delta\cdot\tfrac12\,\Prb\bigl[\abs{\delta Y}\ge2\bigr]
     = \delta\,\Prb\Bigl[\abs{Y}\ge\frac{2}{\delta}\Bigr],
\end{align}
which is \cref{eq:tightness-estimate}.
\end{proof}

\begin{proof}[Proof of \cref{thm:levy}]
(i) If $X_n\Rightarrow X$ then, for each fixed $t$, the functions
$x\mapsto\cos(t\T x)$ and $x\mapsto\sin(t\T x)$ are bounded and continuous, so
$\cf{X_n}(t) = \Ex[\cos(t\T X_n)]+i\,\Ex[\sin(t\T X_n)]$ converges to
$\cf{X}(t)$.

(ii) Assume $\cf{X_n}\to\phi$ pointwise with $\phi$ continuous at $0$.  Since
$\cf{X_n}(0)=1$ for all $n$ we have $\phi(0)=1$.

\emph{Tightness.}  Let $\varepsilon>0$.  By continuity of $\phi$ at $0$ and
$\phi(0)=1$ we may choose $\delta>0$ so small that
\begin{align}
  \frac{1}{\delta}\int_{-\delta}^{\delta}
     \bigl(1-\Real\phi(se_j)\bigr)\,ds &< \frac{\varepsilon}{2d},
     \qquad j=1,\dots,d .
\end{align}
By \cref{eq:cf-projection} the characteristic function of the $j$-th
coordinate $e_j\T X_n$ is $s\mapsto\cf{X_n}(se_j)$, and
$\abs{1-\Real\cf{X_n}(se_j)}\le2$, so dominated convergence on the bounded
interval $[-\delta,\delta]$ gives
\begin{align}
  \frac{1}{\delta}\int_{-\delta}^{\delta}
     \bigl(1-\Real\cf{X_n}(se_j)\bigr)ds
  \xrightarrow[n\to\infty]{}
  \frac{1}{\delta}\int_{-\delta}^{\delta}
     \bigl(1-\Real\phi(se_j)\bigr)ds < \frac{\varepsilon}{2d}.
\end{align}
By \cref{lem:tightness-estimate} there is therefore an $N$ with
$\Prb[\abs{e_j\T X_n}\ge2/\delta]<\varepsilon/d$ for all $n\ge N$ and all
$j = 1,\dots,d$, whence
$\Prb[\norm{X_n}_\infty\ge2/\delta]<\varepsilon$ for $n\ge N$.
Enlarging the compact set $[-2/\delta,2/\delta]^d$ so as to accommodate the
finitely many laws of $X_1,\dots,X_{N-1}$ as well, we conclude that
$(\Law(X_n))_{n\in\N}$ is tight.

\emph{Identification of the limit.}  Let $(X_{n_l})_l$ be an arbitrary
subsequence.  By Prokhorov's theorem it has a further subsequence converging
weakly to the law of some random vector $X$.  By part~(i) the characteristic
function of that limit is $\lim_l\cf{X_{n_l}} = \phi$; in particular $\phi$ is
a characteristic function.  By \cref{thm:uniqueness}\,(i) every subsequential
weak limit has the same law, namely the one with characteristic function
$\phi$.  Now fix a bounded continuous $g\colon\R^d\to\R$.  By the previous paragraph,
every subsequence of the real sequence $\bigl(\Ex[g(X_n)]\bigr)_{n}$ has a
further subsequence converging to $\Ex[g(X)]$, and a real sequence with that
property converges to $\Ex[g(X)]$.  As $g$ was an arbitrary bounded continuous
function, $X_n\Rightarrow X$.
\end{proof}

\subsection{Entire characteristic functions}
\label{ssec:entire-cf}

Both remaining proofs rest on the following observation: a characteristic
function cannot be analytic ``by accident''.  As soon as it agrees with an
entire function near the origin, the underlying law has Gaussian-type
integrability.

\begin{lemma}[Moments from differentiability at the origin]
\label{lem:moments-from-derivatives}
Let $Y$ be a real-valued random variable, $n\in\N$, and suppose that
$\cf{Y}$ is $2n$ times differentiable at $t=0$.  Then $\Ex[Y^{2n}]<\infty$.
\end{lemma}

\begin{proof}
We induct on $n$, the case $n=0$ being trivial.
Assume $\cf{Y}^{(2n+2)}(0)$ exists.  Then $\cf{Y}^{(2n+1)}$ exists on a
neighbourhood of the origin, so $\cf{Y}$ is in particular $2n$ times
differentiable at $0$ and the induction hypothesis gives
$\Ex[Y^{2n}]<\infty$.
By \cref{prop:moments} the function $\cf{Y}$ is $2n$ times continuously
differentiable on $\R$ with
$\cf{Y}^{(2n)}(t) = (-1)^n\,\Ex[Y^{2n}e^{itY}]$.
Put $c := \Ex[Y^{2n}]$.  If $c=0$ then $Y=0$ almost surely and there is
nothing to prove, so assume $c>0$ and let $\varrho$ be the probability measure
with $\varrho(dy) := c^{-1}y^{2n}\Law(Y)(dy)$.
Then $(-1)^n c^{-1}\cf{Y}^{(2n)}$ is the characteristic function $\cf{\varrho}$ of
$\varrho$, and by assumption it is twice differentiable at $0$.
Writing $h := \cf{\varrho}$ and using $h(t)+h(-t)-2h(0) = -2\int(1-\cos(ty))\varrho(dy)$
we get, by Fatou's lemma,
\begin{align}
  \int_\R y^2\,\varrho(dy)
    &\le \liminf_{t\to0}\int_\R\frac{2\bigl(1-\cos(ty)\bigr)}{t^2}\,\varrho(dy)
     = -\lim_{t\to0}\frac{h(t)+h(-t)-2h(0)}{t^2}
     = -h''(0) < \infty,
\end{align}
because $2(1-\cos(ty))/t^2\to y^2$ pointwise as $t\to0$, and because the
second symmetric difference quotient converges to $h''(0)$ whenever the
latter exists, by Taylor's theorem with Peano remainder.
Hence $\Ex[Y^{2n+2}] = c\int y^2\varrho(dy)<\infty$.
\end{proof}

\begin{lemma}[Entire characteristic functions have all exponential moments]
\label{lem:entire-cf}
Let $Y$ be a real-valued random variable and suppose that there are $\delta>0$
and an entire function $F$ with $\cf{Y}(t) = F(t)$ for all $t\in(-\delta,\delta)$.
Then
\begin{align}
  \Ex\bigl[\exp(a\abs{Y})\bigr] &< \infty \qquad\text{for every } a>0,
\end{align}
the characteristic function $\cf{Y}$ extends to the entire function
$z\mapsto\Ex[e^{izY}]$, and this extension equals $F$ on all of $\C$.
\end{lemma}

\begin{proof}
Since $F$ is entire, $\cf{Y}$ is infinitely differentiable on
$(-\delta,\delta)$, so by \cref{lem:moments-from-derivatives} all \emph{even}
moments of $Y$ are finite, hence by Lyapunov's inequality all moments are, and
by \cref{prop:moments}
$\cf{Y}^{(m)}(0) = i^m\Ex[Y^m]$ for every $m$.
These are also the derivatives of $F$ at $0$, so the Taylor series of $F$ at
the origin is
\begin{align}
  F(z) &= \sum_{m\ge0}\frac{i^m\,\Ex[Y^m]}{m!}\,z^m,
\end{align}
and, $F$ being entire, this series converges absolutely for every $z\in\C$.
In particular, for every $a>0$,
\begin{align}
  \sum_{k\ge0}\frac{a^{2k}}{(2k)!}\,\Ex\bigl[Y^{2k}\bigr] < \infty .
\end{align}
Using $\exp(a\abs{y})\le\exp(ay)+\exp(-ay)$ and monotone convergence,
\begin{align}
  \Ex\bigl[\exp(a\abs{Y})\bigr]
    \le \Ex\bigl[e^{aY}\bigr]+\Ex\bigl[e^{-aY}\bigr]
     = 2\sum_{k\ge0}\frac{a^{2k}}{(2k)!}\Ex\bigl[Y^{2k}\bigr] < \infty .
\end{align}
By \cref{thm:analytic-strip} the map $z\mapsto\Ex[e^{izY}]$ is therefore entire
and restricts to $\cf{Y}$ on $\R$.  It agrees with $F$ on the interval
$(-\delta,\delta)$, hence on all of $\C$ by the identity theorem.
\end{proof}

\subsection{Marcinkiewicz' theorem}
\label{ssec:proof-marcinkiewicz}

\begin{proof}[Proof of \cref{thm:marcinkiewicz}]
By assumption $\cf{Y} = e^{g}$ on a neighbourhood of the origin, with a
polynomial $g$; replacing $g$ by $g-g(0)$ -- which changes $e^g$ not at all,
and the degree only in the trivial case of a constant $g$, since
$e^{g(0)} = \cf{Y}(0) = 1$ forces $g(0)\in2\pi i\,\Ints$ -- we may and do
assume $g(0)=0$.
The function $e^{g}$ is entire, so \cref{lem:entire-cf} applies: $Y$ has all
exponential moments, and
\begin{align}
  \label{eq:marc-entire}
  \cf{Y}(z) &= \Ex\bigl[e^{izY}\bigr] = \exp\bigl(g(z)\bigr),
  \qquad z\in\C .
\end{align}
Write $g(z) = \sum_{m=0}^{n}c_mz^m$ with $c_n\neq0$ and suppose, aiming at a
contradiction, that $n\ge3$.

Since $Y$ has all exponential moments, \cref{prop:ridge} applies on every
horizontal strip and gives, for $z = x+iy$ with $x,y\in\R$,
\begin{align}
  \exp\bigl(\Real g(x+iy)\bigr) = \bigabs{\cf{Y}(x+iy)}
    \le \cf{Y}(iy) = \exp\bigl(\Real g(iy)\bigr),
\end{align}
where we used that $\cf{Y}(iy) = \Ex[e^{-yY}]>0$ is real, so that
$\cf{Y}(iy) = \exp(\Real g(iy))$.  Taking logarithms,
\begin{align}
  \label{eq:marc-ridge}
  u(x,y) &\le u(0,y),
  & u(x,y) &:= \Real g(x+iy), \qquad x,y\in\R .
\end{align}

Write $c_n = \abs{c_n}e^{i\psi}$ and, for $\theta\in[0,2\pi)$ and $r>0$, put
$z = re^{i\theta}$.  Since $g$ is a polynomial of degree $n$,
\begin{align}
  \label{eq:marc-lhs}
  u\bigl(r\cos\theta,\,r\sin\theta\bigr)
    &= \Real\bigl(c_nr^ne^{in\theta}\bigr) + O\bigl(r^{n-1}\bigr)
     = \abs{c_n}\cos(\psi+n\theta)\,r^n + O\bigl(r^{n-1}\bigr),
\end{align}
uniformly in $\theta$.  Likewise, for $y\in\R$,
\begin{align}
  \label{eq:marc-rhs}
  u(0,y) &= \Real\bigl(c_n(iy)^n\bigr) + O\bigl(\abs{y}^{n-1}\bigr)
          \le \abs{c_n}\abs{y}^n + O\bigl(\abs{y}^{n-1}\bigr).
\end{align}

Let $\theta^\ast$ be any maximiser of $\theta\mapsto\cos(\psi+n\theta)$, i.e.\
$\cos(\psi+n\theta^\ast)=1$.  Applying \cref{eq:marc-ridge} with
$x = r\cos\theta^\ast$ and $y = r\sin\theta^\ast$ and inserting
\cref{eq:marc-lhs,eq:marc-rhs} gives
\begin{align}
  \abs{c_n}r^n + O\bigl(r^{n-1}\bigr)
    &\le \abs{c_n}\,\abs{\sin\theta^\ast}^n r^n + O\bigl(r^{n-1}\bigr).
\end{align}
Dividing by $r^n$ and letting $r\to\infty$ yields
$\abs{c_n}\le\abs{c_n}\abs{\sin\theta^\ast}^n$, and since $c_n\neq0$ and
$\abs{\sin\theta^\ast}\le1$ this forces $\abs{\sin\theta^\ast}=1$, i.e.\
$\theta^\ast\in\{\pi/2,\,3\pi/2\}$.

But the equation $\cos(\psi+n\theta)=1$ holds exactly when
$\psi+n\theta\in2\pi\,\Ints$, so its solutions form an arithmetic progression
of spacing $2\pi/n$, of which exactly $n$ lie in $[0,2\pi)$.  For $n\ge3$ there are thus at least three maximisers, while
only two values of $\theta$ in $[0,2\pi)$ satisfy $\abs{\sin\theta}=1$.  This
contradiction shows $n\le2$.

It remains to determine the coefficients.  We have $g(t) = c_1t+c_2t^2$ with
$g(0)=0$.  From $\cf{Y}(-t) = \overline{\cf{Y}(t)}$
(\cref{prop:cf-elementary}\,(iii)) we get
$\exp(g(-t)) = \exp(\overline{g(t)})$ for all real $t$, so the polynomial
$t\mapsto g(-t)-\overline{g(t)}$ takes values in $2\pi i\,\Ints$; being
continuous and vanishing at $t=0$, it vanishes identically.  Hence
$-c_1t+c_2t^2 = \overline{c_1}t+\overline{c_2}t^2$ for all real $t$, whence
$c_2 = \overline{c_2}$ is real and $c_1 = -\overline{c_1}$ is purely
imaginary, say $c_1 = i\mu$ with $\mu\in\R$.
Then $\abs{\cf{Y}(t)} = \exp(c_2t^2)\le1$ for all $t\in\R$ forces $c_2\le0$,
so $c_2 = -\tfrac12\sigma^2$ with $\sigma^2 := -2c_2\in\Rnn$.
Therefore $\cf{Y}(t) = \exp(i\mu t-\tfrac12\sigma^2t^2)$, which by
\cref{not:degenerate} and \cref{thm:uniqueness}\,(i) means
$Y\sim\Normal(\mu,\sigma^2)$.
\end{proof}

\subsection{Cram\'er's decomposition theorem}
\label{ssec:proof-cramer}

The growth estimate needed for Cram\'er's theorem is supplied by the following
classical coefficient bound, which replaces an appeal to Hadamard's
factorisation theorem.

\begin{lemma}[Borel--Carath\'eodory coefficient bound]
\label{lem:borel-caratheodory}
Let $h$ be an entire function with $h(0)=0$ and Taylor expansion
$h(z)=\sum_{m\ge1}a_mz^m$.  Then for every $r>0$ and every $m\ge1$,
\begin{align}
  \label{eq:bc-bound}
  \abs{a_m}\,r^m &\le 4\,M(r),
  & M(r) &:= \max_{\abs{z}=r}\Real h(z) .
\end{align}
\end{lemma}

\begin{proof}
Note first that $M(r)\ge\Real h(0)=0$ by the mean value property applied to
the harmonic function $\Real h$.
Write $a_m = \alpha_m+i\beta_m$ with $\alpha_m,\beta_m\in\R$.  On the circle
$z=re^{i\vartheta}$,
\begin{align}
  \Real h\bigl(re^{i\vartheta}\bigr)
    &= \sum_{m\ge1}r^m\bigl(\alpha_m\cos(m\vartheta)
       -\beta_m\sin(m\vartheta)\bigr),
\end{align}
the series converging uniformly in $\vartheta$.  Termwise integration gives
the Fourier coefficients
\begin{align}
  \frac1\pi\int_0^{2\pi}\Real h\bigl(re^{i\vartheta}\bigr)
      \cos(m\vartheta)\,d\vartheta &= r^m\alpha_m,
  &
  \frac1\pi\int_0^{2\pi}\Real h\bigl(re^{i\vartheta}\bigr)
      \sin(m\vartheta)\,d\vartheta &= -r^m\beta_m,
\end{align}
while the mean value property gives
$\int_0^{2\pi}\Real h(re^{i\vartheta})\,d\vartheta = 2\pi\,\Real h(0) = 0$.
Therefore, for either sign,
\begin{align}
  \pm r^m\alpha_m
    &= \frac1\pi\int_0^{2\pi}\Real h\bigl(re^{i\vartheta}\bigr)
       \bigl(1\pm\cos(m\vartheta)\bigr)d\vartheta
    \le \frac{M(r)}{\pi}\int_0^{2\pi}\bigl(1\pm\cos(m\vartheta)\bigr)
       d\vartheta = 2M(r),
\end{align}
where we used $1\pm\cos(m\vartheta)\ge0$ and
$\Real h(re^{i\vartheta})\le M(r)$.
Hence $\abs{\alpha_m}r^m\le2M(r)$, and the same argument with
$1\pm\sin(m\vartheta)$ gives $\abs{\beta_m}r^m\le2M(r)$.
Thus $\abs{a_m}r^m\le(\abs{\alpha_m}+\abs{\beta_m})r^m\le4M(r)$.
\end{proof}

\begin{proof}[Proof of \cref{thm:cramer}]
Let $Y_1\indep Y_2$ with $S := Y_1+Y_2\sim\Normal(\mu,\sigma^2)$.

\emph{The degenerate case.}  If $\sigma^2=0$ then $\cf{S} = \cf{Y_1}\cf{Y_2}$
has modulus $1$ everywhere, and since $\abs{\cf{Y_j}}\le1$ this forces
$\abs{\cf{Y_j}}\equiv1$.  Then the symmetrised variable $Y_j-Y_j'$, with
$Y_j'$ an independent copy of $Y_j$, has characteristic function
$\abs{\cf{Y_j}}^2\equiv1$, hence is almost surely $0$ by
\cref{thm:uniqueness}\,(i); so $Y_j$ is almost surely constant, i.e.\
degenerate Gaussian.  Assume from now on $\sigma>0$, and, replacing $Y_1$ by
$Y_1-\mu$, that $\mu=0$.

\emph{Step 1: Gaussian tails.}  Let $(Y_1',Y_2')$ be an independent copy of
$(Y_1,Y_2)$ and put $\widetilde Y_j := Y_j-Y_j'$ and
$\widetilde S := \widetilde Y_1+\widetilde Y_2 \sim \Normal(0,2\sigma^2)$.
The variables $\widetilde Y_1,\widetilde Y_2$ are independent and symmetric,
so $\Prb[\widetilde Y_2\ge0]\ge\tfrac12$ and hence, for $x>0$,
\begin{align}
  \Prb\bigl[\widetilde S>x\bigr]
    &\ge \Prb\bigl[\widetilde Y_1>x,\ \widetilde Y_2\ge0\bigr]
     = \Prb\bigl[\widetilde Y_1>x\bigr]\,\Prb\bigl[\widetilde Y_2\ge0\bigr]
    \ge \tfrac12\,\Prb\bigl[\widetilde Y_1>x\bigr].
\end{align}
A Chernoff bound for $\widetilde S\sim\Normal(0,2\sigma^2)$ gives
$\Prb[\widetilde S>x]\le\exp(-x^2/(4\sigma^2))$, so by symmetry
\begin{align}
  \label{eq:cramer-tail}
  \Prb\bigl[\abs{\widetilde Y_1}>x\bigr] &\le 4\exp\Bigl(-\frac{x^2}{4\sigma^2}\Bigr),
  \qquad x>0 .
\end{align}
Let $\theta$ be a median of $Y_1$, so that $\Prb[Y_1'\le\theta]\ge\tfrac12$
and $\Prb[Y_1'\ge\theta]\ge\tfrac12$.  Then
\begin{align}
  \Prb\bigl[\widetilde Y_1>x\bigr]
    \ge \Prb\bigl[Y_1>\theta+x,\ Y_1'\le\theta\bigr]
    \ge \tfrac12\,\Prb\bigl[Y_1>\theta+x\bigr],
\end{align}
and symmetrically for the left tail, so \cref{eq:cramer-tail} yields
$\Prb[\abs{Y_1-\theta}>x]\le16\exp(-x^2/(4\sigma^2))$.
Consequently $\Ex[\exp(a\abs{Y_1})]<\infty$ for every $a>0$, and the same
holds for $Y_2$.

\emph{Step 2: entire extensions.}  By \cref{thm:analytic-strip} both
$\cf{Y_1}$ and $\cf{Y_2}$ extend to entire functions, given by
$\cf{Y_j}(z) = \Ex[e^{izY_j}]$.  The entire functions
$z\mapsto\cf{Y_1}(z)\cf{Y_2}(z)$ and $z\mapsto\exp(-\tfrac12\sigma^2z^2)$
agree on $\R$, hence on $\C$ by the identity theorem:
\begin{align}
  \label{eq:cramer-product}
  \cf{Y_1}(z)\,\cf{Y_2}(z) &= \exp\Bigl(-\tfrac12\sigma^2z^2\Bigr),
  \qquad z\in\C .
\end{align}
In particular neither factor has a zero in $\C$.

\emph{Step 3: growth.}  Fix $j\in\{1,2\}$ and let $l$ be the other index.
By \cref{prop:ridge}, for $z=x+iy$,
$\abs{\cf{Y_j}(z)}\le\cf{Y_j}(iy) = \Ex[e^{-yY_j}]$.
By \cref{eq:cramer-product} at $z=iy$,
$\cf{Y_j}(iy)\,\cf{Y_l}(iy) = \exp(\tfrac12\sigma^2y^2)$, and Jensen's
inequality gives $\cf{Y_l}(iy) = \Ex[e^{-yY_l}]\ge\exp(-y\,\Ex[Y_l])$.
Hence, with $C := \abs{\Ex[Y_l]}$,
\begin{align}
  \label{eq:cramer-growth}
  \bigabs{\cf{Y_j}(z)}
    \le \frac{\exp\bigl(\tfrac12\sigma^2y^2\bigr)}{\cf{Y_l}(iy)}
    \le \exp\Bigl(\tfrac12\sigma^2y^2 + C\abs{y}\Bigr)
    \le \exp\Bigl(\tfrac12\sigma^2\abs{z}^2 + C\abs{z}\Bigr).
\end{align}

\emph{Step 4: conclusion.}  Since $\cf{Y_j}$ is entire and nowhere zero on the
simply connected domain $\C$, there is an entire $h_j$ with
$\cf{Y_j} = e^{h_j}$, and we may normalise $h_j(0)=0$ because
$\cf{Y_j}(0)=1$.  By \cref{eq:cramer-growth},
\begin{align}
  \Real h_j(z) = \log\bigabs{\cf{Y_j}(z)}
    \le \tfrac12\sigma^2\abs{z}^2 + C\abs{z},
\end{align}
so in the notation of \cref{lem:borel-caratheodory},
$M(r)\le\tfrac12\sigma^2r^2+Cr$.  For the Taylor coefficients $a_m$ of $h_j$
and every $m\ge3$, \cref{eq:bc-bound} gives
\begin{align}
  \abs{a_m} &\le \frac{4M(r)}{r^m}
    \le \frac{2\sigma^2r^2+4Cr}{r^m}
    \xrightarrow[r\to\infty]{} 0,
\end{align}
so $a_m=0$ for all $m\ge3$ and $h_j$ is a polynomial of degree at most $2$.
Thus $\cf{Y_j} = \exp(h_j)$ with a polynomial $h_j$, and Marcinkiewicz'
\cref{thm:marcinkiewicz} shows that $Y_j$ is Gaussian, possibly degenerate.
\end{proof}

\section{A Proof of the Kagan--Linnik--Rao Theorem}
\label{app:kagan-proof}

This appendix proves \cref{thm:kagan}, the one result that
\cref{sec:non-constant,sec:non-gaussian,sec:gaussian-free,sec:algorithm} take
as given.  Its first part -- the equality of images and ranks -- was already
proved in \cref{prop:affine-hull}; what remains is the column dichotomy
\cref{thm:kagan}\,(1)--(2).

The strategy is the classical one that goes back to \citet{marcinkiewicz1939},
\citet{darmois1953} and \citet{skitovich1954}, and that is developed
systematically in \citet[Chapters~3 and 10]{kagan1973characterization} and
\citet{linnik1977}.  It has three steps.
\begin{enumerate}[label=(\arabic*)]
  \item Take distinguished logarithms in the identity between the two
        characteristic functions.  Because each representation is an
        independent sum of \emph{one-dimensional} sources evaluated along the
        columns of the mixing matrix, the identity becomes a linear relation
        among \emph{ridge functions} $t\mapsto\psi(a\T t)$ on $\R^p$.
  \item Kill all but one of these ridge functions with finite-difference
        operators: a ridge function in direction $a$ is annihilated by the
        difference operator in any direction $h\perp a$, and the hypothesis
        that the columns are pairwise non-proportional is exactly what
        guarantees that each unwanted direction admits such an $h$ that does
        \emph{not} annihilate the term we want to keep.
  \item What survives is a one-dimensional function all of whose iterated
        differences vanish.  By Fr\'echet's functional equation such a
        function is a polynomial, and Marcinkiewicz'
        \cref{thm:marcinkiewicz} converts ``polynomial cumulant generating
        function'' into ``Gaussian''.
\end{enumerate}
Everything below uses only the elementary properties of characteristic
functions collected in \cref{prop:cf-elementary} -- in particular
\cref{eq:cf-affine,eq:cf-product} -- the distinguished logarithm of
\cref{lem:distinguished-log}, Marcinkiewicz' \cref{thm:marcinkiewicz},
\cref{prop:affine-hull} for the first part, and elementary calculus.

\subsection{Finite differences}
\label{ssec:finite-differences}

\begin{definition}[Difference operator]
\label{def:difference-operator}
For a function $F$ defined on a subset of $\R^p$ and a vector $h\in\R^p$ we
write
\begin{align}
  \bigl(\Delta_h F\bigr)(t) &:= F(t+h) - F(t),
\end{align}
whenever both $t$ and $t+h$ lie in the domain of $F$.
Difference operators in different directions commute, and iterating gives the
inclusion--exclusion formula
\begin{align}
  \label{eq:difference-expansion}
  \bigl(\Delta_{h_1}\cdots\Delta_{h_m}F\bigr)(t)
    &= \sum_{S\subseteq\{1,\dots,m\}}(-1)^{m-\abs{S}}\,
       F\Bigl(t+\sum_{r\in S}h_r\Bigr),
\end{align}
which requires $F$ to be defined at the $2^m$ points $t+\sum_{r\in S}h_r$.
\end{definition}

A function of the form $t\mapsto \psi(a\T t)$, constant along the hyperplane
$a^\perp$, is called a \emph{ridge function} in the direction $a$.  The next
two lemmas are the whole geometric input.

\begin{lemma}[Ridge functions are annihilated by orthogonal differences]
\label{lem:ridge-kill}
Let $a\in\R^p$, let $\psi$ be a function on a subset of $\R$ and put
$F(t) := \psi(a\T t)$.  If $h\in\R^p$ satisfies $a\T h = 0$, then
$\Delta_hF = 0$ on the domain of $F$.
\end{lemma}

\begin{proof}
$F(t+h) = \psi\bigl(a\T t + a\T h\bigr) = \psi(a\T t) = F(t)$.
\end{proof}

\begin{lemma}[Separating two non-proportional directions]
\label{lem:separating-direction}
Let $b,d\in\R^p\setminus\{0\}$ be non-proportional.  Then there exists
$h\in\R^p$ with
\begin{align}
  d\T h &= 0, & b\T h &\neq 0 .
\end{align}
Moreover the set of such $h$ is open in the hyperplane $d^\perp$, so $h$ may
be chosen with $\norm{h}$ arbitrarily small, and $sh$ is admissible for every
$s\neq0$.
\end{lemma}

\begin{proof}
Suppose no such $h$ exists, i.e.\ $d^\perp\subseteq b^\perp$.  Taking
orthogonal complements reverses the inclusion, so
$\spn\{b\} = \bigl(b^\perp\bigr)^\perp \subseteq
 \bigl(d^\perp\bigr)^\perp = \spn\{d\}$.
As $b\neq0$ this makes $b$ a non-zero multiple of $d$, contradicting
non-proportionality.
The set $\{h\in d^\perp : b\T h\neq0\}$ is the complement in $d^\perp$ of the
closed subspace $b^\perp\cap d^\perp$, hence open in $d^\perp$; and it is
invariant under $h\mapsto sh$ for $s\neq0$, so it contains vectors of
arbitrarily small norm.
\end{proof}

The analytic input is the local form of Fr\'echet's functional equation.  We
prove it from scratch; the argument is a smoothing bootstrap followed by one
differentiation.

\begin{lemma}[Fr\'echet's functional equation, local form]
\label{lem:frechet}
Let $m\in\N$ and let $\rho>\sigma>0$ and $\varepsilon>0$ satisfy
$\sigma+m\varepsilon\le\rho$.
Let $F\colon(-\rho,\rho)\to\C$ be continuous and assume
\begin{align}
  \label{eq:frechet-hypothesis}
  \bigl(\Delta_{c_1}\cdots\Delta_{c_m}F\bigr)(u) &= 0
  \qquad\text{whenever } \abs{c_1},\dots,\abs{c_m}<\varepsilon
  \text{ and } \abs{u}<\sigma .
\end{align}
Then $F$ coincides on $\bigl(-\tfrac{\sigma}{4},\tfrac{\sigma}{4}\bigr)$ with
a polynomial of degree at most $m-1$.
\end{lemma}

\begin{proof}
Note first that the condition $\sigma+m\varepsilon\le\rho$ makes
\cref{eq:frechet-hypothesis} meaningful: by
\cref{eq:difference-expansion} the left hand side only involves the values of
$F$ at the points $u+\sum_{r\in S}c_r$, all of which lie in $(-\rho,\rho)$.

Put
\begin{align}
  \kappa &:= \tfrac12\min\Bigl(\varepsilon,\ \frac{\sigma}{2m^2}\Bigr) \;>\;0,
\end{align}
and define the symmetric averaging operator
\begin{align}
  \bigl(M_\kappa G\bigr)(u)
    &:= \frac{1}{2\kappa}\int_{-\kappa}^{\kappa}G(u+s)\,ds
     = \frac{1}{2\kappa}\int_{u-\kappa}^{u+\kappa}G(v)\,dv ,
\end{align}
defined for $\abs{u}<A-\kappa$ whenever $G$ is continuous on $(-A,A)$.

\emph{Step 1: $M_\kappa$ raises smoothness.}
If $G$ is continuous on $(-A,A)$ then, by the fundamental theorem of
calculus, $M_\kappa G$ is continuously differentiable on
$(-A+\kappa,\,A-\kappa)$ with
\begin{align}
  \bigl(M_\kappa G\bigr)'(u)
    &= \frac{1}{2\kappa}\bigl(G(u+\kappa)-G(u-\kappa)\bigr).
\end{align}
Consequently, if $G\in C^{r}$ on $(-A,A)$ then $M_\kappa G\in C^{r+1}$ on
$(-A+\kappa,A-\kappa)$.

\emph{Step 2: the hypothesis says $(\Id{}-M_\kappa)^mF=0$.}
We have
\begin{align}
  F(u) - \bigl(M_\kappa F\bigr)(u)
    &= -\frac{1}{2\kappa}\int_{-\kappa}^{\kappa}\bigl(\Delta_sF\bigr)(u)\,ds ,
\end{align}
and translations -- hence all the operators $\Delta_s$ and $M_\kappa$ --
commute with one another, so Fubini's theorem gives
\begin{align}
  \label{eq:averaged-differences}
  \bigl[(\Id{}-M_\kappa)^mF\bigr](u)
    &= \frac{(-1)^m}{(2\kappa)^m}\int_{-\kappa}^{\kappa}\!\!\cdots\!\!
       \int_{-\kappa}^{\kappa}
       \bigl(\Delta_{s_1}\cdots\Delta_{s_m}F\bigr)(u)\,ds_1\cdots ds_m
     = 0
\end{align}
for $\abs{u}<\sigma$: every $s_r$ occurring in the integral satisfies
$\abs{s_r}\le\kappa<\varepsilon$, so the integrand vanishes identically by
\cref{eq:frechet-hypothesis}.

\emph{Step 3: bootstrap.}
Expanding the binomial in \cref{eq:averaged-differences} and solving for the
$n=0$ term,
\begin{align}
  \label{eq:frechet-bootstrap}
  F &= \sum_{n=1}^{m}(-1)^{n+1}\binom{m}{n}\,M_\kappa^{\,n}F
  \qquad\text{on } (-\sigma,\sigma) .
\end{align}
Every term on the right applies $M_\kappa$ at least once.
Hence, by Step 1, if $F\in C^{r}$ on $(-A,A)$ for some $A\le\rho$, then the
right hand side of \cref{eq:frechet-bootstrap} is $C^{r+1}$ on
$(-A+m\kappa,A-m\kappa)$, and therefore $F\in C^{r+1}$ on
$\bigl(-A',A'\bigr)$ with $A' := \min(\sigma,\,A-m\kappa)$.
Starting from $F\in C^0$ on $(-\rho,\rho)$ and iterating $m$ times, and using
$\rho-m\kappa\ge\rho-m\varepsilon\ge\sigma$ at the first step, we obtain
$F\in C^{m}$ on $(-A_m,A_m)$ with
$A_m \ge \sigma-(m-1)m\kappa \ge \sigma-m^2\kappa \ge \tfrac34\sigma$,
where the last inequality uses $\kappa\le\sigma/(4m^2)$.

\emph{Step 4: differentiate.}
Fix $u$ with $\abs{u}<\sigma/4$ and let $c_1,\dots,c_m$ range over
$(-\delta,\delta)$, where
$\delta := \tfrac12\min\bigl(\varepsilon,\,\sigma/(2m)\bigr)$, so that
$\abs{u}+m\delta<\tfrac34\sigma\le A_m$ and $\delta<\varepsilon$.
By \cref{eq:difference-expansion} and \cref{eq:frechet-hypothesis},
\begin{align}
  \sum_{S\subseteq\{1,\dots,m\}}(-1)^{m-\abs{S}}\,
    F\Bigl(u+\sum_{r\in S}c_r\Bigr) &= 0 ,
\end{align}
identically in $(c_1,\dots,c_m)\in(-\delta,\delta)^m$.
All terms are $C^m$ functions of $(c_1,\dots,c_m)$ by Step 3, so we may apply
$\partial^m/\partial c_1\cdots\partial c_m$.
A summand indexed by $S\neq\{1,\dots,m\}$ does not depend on $c_{r_0}$ for any
$r_0\notin S$ and is therefore annihilated; the summand indexed by
$S=\{1,\dots,m\}$ contributes $F^{(m)}(u+c_1+\dots+c_m)$.
Hence $F^{(m)}(u+c_1+\dots+c_m)=0$, and setting $c_1=\dots=c_m=0$ gives
$F^{(m)}(u)=0$.
As $u$ was arbitrary in the interval
$\bigl(-\tfrac{\sigma}{4},\tfrac{\sigma}{4}\bigr)$, which is connected, $F$ is
there a polynomial of degree at most $m-1$.
\end{proof}

\begin{remark}[What the hypotheses of \cref{lem:frechet} do and do not say]
\label{rem:frechet-scope}
Two features of \cref{eq:frechet-hypothesis} are essential.
First, the steps $c_1,\dots,c_m$ are allowed to \emph{vary} over a
neighbourhood of the origin, which is what makes the averaging in
\cref{eq:averaged-differences} legitimate; for a single fixed step vector the
conclusion is false, as the $1$-periodic function $F(u)=\sin(2\pi u)$ with
$m=1$, $c_1=1$ shows.
Second, only continuity of $F$ is assumed; no differentiability is available a
priori, and Step~3 is precisely what manufactures it.
The global version of the statement -- a continuous $F$ on $\R$ with
$\Delta_{c_1}\cdots\Delta_{c_m}F\equiv0$ for all $c_1,\dots,c_m\in\R$ is a
polynomial of degree $<m$ -- is due to \citet{frechet1909}, and the equation
carries his name.
Restricted-domain versions, in which the increments and the base point are
confined to a subset, have been studied since; \citet{ger1994} shows that a
solution on a sufficiently regular domain always extends to a polynomial
function on the whole space.
We have given a direct proof of the one-dimensional continuous case, which is
all that is needed below and is short enough not to warrant a detour.
\end{remark}

\subsection{Proof of the column dichotomy}
\label{ssec:kagan-proof}

\begin{proof}[Proof of \cref{thm:kagan}\,(1)--(2)]
Write $k_i := k\up{i}$, let $a_1,\dots,a_{k_1}$ be the columns of $A\up{1}$
and $b_1,\dots,b_{k_2}$ the columns of $A\up{2}$, and put
\begin{align}
  \phi_j &:= \cf{Z\up{1}_j}, & \gamma_l &:= \cf{Z\up{2}_l},
  & c &:= \mu\up{2}-\mu\up{1} .
\end{align}
By \cref{eq:cf-affine,eq:cf-product}, the two representations
\cref{eq:two-rep-kagan} give, for every $t\in\R^p$,
\begin{align}
  \label{eq:kagan-cf-identity}
  \exp\bigl(i\,t\T\mu\up{1}\bigr)\prod_{j=1}^{k_1}\phi_j\bigl(a_j\T t\bigr)
    &= \cf{X}(t)
     = \exp\bigl(i\,t\T\mu\up{2}\bigr)
       \prod_{l=1}^{k_2}\gamma_l\bigl(b_l\T t\bigr).
\end{align}

\emph{Step 1: the logarithmic form.}
By \cref{prop:cf-elementary}\,(i)--(ii) each $\phi_j$ and each $\gamma_l$ is
continuous and equals $1$ at the origin, so there is a $\delta_0>0$ such that
all of them are non-vanishing on $(-\delta_0,\delta_0)$, and
\cref{lem:distinguished-log} supplies the cumulant generating functions
$\psi_j := \psi_{Z\up{1}_j}$ and $\chi_l := \psi_{Z\up{2}_l}$, continuous on
$(-\delta_0,\delta_0)$ with $\psi_j(0)=\chi_l(0)=0$ and
$\phi_j=\exp(\psi_j)$, $\gamma_l=\exp(\chi_l)$ there.
Choose $R>0$ so small that $\abs{a_j\T t}<\delta_0$ and $\abs{b_l\T t}<\delta_0$
for all $j,l$ whenever $\norm{t}<R$, and define on the ball $B_R(0)\subseteq\R^p$
\begin{align}
  \label{eq:kagan-Phi}
  \Phi(t) &:= \sum_{j=1}^{k_1}\psi_j\bigl(a_j\T t\bigr)
             - \sum_{l=1}^{k_2}\chi_l\bigl(b_l\T t\bigr)
             - i\,t\T c .
\end{align}
Then $\Phi$ is continuous, $\Phi(0)=0$, and by \cref{eq:kagan-cf-identity}
$\exp(\Phi)\equiv1$ on $B_R(0)$.
Hence $\Phi/(2\pi i)$ is a continuous, integer-valued function on the
connected set $B_R(0)$ vanishing at the origin, so it vanishes identically:
\begin{align}
  \label{eq:kagan-log-identity}
  \sum_{j=1}^{k_1}\psi_j\bigl(a_j\T t\bigr)
    - \sum_{l=1}^{k_2}\chi_l\bigl(b_l\T t\bigr) &= i\,t\T c ,
  \qquad t \in B_R(0) .
\end{align}
This is a linear relation among ridge functions, and it is the only
consequence of the hypotheses that we shall use.

\emph{Step 2: choosing the directions.}
We treat both parts at once.  Fix the index $l_0\in\{1,\dots,k_2\}$ of the
column under consideration and let
\begin{align}
  \label{eq:kagan-target}
  e &:= \begin{cases}
          b_{l_0}, & \text{in case (1)},\\
          a_{j_0}, & \text{in case (2), where } b_{l_0}=\lambda a_{j_0} ,
        \end{cases}
\end{align}
and let $\mathcal{D}$ be the multiset of the remaining column directions,
\begin{align}
  \mathcal{D} &:= \begin{cases}
     \{a_1,\dots,a_{k_1}\}\uplus\{b_l : l\neq l_0\}, & \text{in case (1)},\\
     \{a_j : j\neq j_0\}\uplus\{b_l : l\neq l_0\}, & \text{in case (2)},
  \end{cases}
\end{align}
where $\uplus$ is the disjoint union, so that
$\abs{\mathcal{D}} = k_1+k_2-1$ in case (1) and $k_1+k_2-2$ in case (2) even
if the same vector happens to be a column of both matrices.
We claim that no $d\in\mathcal{D}$ is proportional to $e$.
In case (1) this holds for $d=a_j$ by the hypothesis of part~(1) and for
$d=b_l$, $l\neq l_0$, because the columns of $A\up{2}$ are pairwise
non-proportional.
In case (2) it holds for $d=a_j$, $j\neq j_0$, because the columns of
$A\up{1}$ are pairwise non-proportional, and for $d=b_l$, $l\neq l_0$, because
$b_l\parallel a_{j_0}$ would give $b_l\parallel\lambda a_{j_0}=b_{l_0}$, again
contradicting pairwise non-proportionality.
All columns are non-zero by hypothesis (i), so $e\neq0$ and $d\neq0$.

Let $m_0 := \abs{\mathcal{D}}$ and set $m := \max(m_0,2)$.
For $r=1,\dots,m_0$ apply \cref{lem:separating-direction} to the pair
$(e,d_r)$: it yields $h_r\in\R^p$ with $d_r\T h_r=0$ and
$\alpha_r := e\T h_r\neq0$, and $h_r$ may be taken with
$\norm{h_r}<R/(2m)$.
For $r=m_0+1,\dots,m$ (there are at most two such $r$, and only when
$m_0<2$) simply put $h_r := \tau\, e/\norm{e}^2$ with $\tau>0$ small, so
that $\alpha_r = \tau \neq0$ and $\norm{h_r}<R/(2m)$ as well.
Note that each $h_r$ may be replaced by $s_rh_r$ for any $s_r\in(-1,1)
\setminus\{0\}$ without disturbing any of these properties, and that this
replaces $\alpha_r$ by $s_r\alpha_r$.

\emph{Step 3: annihilating all but one term.}
Apply the commuting operators $\Delta_{h_1},\dots,\Delta_{h_m}$ to
\cref{eq:kagan-log-identity}.  This is legitimate on $B_{R/2}(0)$, because
$\norm{t}<R/2$ and $\norm{h_r}<R/(2m)$ force
$\norm{t+\sum_{r\in S}h_r}<R$ for every $S\subseteq\{1,\dots,m\}$.
Each ridge function indexed by a direction $d_r\in\mathcal{D}$ is annihilated
by $\Delta_{h_r}$ (\cref{lem:ridge-kill}), hence by the whole composition; and
the affine term $i\,t\T c$ is annihilated by any two of the operators, of
which there are $m\ge2$.  What survives is
\begin{align}
  \label{eq:kagan-survivor}
  \Delta_{h_1}\cdots\Delta_{h_m}\bigl[\Xi\bigl(e\T t\bigr)\bigr] &= 0,
  \qquad \norm{t}<R/2, \qquad\text{where}\qquad
  \Xi := \begin{cases}
     -\chi_{l_0}, & \text{in case (1)},\\
     \psi_{j_0}-\chi_{l_0}(\lambda\,\cdot\,), & \text{in case (2)} .
  \end{cases}
\end{align}
In case (2) we used that $b_{l_0}\T t = \lambda\,a_{j_0}\T t = \lambda\,e\T t$,
so that the two surviving ridge functions share the direction $e$ and combine
into a single one.

\emph{Step 4: Fr\'echet.}
Set $\rho := \norm{e}R$.  Every $u$ with $\abs{u}<\rho$ is of the form
$u = e\T t$ for some $t\in B_R(0)$, and for such $t$ the choice of $R$ puts
every $a_j\T t$ and every $b_l\T t$ inside $(-\delta_0,\delta_0)$.
In case (1), $\Xi=-\chi_{l_0}$ is evaluated at $u = b_{l_0}\T t$; in case (2)
it is built from $\psi_{j_0}$ evaluated at $u = a_{j_0}\T t$ and from
$\chi_{l_0}$ evaluated at $\lambda u = b_{l_0}\T t$.
In either case all arguments lie in $(-\delta_0,\delta_0)$, so $\Xi$ is
defined and continuous on $(-\rho,\rho)$.
Because $e\neq0$, the linear functional $t\mapsto e\T t$ maps $B_{R/2}(0)$
\emph{onto} $(-\sigma,\sigma)$ with $\sigma := \norm{e}R/2 = \rho/2$, and by
\cref{eq:difference-expansion} the left hand side of
\cref{eq:kagan-survivor} equals
$\bigl(\Delta_{\alpha_1}\cdots\Delta_{\alpha_m}\Xi\bigr)(e\T t)$.
Hence
\begin{align}
  \label{eq:kagan-1d-differences}
  \bigl(\Delta_{\alpha_1}\cdots\Delta_{\alpha_m}\Xi\bigr)(u) &= 0,
  \qquad \abs{u}<\sigma .
\end{align}
Replacing each $h_r$ by $s_rh_r$ with $s_r\in(-1,1)$, as permitted in
Step~2, replaces $\alpha_r$ by $s_r\alpha_r$ without affecting anything else,
and as $s_r$ ranges over $(-1,1)$ the number $s_r\alpha_r$ ranges over the
whole interval $(-\abs{\alpha_r},\abs{\alpha_r})$.
So \cref{eq:kagan-1d-differences} holds with $\alpha_r$ replaced by any
$c_r$ with $0<\abs{c_r}<\varepsilon := \min_r\abs{\alpha_r}$; and for
$c_r=0$ it holds trivially, because $\Delta_0 = 0$.
This is exactly the hypothesis \cref{eq:frechet-hypothesis}.
Finally $\abs{\alpha_r} = \abs{e\T h_r} \le \norm{e}\norm{h_r} <
\norm{e}R/(2m) = \rho/(2m)$, so
$\sigma + m\varepsilon < \rho/2+\rho/2 = \rho$ and \cref{lem:frechet} applies.
It makes $\Xi$ a polynomial of degree at most $m-1$ on
$\bigl(-\tfrac{\sigma}{4},\tfrac{\sigma}{4}\bigr)$, a neighbourhood of the
origin in $\R$.

\emph{Step 5: conclusion.}
In case (1) we have obtained that $\chi_{l_0}=-\Xi$ is a polynomial near the
origin, so
\begin{align}
  \cf{Z\up{2}_{l_0}}(t) &= \exp\bigl(\chi_{l_0}(t)\bigr)
\end{align}
near the origin with a polynomial exponent.  By Marcinkiewicz'
\cref{thm:marcinkiewicz}, $Z\up{2}_{l_0}$ is a (possibly degenerate) Gaussian
random variable, which is assertion (1).

In case (2), $\Xi(t) = \psi_{j_0}(t)-\chi_{l_0}(\lambda t)$ is a polynomial
near the origin; writing $g := -\Xi$ and exponentiating,
\begin{align}
  \cf{Z\up{2}_{l_0}}(\lambda t)
    &= \exp\bigl(\chi_{l_0}(\lambda t)\bigr)
     = \exp\bigl(\psi_{j_0}(t)\bigr)\exp\bigl(g(t)\bigr)
     = \cf{Z\up{1}_{j_0}}(t)\cdot\exp\bigl(g(t)\bigr)
\end{align}
near the origin, which is \cref{eq:cf-kagan}.
For the final assertion of (2), note that this is exactly the situation of
\cref{rem:marcinkiewicz-use}\,(i) \emph{with the factor $\lambda$ present},
which is why that remark was stated with the factor: it gives directly that
$Z\up{2}_{l_0}$ is Gaussian if and only if $Z\up{1}_{j_0}$ is.
(The factor cannot be dropped: by \cref{eq:cf-affine} the left hand side is
$\cf{\lambda Z\up{2}_{l_0}}(t)$, and it is only because $\lambda\neq0$ that
$\lambda Z\up{2}_{l_0}$ being Gaussian is equivalent to $Z\up{2}_{l_0}$ being
Gaussian; cf.\ \cref{fn:missing-constant}.)
\end{proof}

\begin{remark}[Reading the proof]
\label{rem:kagan-proof-reading}
Three points are worth isolating.

\emph{Where each hypothesis enters.}
That the columns are non-zero is used twice: in Step~2, to make
\cref{lem:separating-direction} applicable to the pairs $(e,d_r)$, and in
Step~4, where it makes $e\T t$ surject onto an interval, so that a statement
about $t\in\R^p$ becomes a statement about $u\in\R$.
Pairwise non-proportionality is used only in Step~2, to make
\cref{lem:separating-direction} applicable, and only for those columns that
have to be separated from $e$: in case (2) for the columns of both matrices,
in case (1) for the columns of $A\up{2}$ only -- there the separation of
$b_{l_0}$ from the columns of $A\up{1}$ is the hypothesis of part~(1) itself,
and the columns of $A\up{1}$ may perfectly well be proportional to one
another.
Mutual independence of the components, hypothesis (iii)(a), is used once, to
factorise the characteristic functions in \cref{eq:kagan-cf-identity}.
Non-constancy of the sources, hypothesis (iii)(b), is not needed for
(1)--(2) at all; it enters only through \cref{prop:affine-hull}, i.e.\ in the
first part of \cref{thm:kagan}.

\emph{Why the degree bound is irrelevant.}
Step~4 produces a polynomial of degree at most $m-1$, where
$m=\max(k_1+k_2-1,2)$ in case (1) and $m=\max(k_1+k_2-2,2)$ in case (2);
this bound grows with the number of sources and is of no use by itself.
It is Marcinkiewicz' theorem that collapses it to $2$ -- which is precisely
why \cref{thm:marcinkiewicz} was called the single most important analytic
input of these notes in \cref{ssec:cf-analytic}.

\emph{Why the argument is local.}
Everything happens on a ball around the origin, because
\cref{lem:distinguished-log} produces the cumulant generating functions only
there.  This is not a defect: by \cref{rem:local-not-enough} a local identity
between characteristic functions carries no information about the laws in
general, and the reason it does here is that the local conclusion is
``$\exp$ of a polynomial'', a form rigid enough for \cref{thm:marcinkiewicz}
to upgrade it to a global statement about the distribution.
\end{remark}

\begin{remark}[Relation to the Darmois--Skitovich theorem]
\label{rem:darmois-skitovich}
The classical theorem of \citet{darmois1953} and \citet{skitovich1954} says
this: if $Y_1,\dots,Y_k$ are independent and the two linear forms
$L_1=\sum_j\alpha_jY_j$ and $L_2=\sum_j\beta_jY_j$ are independent of
\emph{each other}, then every $Y_j$ with $\alpha_j\beta_j\neq0$ is Gaussian.
It is the case $p=2$ of \cref{thm:kagan}.
Indeed, put $X := [L_1,L_2]\T\in\R^2$ and compare the two representations
\begin{align}
  \Id{2}\,\bmat{L_1\\L_2}
    \;=\; X \;=\;
  \bmat{\alpha_1 & \cdots & \alpha_k\\ \beta_1&\cdots&\beta_k}\,Y ,
\end{align}
so that $k\up{1}=2$ with $A\up{1}=\Id{2}$, and $k\up{2}=k$ with
$a\up{2}_j = [\alpha_j,\beta_j]\T$.
A \emph{non-zero} column $a\up{2}_j$ is proportional to a column of $\Id{2}$
precisely when $\alpha_j\beta_j=0$ (zero columns having been removed by
\cref{rem:normalisation}), so for every $j$ with $\alpha_j\beta_j\neq0$
part~(1) declares $Y_j$ Gaussian.
Two caveats: the columns $a\up{2}_j$ need not be pairwise non-proportional, so
one first passes to a normalised representation as in
\cref{rem:normalisation}, which \emph{merges} the sources belonging to
proportional columns; and one then recovers the individual $Y_j$ from the
merged sum by finitely many applications of Cram\'er's \cref{thm:cramer},
the case of an a.s.\ constant merged source being covered by the convention
that constants are degenerate Gaussians (\cref{not:degenerate}).

The two-representation formulation of \cref{thm:kagan} is what makes the
result directly usable for identifiability, since it compares two
\emph{models} rather than two forms, and it is proved by the same
finite-difference argument.
See \citet[Chapter~3]{kagan1973characterization} for the classical statement
and \citet{ghurye1962} for a multivariate version.
\end{remark}

\phantomsection


\addcontentsline{toc}{section}{References}

\begin{thebibliography}{Kagan et al.(1973)}
\setlength{\itemsep}{3pt}
\sloppy\emergencystretch=2em

\bibitem[Amari et al.(1996)]{amari1996newlearning}
S.~Amari, A.~Cichocki, and H.~H. Yang.
\newblock A new learning algorithm for blind signal separation.
\newblock In D.~S. Touretzky, M.~C. Mozer, and M.~E. Hasselmo, editors,
  \emph{Advances in Neural Information Processing Systems 8 (NIPS 1995)},
  pages 757--763. MIT Press, Cambridge, MA, 1996.
\newblock \biburl{https://papers.nips.cc/paper/1115-a-new-learning-algorithm-for-blind-signal-separation}.

\bibitem[Amari et al.(1997)]{amari1997stability}
S.~Amari, T.-P. Chen, and A.~Cichocki.
\newblock Stability analysis of learning algorithms for blind source
  separation.
\newblock \emph{Neural Networks}, 10(8):1345--1351, 1997.
\newblock \doi{10.1016/S0893-6080(97)00039-7}.

\bibitem[Amari and Cardoso(1997)]{amari1997semiparametric}
S.~Amari and J.-F. Cardoso.
\newblock Blind source separation---semiparametric statistical approach.
\newblock \emph{IEEE Transactions on Signal Processing}, 45(11):2692--2700,
  1997.
\newblock \doi{10.1109/78.650095}.

\bibitem[Amari(1998)]{amari1998natural}
S.~Amari.
\newblock Natural gradient works efficiently in learning.
\newblock \emph{Neural Computation}, 10(2):251--276, 1998.
\newblock \doi{10.1162/089976698300017746}.

\bibitem[Balanda and MacGillivray(1988)]{balanda1988kurtosis}
K.~P. Balanda and H.~L. MacGillivray.
\newblock Kurtosis: a critical review.
\newblock \emph{The American Statistician}, 42(2):111--119, 1988.
\newblock \doi{10.1080/00031305.1988.10475539}.

\bibitem[Bell and Sejnowski(1995)]{bell1995infomax}
A.~J. Bell and T.~J. Sejnowski.
\newblock An information-maximization approach to blind separation and blind
  deconvolution.
\newblock \emph{Neural Computation}, 7(6):1129--1159, 1995.
\newblock \doi{10.1162/neco.1995.7.6.1129}.

\bibitem[Cardoso and Laheld(1996)]{cardoso1996equivariant}
J.-F. Cardoso and B.~H. Laheld.
\newblock Equivariant adaptive source separation.
\newblock \emph{IEEE Transactions on Signal Processing}, 44(12):3017--3030,
  1996.
\newblock \doi{10.1109/78.553476}.

\bibitem[Cardoso(1997)]{cardoso1997infomax}
J.-F. Cardoso.
\newblock Infomax and maximum likelihood for blind source separation.
\newblock \emph{IEEE Signal Processing Letters}, 4(4):112--114, 1997.
\newblock \doi{10.1109/97.566704}.

\bibitem[Cardoso(1998a)]{cardoso1998multidim}
J.-F. Cardoso.
\newblock Multidimensional independent component analysis.
\newblock In \emph{Proceedings of the 1998 IEEE International Conference on
  Acoustics, Speech and Signal Processing (ICASSP~'98)}, volume~4, pages
  1941--1944, Seattle, WA, 1998a. IEEE.
\newblock \doi{10.1109/ICASSP.1998.681443}.

\bibitem[Cardoso(1998b)]{cardoso1998statistical}
J.-F. Cardoso.
\newblock Blind signal separation: statistical principles.
\newblock \emph{Proceedings of the IEEE}, 86(10):2009--2025, 1998b.
\newblock \doi{10.1109/5.720250}.

\bibitem[Comon(1994)]{comon1994}
P.~Comon.
\newblock Independent component analysis, a new concept?
\newblock \emph{Signal Processing}, 36(3):287--314, 1994.
\newblock \doi{10.1016/0165-1684(94)90029-9}.

\bibitem[Comon and Jutten(2010)]{comon2010handbook}
P.~Comon and C.~Jutten, editors.
\newblock \emph{Handbook of Blind Source Separation: Independent Component
  Analysis and Applications}.
\newblock Academic Press (Elsevier), Oxford, 1st edition, 2010.
\newblock ISBN 978-0-12-374726-6 (print), 978-0-08-088494-3 (e-book).

\bibitem[Cram\'er(1936)]{cramer1936}
H.~Cram\'er.
\newblock \"Uber eine Eigenschaft der normalen Verteilungsfunktion.
\newblock \emph{Mathematische Zeitschrift}, 41(1):405--414, 1936.
\newblock \doi{10.1007/BF01180430}.

\bibitem[Darlington(1970)]{darlington1970kurtosis}
R.~B. Darlington.
\newblock Is kurtosis really ``peakedness?''
\newblock \emph{The American Statistician}, 24(2):19--22, 1970.
\newblock \doi{10.1080/00031305.1970.10478885}.

\bibitem[Darmois(1953)]{darmois1953}
G.~Darmois.
\newblock Analyse g\'en\'erale des liaisons stochastiques: \'etude
  particuli\`ere de l'analyse factorielle lin\'eaire.
\newblock \emph{Revue de l'Institut International de Statistique / Review of
  the International Statistical Institute}, 21(1/2):2--8, 1953.
\newblock \doi{10.2307/1401511}.

\bibitem[Eriksson and Koivunen(2004)]{eriksson2004}
J.~Eriksson and V.~Koivunen.
\newblock Identifiability, separability, and uniqueness of linear ICA models.
\newblock \emph{IEEE Signal Processing Letters}, 11(7):601--604, 2004.
\newblock \doi{10.1109/LSP.2004.830118}.

\bibitem[Eriksson and Koivunen(2006)]{eriksson2006}
J.~Eriksson and V.~Koivunen.
\newblock Complex random vectors and ICA models: identifiability, uniqueness,
  and separability.
\newblock \emph{IEEE Transactions on Information Theory}, 52(3):1017--1029,
  2006.
\newblock \doi{10.1109/TIT.2005.864440}.
\newblock Preprint: \biburl{https://arxiv.org/abs/cs/0512063}.

\bibitem[Feller(1971)]{feller1971}
W.~Feller.
\newblock \emph{An Introduction to Probability Theory and Its Applications},
  volume~II.
\newblock Wiley Series in Probability and Mathematical Statistics. John Wiley
  \& Sons, New York, 2nd edition, 1971.
\newblock 669 pp. ISBN 978-0-471-25709-7.

\bibitem[Fr\'echet(1909)]{frechet1909}
M.~Fr\'echet.
\newblock Une d\'efinition fonctionnelle des polynomes.
\newblock \emph{Nouvelles annales de math\'ematiques}, 4e s\'erie,
  9:145--162, 1909.
\newblock \biburl{https://www.numdam.org/item/NAM_1909_4_9__145_0/}.

\bibitem[Ger(1994)]{ger1994}
R.~Ger.
\newblock On extensions of polynomial functions.
\newblock \emph{Results in Mathematics}, 26(3--4):281--289, 1994.
\newblock \doi{10.1007/BF03323050}.

\bibitem[Ghurye and Olkin(1962)]{ghurye1962}
S.~G. Ghurye and I.~Olkin.
\newblock A characterization of the multivariate normal distribution.
\newblock \emph{The Annals of Mathematical Statistics}, 33(2):533--541, 1962.
\newblock \doi{10.1214/aoms/1177704579}.

\bibitem[Gresele et al.(2021)]{gresele2021ima}
L.~Gresele, J.~von K\"ugelgen, V.~Stimper, B.~Sch\"olkopf, and M.~Besserve.
\newblock Independent mechanism analysis, a new concept?
\newblock In \emph{Advances in Neural Information Processing Systems 34
  (NeurIPS 2021)}, pages 28233--28248. Curran Associates, 2021.
\newblock \biburl{https://proceedings.neurips.cc/paper/2021/hash/edc27f139c3b4e4bb29d1cdbc45663f9-Abstract.html}.
\newblock Preprint: \biburl{https://arxiv.org/abs/2106.05200}.

\bibitem[Hoyer et al.(2008)]{hoyer2008anm}
P.~O. Hoyer, D.~Janzing, J.~M. Mooij, J.~Peters, and B.~Sch\"olkopf.
\newblock Nonlinear causal discovery with additive noise models.
\newblock In \emph{Advances in Neural Information Processing Systems 21
  (NIPS 2008)}, pages 689--696, 2008.
\newblock \biburl{https://proceedings.neurips.cc/paper/2008/hash/f7664060cc52bc6f3d620bcedc94a4b6-Abstract.html}.

\bibitem[Hyv\"arinen and Morioka(2016)]{hyvarinen2016tcl}
A.~Hyv\"arinen and H.~Morioka.
\newblock Unsupervised feature extraction by time-contrastive learning and
  nonlinear ICA.
\newblock In \emph{Advances in Neural Information Processing Systems 29
  (NIPS 2016)}, pages 3765--3773. Curran Associates, 2016.
\newblock \biburl{https://proceedings.neurips.cc/paper/2016/hash/d305281faf947ca7acade9ad5c8c818c-Abstract.html}.
\newblock Preprint: \biburl{https://arxiv.org/abs/1605.06336}.

\bibitem[Hyv\"arinen and Morioka(2017)]{hyvarinen2017pcl}
A.~Hyv\"arinen and H.~Morioka.
\newblock Nonlinear ICA of temporally dependent stationary sources.
\newblock In A.~Singh and J.~Zhu, editors, \emph{Proceedings of the 20th
  International Conference on Artificial Intelligence and Statistics
  (AISTATS)}, volume~54 of \emph{Proceedings of Machine Learning Research},
  pages 460--469. PMLR, 2017.
\newblock \biburl{https://proceedings.mlr.press/v54/hyvarinen17a.html}.

\bibitem[Hyv\"arinen and Oja(2000)]{hyvarinen2000}
A.~Hyv\"arinen and E.~Oja.
\newblock Independent component analysis: algorithms and applications.
\newblock \emph{Neural Networks}, 13(4--5):411--430, 2000.
\newblock \doi{10.1016/S0893-6080(00)00026-5}.

\bibitem[Hyv\"arinen and Pajunen(1999)]{hyvarinen1999nonlinear}
A.~Hyv\"arinen and P.~Pajunen.
\newblock Nonlinear independent component analysis: existence and uniqueness
  results.
\newblock \emph{Neural Networks}, 12(3):429--439, 1999.
\newblock \doi{10.1016/S0893-6080(98)00140-3}.

\bibitem[Hyv\"arinen et al.(2001)]{hyvarinen2001}
A.~Hyv\"arinen, J.~Karhunen, and E.~Oja.
\newblock \emph{Independent Component Analysis}.
\newblock Wiley Series on Adaptive and Learning Systems for Signal Processing,
  Communications, and Control. John Wiley \& Sons, New York, 2001.
\newblock 504 pp. ISBN 978-0-471-40540-5.
\newblock \doi{10.1002/0471221317}.

\bibitem[Kagan et al.(1973)]{kagan1973characterization}
A.~M. Kagan, Yu.~V. Linnik, and C.~R. Rao.
\newblock \emph{Characterization Problems in Mathematical Statistics}.
\newblock Wiley Series in Probability and Mathematical Statistics. John Wiley
  \& Sons, New York, 1973.
\newblock xii+499 pp. ISBN 978-0-471-45421-2.
\newblock Translated from the Russian by B.~Ramachandran; Russian original:
  \emph{Kharakterizatsionnye zadachi matematicheskoi statistiki}, Nauka,
  Moscow, 1972.

\bibitem[Kallenberg(2021)]{kallenberg2021}
O.~Kallenberg.
\newblock \emph{Foundations of Modern Probability}, volume~99 of
  \emph{Probability Theory and Stochastic Modelling}.
\newblock Springer, Cham, 3rd edition, 2021.
\newblock xii+946 pp. ISBN 978-3-030-61870-4.
\newblock \doi{10.1007/978-3-030-61871-1}.

\bibitem[Kaplansky(1945)]{kaplansky1945common}
I.~Kaplansky.
\newblock A common error concerning kurtosis.
\newblock \emph{Journal of the American Statistical Association}, 40(230):259,
  1945.
\newblock \doi{10.1080/01621459.1945.10501856}.

\bibitem[Khemakhem et al.(2020)]{khemakhem2020ivae}
I.~Khemakhem, D.~P. Kingma, R.~P. Monti, and A.~Hyv\"arinen.
\newblock Variational autoencoders and nonlinear ICA: a unifying framework.
\newblock In S.~Chiappa and R.~Calandra, editors, \emph{Proceedings of the
  23rd International Conference on Artificial Intelligence and Statistics
  (AISTATS)}, volume~108 of \emph{Proceedings of Machine Learning Research},
  pages 2207--2217. PMLR, 2020.
\newblock \biburl{https://proceedings.mlr.press/v108/khemakhem20a.html}.
\newblock Preprint: \biburl{https://arxiv.org/abs/1907.04809}.

\bibitem[Kim et al.(2006)]{kim2006iva}
T.~Kim, T.~Eltoft, and T.-W. Lee.
\newblock Independent vector analysis: an extension of ICA to multivariate
  components.
\newblock In \emph{Independent Component Analysis and Blind Signal Separation
  (ICA 2006)}, volume 3889 of \emph{Lecture Notes in Computer Science}, pages
  165--172. Springer, Berlin, Heidelberg, 2006.
\newblock ISBN 978-3-540-32630-4.
\newblock \doi{10.1007/11679363_21}.

\bibitem[Klenke(2020)]{klenke2020}
A.~Klenke.
\newblock \emph{Probability Theory: A Comprehensive Course}.
\newblock Universitext. Springer, Cham, 3rd edition, 2020.
\newblock ISBN 978-3-030-56401-8.
\newblock \doi{10.1007/978-3-030-56402-5}.
\newblock L\'evy's continuity theorem is Section~15.3.

\bibitem[Lee et al.(1999)]{lee1999extended}
T.-W. Lee, M.~Girolami, and T.~J. Sejnowski.
\newblock Independent component analysis using an extended infomax algorithm
  for mixed subgaussian and supergaussian sources.
\newblock \emph{Neural Computation}, 11(2):417--441, 1999.
\newblock \doi{10.1162/089976699300016719}.

\bibitem[Linnik and Ostrovskii(1977)]{linnik1977}
Yu.~V. Linnik and I.~V. Ostrovskii.
\newblock \emph{Decomposition of Random Variables and Vectors}, volume~48 of
  \emph{Translations of Mathematical Monographs}.
\newblock American Mathematical Society, Providence, RI, 1977.
\newblock ix+380 pp. ISBN 978-0-8218-1598-4.
\newblock Translated from the Russian; translation edited by J.~Rosenblatt.
\newblock \doi{10.1090/mmono/048}.

\bibitem[Lukacs(1970)]{lukacs1970}
E.~Lukacs.
\newblock \emph{Characteristic Functions}.
\newblock Charles Griffin \& Company Limited, London, 2nd, revised and
  enlarged edition, 1970.
\newblock x+350 pp. ISBN 0-85264-170-2. LCCN 70-513840.

\bibitem[MacKay(1996)]{mackay1996ica}
D.~J.~C. MacKay.
\newblock Maximum likelihood and covariant algorithms for independent
  component analysis.
\newblock Unpublished report, Cavendish Laboratory, University of Cambridge,
  1996.
\newblock \biburl{https://www.inference.org.uk/mackay/ica.pdf}.
\newblock Version 3.8, 8 January 1999, with minor corrections of 8 October
  2002.

\bibitem[MacKay(2003)]{mackay2003itila}
D.~J.~C. MacKay.
\newblock \emph{Information Theory, Inference, and Learning Algorithms}.
\newblock Cambridge University Press, Cambridge, 2003.
\newblock xii+628 pp. ISBN 978-0-521-64298-9.
\newblock Chapter~34, ``Independent Component Analysis and Latent Variable
  Modelling'', pp.~437--444.
\newblock \biburl{https://www.inference.org.uk/mackay/itila/}.

\bibitem[Marcinkiewicz(1939)]{marcinkiewicz1939}
J.~Marcinkiewicz.
\newblock Sur une propri\'et\'e de la loi de Gauss.
\newblock \emph{Mathematische Zeitschrift}, 44(1):612--618, 1939.
\newblock \doi{10.1007/BF01210677}.

\bibitem[Moors(1986)]{moors1986meaning}
J.~J.~A. Moors.
\newblock The meaning of kurtosis: Darlington reexamined.
\newblock \emph{The American Statistician}, 40(4):283--284, 1986.
\newblock \doi{10.1080/00031305.1986.10475415}.

\bibitem[Pandeva and Forr\'e(2023a)]{pandeva2023multiview}
T.~Pandeva and P.~Forr\'e.
\newblock Multi-view independent component analysis with shared and individual
  sources.
\newblock In R.~J. Evans and I.~Shpitser, editors, \emph{Proceedings of the
  39th Conference on Uncertainty in Artificial Intelligence (UAI)}, volume 216
  of \emph{Proceedings of Machine Learning Research}, pages 1639--1650. PMLR,
  2023a.
\newblock \biburl{https://proceedings.mlr.press/v216/pandeva23a.html}.
\newblock Preprint: \biburl{https://arxiv.org/abs/2210.02083}.

\bibitem[Pandeva and Forr\'e(2023b)]{pandeva2023omics}
T.~Pandeva and P.~Forr\'e.
\newblock Multi-view independent component analysis for omics data
  integration.
\newblock \emph{ICLR 2023 Workshop on Machine Learning and Global Health},
  2023b.
\newblock \biburl{https://openreview.net/forum?id=r5KL-AfXt75}.

\bibitem[Pandeva et al.(2025)]{pandeva2025robust}
T.~Pandeva, M.~J. Jonker, L.~Hamoen, J.~Mooij, and P.~Forr\'e.
\newblock Robust multi-view co-expression network inference.
\newblock In B.~Huang and M.~Drton, editors, \emph{Proceedings of the 4th
  Conference on Causal Learning and Reasoning (CLeaR)}, volume 275 of
  \emph{Proceedings of Machine Learning Research}, pages 490--513. PMLR, 2025.
\newblock \biburl{https://proceedings.mlr.press/v275/pandeva25a.html}.
\newblock Preprint: \biburl{https://arxiv.org/abs/2409.19991}.

\bibitem[P\'olya(1949)]{polya1949}
G.~P\'olya.
\newblock Remarks on characteristic functions.
\newblock In J.~Neyman, editor, \emph{Proceedings of the Berkeley Symposium on
  Mathematical Statistics and Probability}, pages 115--123. University of
  California Press, Berkeley and Los Angeles, 1949.
\newblock \biburl{https://digitalassets.lib.berkeley.edu/math/ucb/text/math_s1_article-08.pdf}.

\bibitem[Robbins and Monro(1951)]{robbins1951stochastic}
H.~Robbins and S.~Monro.
\newblock A stochastic approximation method.
\newblock \emph{The Annals of Mathematical Statistics}, 22(3):400--407, 1951.
\newblock \doi{10.1214/aoms/1177729586}.

\bibitem[Shimizu et al.(2006)]{shimizu2006}
S.~Shimizu, P.~O. Hoyer, A.~Hyv\"arinen, and A.~Kerminen.
\newblock A linear non-Gaussian acyclic model for causal discovery.
\newblock \emph{Journal of Machine Learning Research}, 7:2003--2030, 2006.
\newblock \biburl{https://jmlr.org/papers/v7/shimizu06a.html}.

\bibitem[Shimizu et al.(2011)]{shimizu2011directlingam}
S.~Shimizu, T.~Inazumi, Y.~Sogawa, A.~Hyv\"arinen, Y.~Kawahara, T.~Washio,
  P.~O. Hoyer, and K.~Bollen.
\newblock DirectLiNGAM: a direct method for learning a linear non-Gaussian
  structural equation model.
\newblock \emph{Journal of Machine Learning Research}, 12:1225--1248, 2011.
\newblock \biburl{https://jmlr.org/papers/v12/shimizu11a.html}.

\bibitem[Skitovich(1954)]{skitovich1954}
V.~P. Skitovich.
\newblock Linear forms of independent random variables and the normal
  distribution law.
\newblock \emph{Izvestiya Akademii Nauk SSSR, Seriya Matematicheskaya},
  18(2):185--200, 1954.
\newblock In Russian.
\newblock \biburl{https://www.mathnet.ru/eng/im3497}.
\newblock Announced in \emph{Doklady Akademii Nauk SSSR} (N.S.),
  89:217--219, 1953.

\bibitem[Westfall(2014)]{westfall2014}
P.~H. Westfall.
\newblock Kurtosis as peakedness, 1905--2014. R.I.P.
\newblock \emph{The American Statistician}, 68(3):191--195, 2014.
\newblock \doi{10.1080/00031305.2014.917055}.

\end{thebibliography}
\end{document}